\documentclass{article}%
\usepackage{amsmath}
\usepackage{amsfonts}
\usepackage{amssymb}
\usepackage{graphicx}%
\numberwithin{equation}{section}
\newtheorem{theorem}{Theorem}[section]

\newtheorem{corollary}[theorem]{Corollary}

\newtheorem{definition}[theorem]{Definition}

\newtheorem{lemma}[theorem]{Lemma}

\newtheorem{proposition}[theorem]{Proposition}
\newtheorem{remark}[theorem]{Remark}

\newenvironment{proof}[1][Proof]{\noindent\textbf{#1.} }{\ \rule{0.5em}{0.5em}}
\begin{document}

\title{$L^{p}$ and Schauder estimates for nonvariational operators structured on
H\"{o}rmander vector fields with drift\thanks{\textbf{2000
AMS\ Classification}: Primary 35H20. Secondary: 35B45, 42B20, 53C17.
\textbf{Keywords}: H\"{o}rmander's vector fields, Schauder estimates, $L^{p}$
estimates, drift}}
\author{Marco Bramanti and Maochun Zhu}
\maketitle

\begin{abstract}
Let%
\[
\mathcal{L}=\sum_{i,j=1}^{q}a_{ij}(x)X_{i}X_{j}+a_{0}\left(  x\right)  X_{0},
\]
where $X_{0},X_{1},...,X_{q}$ are real smooth vector fields satisfying
H\"{o}rmander's condition in some bounded domain $\Omega\subset\mathbb{R}^{n}$
($n>q+1$), the coefficients $a_{ij}=a_{ji},a_{0}$ are real valued, bounded
measurable functions defined in $\Omega$, satisfying the uniform positivity
conditions:
\[
\mu|\xi|^{2}\leq\sum_{i,j=1}^{q}a_{ij}(x)\xi_{i}\xi_{j}\leq\mu^{-1}|\xi
|^{2};\mu\leq a_{0}\left(  x\right)  \leq\mu^{-1}%
\]
for a.e.$\,\ x\in\Omega,$ every $\xi\in\mathbb{R}^{q}$, some constant $\mu>0$.

We prove that if the coefficients $a_{ij},a_{0}$ belong to the H\"{o}lder
space $C_{X}^{\alpha}\left(  \Omega\right)  $ with respect to the distance
induced by the vector fields, then local Schauder estimates of the following
kind hold:%
\[
\left\Vert X_{i}X_{j}u\right\Vert _{C_{X}^{\alpha}\left(  \Omega^{\prime
}\right)  }+\left\Vert X_{0}u\right\Vert _{C_{X}^{\alpha}\left(
\Omega^{\prime}\right)  }\leq c\left\{  \left\Vert Lu\right\Vert
_{C_{X}^{\alpha}\left(  \Omega\right)  }+\left\Vert u\right\Vert _{L^{\infty
}\left(  \Omega\right)  }\right\}
\]
for any $\Omega^{\prime}\Subset\Omega;$

if the coefficients $a_{ij},a_{0}$ belong to the space $VMO_{X,loc}\left(
\Omega\right)  $ with respect to the distance induced by the vector fields,
then local $L^{p}$ estimates of the following kind hold, for every
$p\in\left(  1,\infty\right)  $:%
\[
\left\Vert X_{i}X_{j}u\right\Vert _{L^{p}\left(  \Omega^{\prime}\right)
}+\left\Vert X_{0}u\right\Vert _{L^{p}\left(  \Omega^{\prime}\right)  }\leq
c\left\{  \left\Vert Lu\right\Vert _{L^{p}\left(  \Omega\right)  }+\left\Vert
u\right\Vert _{L^{p}\left(  \Omega\right)  }\right\}  .
\]

\pagebreak

\end{abstract}
\tableofcontents

\section{Introduction}

Let us consider a family of real smooth vector fields%
\[
X_{i}=\sum_{j=1}^{n}b_{ij}\left(  x\right)  \partial_{x_{j}},\text{
\ }i=0,1,2,...,q
\]
$\left(  q+1<n\right)  $ defined in some bounded domain $\Omega$ of
$\mathbb{R}^{n}$ and satisfying H\"{o}rmander's condition: the Lie algebra
generated by the $X_{i}$'s at any point of $\Omega$ span $\mathbb{R}^{n}.$
Under these assumptions, H\"{o}rmander's operators%
\[
\mathcal{L}=\sum_{i=1}^{q}X_{i}^{2}+X_{0}%
\]
have been studied since the late 1960's. H\"{o}rmander \cite{h} proved that
$\mathcal{L}$ is hypoelliptic, while Rothschild-Stein \cite{rs} proved that
for these operators \textit{a priori }estimates of $L^{p}$ type for second
order derivatives with respect to the vector fields hold, namely:%
\begin{equation}
\sum_{i,j=1}^{q}\left\Vert X_{i}X_{j}u\right\Vert _{L^{p}\left(
\Omega^{\prime}\right)  }+\left\Vert X_{0}u\right\Vert _{L^{p}\left(
\Omega^{\prime}\right)  }\leq c\left\{  \left\Vert \mathcal{L}u\right\Vert
_{L^{p}\left(  \Omega\right)  }+\left\Vert u\right\Vert _{L^{p}\left(
\Omega\right)  }+\sum_{i=1}^{q}\left\Vert X_{i}u\right\Vert _{L^{p}\left(
\Omega\right)  }\right\}  \label{a priori RS}%
\end{equation}
for any $p\in\left(  1,\infty\right)  ,\Omega^{\prime}\Subset\Omega.$

Note that the \textquotedblleft drift\textquotedblright vector field $X_{0}$
has weight two, compared with the vector fields $X_{i}$ for $i=1,2,...,q.$ For
operators without the drift term (\textquotedblleft sum of
squares\textquotedblright\ of H\"{o}rmander type) many more results have been
proved in the literature than for complete H\"{o}rmander's operators. On the
other hand, complete operators owe their interest, for instance, to the class
of Kolmogorov-Fokker-Planck operators, naturally arising in many fields of
physics, natural sciences and finance, as the transport-diffusion equations
satisfied by the transition probability density of stochastic systems of
O.D.E.s which describe some real system led to a basically deterministic law
perturbed by some kind of white noise. The study of Kolmogorov-Fokker-Planck
operators in the framework of H\"{o}rmander's operators received a strong
impulse by the paper \cite{lp} by Lanconelli-Polidoro, which started a lively
line of research. We refer to \cite{LPP} for a good survey on this field, with
further motivations for the study of these equations and related references.

Let us also note that the study of H\"{o}rmander's operators is considerably
easier when $\mathcal{L}$ is left invariant with respect to a suitable Lie
group of translations and homogeneous of degree two with respect to a suitable
family of dilations (which are group automorphisms of the corresponding group
of translations). In this case we say that $\mathcal{L}$ has an underlying
structure of homogeneous group and, by a famous result due to Folland
\cite{fo}, $\mathcal{L}$ possesses a homogeneous left invariant global
fundamental solution, which turns out to be a precious tool in proving
\textit{a priori} estimates.

In the last ten years, more general classes of nonvariational operators
structured on H\"{o}rmander's vector fields have been studied, namely%
\begin{align}
\mathcal{L}  &  =\sum_{i,j=1}^{q}a_{ij}\left(  x\right)  X_{i}X_{j}%
\label{nonvar 1}\\
\mathcal{L}  &  =\sum_{i,j=1}^{q}a_{ij}\left(  x\right)  X_{i}X_{j}%
-\partial_{t}\label{nonvar 2}\\
\mathcal{L}  &  =\sum_{i,j=1}^{q}a_{ij}\left(  x\right)  X_{i}X_{j}%
+a_{0}\left(  x\right)  X_{0} \label{nonvar 3}%
\end{align}
where the matrix $\left\{  a_{ij}\left(  x\right)  \right\}  _{i,j=1}^{q}$ is
symmetric positive definite, the coefficients are bounded ($a_{0}$ is bounded
away from zero) and satisfy suitable mild regularity assumptions, for instance
they belong to H\"{o}lder or $VMO$ spaces defined with respect to the distance
induced by the vector fields. Since the $a_{ij}$'s are not $C^{\infty},$ these
operators are no longer hypoelliptic. Nevertheless, \textit{a priori}
estimates on second order derivatives with respect to the vector fields are a
natural result which does not in principle require smoothness of the
coefficients. Namely, \textit{a priori} estimates in $L^{p}$ (with
coefficients $a_{ij}$ in $VMO\cap L^{\infty}$) have been proved in \cite{bb2}
for operators (\ref{nonvar 1}) and in \cite{bb1} for operators (\ref{nonvar 3}%
) but in homogeneous groups; \textit{a priori} estimates in $C^{\alpha}$
spaces (with coefficients $a_{ij}$ in $C^{\alpha}$) have been proved in
\cite{bb4} for operators (\ref{nonvar 2}) and in \cite{GL} for operators
(\ref{nonvar 3}) but in homogeneous groups.

In the particular case of Kolmogorov-Fokker-Planck operators, which can be
written as
\[
\mathcal{L}=\sum_{i,j=1}^{q}a_{ij}\left(  x\right)  \partial_{x_{i}x_{j}}%
^{2}+X_{0}%
\]
for a suitable drift $X_{0}$, $L^{p}$ estimates (when $a_{ij}$ are $VMO$) have
been proved in \cite{bcm} in homogeneous groups, while Schauder estimates
(when $a_{ij}$ are H\"{o}lder continuous) have been proved in \cite{DP}, under
more general assumptions (namely, assuming the existence of translations but
not necessarily dilations, adapted to the operator). We recall that the idea
of proving $L^{p}$ estimates for nonvariational operators with leading
coefficients in $VMO\cap L^{\infty}$ (instead of assuming their uniform
continuity) appeared for the first time in the papers \cite{cfl1}, \cite{cfl2}
by Chiarenza-Frasca-Longo, in the uniformly elliptic case.

The aim of the present paper is to prove both $L^{p}$ and $C^{\alpha}$ local
estimates for general operators (\ref{nonvar 3}) structured on H\"{o}rmander's
vector fields \textquotedblleft with drift\textquotedblright, without assuming
the existence of any group structure, under the appropriate assumptions on the
coefficients $a_{ij},a_{0}$. Namely, our basic estimates read as follows:%
\begin{equation}
\left\Vert u\right\Vert _{S_{X}^{2,p}\left(  \Omega^{\prime}\right)  }\leq
c\left\{  \left\Vert \mathcal{L}u\right\Vert _{L^{p}\left(  \Omega\right)
}+\left\Vert u\right\Vert _{L^{p}\left(  \Omega\right)  }\right\}
\label{a priori p}%
\end{equation}
for $p\in\left(  1,\infty\right)  $ and any $\Omega^{\prime}\Subset\Omega$ if
the coefficients are $VMO_{X,loc}\left(  \Omega\right)  ,$ and%
\begin{equation}
\left\Vert u\right\Vert _{C_{X}^{2,\alpha}\left(  \Omega^{\prime}\right)
}\leq c\left\{  \left\Vert \mathcal{L}u\right\Vert _{C_{X}^{\alpha}\left(
\Omega\right)  }+\left\Vert u\right\Vert _{L^{\infty}\left(  \Omega\right)
}\right\}  \label{a priori alpha}%
\end{equation}
for $\alpha\in\left(  0,1\right)  $ and $\Omega^{\prime}\Subset\Omega$ if the
coefficients are $C_{X}^{\alpha}\left(  \Omega\right)  .$ The related Sobolev
and H\"{o}lder spaces $S_{X}^{2,p},C_{X}^{2,\alpha},$ are those induced by the
vector fields $X_{i}$'s, and will be precisely defined in
\S \ref{subsec function spaces}. Clearly, these estimates are more general
than those contained in all the aforementioned papers.

At first sight, this kind of result could seem a straightforward
generalization of existing theories. However, several difficulties exist,
sometimes hidden in subtle details. First of all, we have to remark that in
the paper \cite{rs}, although the results are stated for both sum of squares
and complete H\"{o}rmander's operators, proofs are given only in the first
case. While some adaptations are quite straightforward, this is not always the
case. Therefore, some results proved in the present paper can be seen also as
a detailed proof of results stated in \cite{rs}, in the drift case. We will
justify this statement later, when dealing with these details. For the moment
we just point out that these difficulties are mainly related to the proof of
suitable representation formulas for second order derivatives $X_{i}X_{j}u$ of
a test function, in terms of $u$ and $\mathcal{L}u$, via singular integrals
and commutators of singular integrals. In turn, the reason why these
representation formulas are harder to be proved in presence of a drift relies
on the fact that a technical result which allows to exchange, in a suitable
sense, the action of $X_{i}$-derivatives with that of suitable integral
operators, assumes a more involved form when the drift is present.

Once the suitable representation formulas are established, a real variable
machinery similar to that used in \cite{bb2} and \cite{bb4} can be applied,
and this is the reason why we have chosen to give here a unified treatment to
$L^{p}$ and $C^{\alpha}$ estimates. More specifically, one considers a bounded
domain $\Omega$ endowed with the control distance induced by the vector fields
$X_{i}$'s, which has been defined, in the drift case, by Nagel-Stein-Wainger
in \cite{nsw2}, and the Lebesgue measure, which is locally doubling with
respect to these metric balls, as proved in \cite{nsw2}. However, a problem
arises when trying to apply to this context known results about singular
integrals~in metric doubling spaces (or \textquotedblleft spaces of
homogeneous type\textquotedblright, after \cite{cw}). Namely, what we should
know to apply this theory is a doubling property as%
\begin{equation}
\mu\left(  B\left(  x,2r\right)  \cap\Omega^{\prime}\right)  \leq c\mu\left(
B\left(  x,r\right)  \cap\Omega^{\prime}\right)  \text{ for any }x\in
\Omega^{\prime}\Subset\Omega,r>0\label{doubling 1}%
\end{equation}
while what we actually know in view of \cite{nsw2} is%
\begin{equation}
\mu\left(  B\left(  x,2r\right)  \right)  \leq c\mu\left(  B\left(
x,r\right)  \right)  \text{ for any }x\in\Omega^{\prime}\Subset\Omega
,0<r<r_{0}.\label{doubling 2}%
\end{equation}
Now, it has been known since \cite{FL} that, when $\Omega^{\prime}$ is for
instance a metric ball, condition (\ref{doubling 1}) follows from
(\ref{doubling 2}) as soon as the distance satisfies a kind of \emph{segment
property} which reads as follows: for any couple of points $x_{1},x_{2}$ at
distance $r$ and for any number $\delta<r$ and $\varepsilon>0$ there exists a
point $x_{0}$ having distance $\leq\delta$ from $x_{1}$ and $\leq
r-\delta+\varepsilon$ from $x_{2}$. However, while when the drift term is
lacking the distance induced by the $X_{i}$'s is easily seen to satisfy this
property, this is no longer the case when the field $X_{0}$ with weight two
enters the definition of distance and, as far as we know, a condition of kind
(\ref{doubling 1}) has never been proved in this context for $\Omega^{\prime}$
a metric ball, or for any other special kind of bounded domain $\Omega$. Thus
we are forced to apply a theory of singular integrals which does not require
the full strength of the global doubling condition (\ref{doubling 1}). A first
possibility is to consider the context of \emph{nondoubling spaces}, as
studied by Tolsa, Nazarov-Treil-Volberg, and other authors (see for instance
\cite{T2}, \cite{NTV3}, and references therein). Results of $L^{p}$ and
$C^{\alpha}$ continuity for singular integrals of this kind, applicable to our
context, have been proved in \cite{b2}. However, to prove our $L^{p}$
estimates (\ref{a priori p}), we also need some \emph{commutator estimates},
of the kind of the well-known result proved by \cite{crw}, that, as far as we
know, are not presently available in the framework of general nondoubling
quasimetric (or metric)\ spaces. For this reason, we have recently developed
in \cite{bz} a theory of \emph{locally homogeneous spaces }which is a quite
natural framework where all the results we need about singular integrals and
their commutators with BMO functions can be proved. To give a unified
treatment to both $L^{p}$ and $C^{\alpha}$ estimates, here we have decided to
prove both exploiting the results in \cite{bz}. We note that our Schauder
estimates could also be obtained applying the results in \cite{b2}, while
$L^{p}$ estimates cannot.

The necessity of avoiding the use of a doubling property of type
(\ref{doubling 1}), as well as some modifications required by the presence of
the drift $X_{0}$, also reflects in the way we have studied several properties
of the function spaces $C_{X}^{\alpha}$ and $BMO$, as we will see in
\S \ \ref{subsec function spaces}.

Once the basic estimates on second order derivatives are established, a
natural, but nontrivial, extension consists in proving similar estimates for
derivatives of (weighted) order $k+2,$ in terms of $k$ derivatives of
$\mathcal{L}u$ (assuming, of course, that the coefficients of the operator
possess the corresponding further regularity). In presence of a drift, it is
reasonable to restrict this study to the case of $k$ even, as already appears
from the analog result proved in homogeneous groups in \cite{bb1}. Even in
this case, a proof of this extensions seems to be a difficult task, and we
have preferred not to address this problem in the present paper, in order not
to further increase its length.

\bigskip

\textbf{Acknowledgements.} This research was mainly carried out while Maochun
Zhu was visiting the Department of Mathematics of Politecnico di Milano, which
we wish to thank for the hospitality. The project was supported by the
National Natural Science Foundation of China (Grant No. 10871157), Specialized
Research Fund for the Doctoral Program of Higher Education (No. 200806990032).

\section{Assumptions and main results\label{main result}}

We now state precisely our assumptions and main results. All the function
spaces involved in the statements below will be defined precisely in
\S \ \ref{Section Known results}. Our basic assumption is:

\bigskip

\textbf{Assumption (H). }Let%
\[
\mathcal{L}=\sum_{i,j=1}^{q}a_{ij}(x)X_{i}X_{j}+a_{0}\left(  x\right)  X_{0},
\]
where the $X_{0},X_{1},...,X_{q}$ are real smooth vector fields satisfying
H\"{o}rmander's condition in some bounded domain $\Omega\subset\mathbb{R}^{n}%
$, the coefficients $a_{ij}=a_{ji},a_{0}$ are real valued, bounded measurable
functions defined in $\Omega$, satisfying the uniform positivity conditions:
\begin{align*}
\mu|\xi|^{2}  &  \leq\sum_{i,j=1}^{q}a_{ij}(x)\xi_{i}\xi_{j}\leq\mu^{-1}%
|\xi|^{2};\\
\mu &  \leq a_{0}\left(  x\right)  \leq\mu^{-1}%
\end{align*}
for a.e.$\,\ x\in\Omega,$ every $\xi\in\mathbb{R}^{q}$, some constant $\mu>0$.

\bigskip

Our main results are the following:

\begin{theorem}
\label{schauder estimate}In addition to assumption (H), assume that the
coefficients $a_{ij},a_{0}$ belong to $C_{X}^{\alpha}\left(  \Omega\right)  $
for some $\alpha\in(0,1)$. Then for every domain $\Omega^{\prime}\Subset
\Omega$, there exists a constant $c>0$ depending on $\Omega^{\prime}%
,\Omega,X_{i},\alpha,\mu,\left\Vert a_{ij}\right\Vert _{C_{X}^{\alpha}%
(\Omega)}$ and $\left\Vert a_{0}\right\Vert _{C_{X}^{\alpha}(\Omega)}$ such
that, for every $u\in C_{X}^{2,\alpha}(\Omega)$, one has
\[
\left\Vert u\right\Vert _{C_{X}^{2,\alpha}(\Omega^{\prime})}\leq c\left\{
\left\Vert \mathcal{L}u\right\Vert _{C_{X}^{\alpha}(\Omega)}+\left\Vert
u\right\Vert _{L^{\infty}(\Omega)}\right\}  .
\]

\end{theorem}

\begin{theorem}
\label{lp estimate}In addition to assumption (H), assume that the coefficients
$a_{ij},a_{0}$ belong to the space $VMO_{X,loc}\left(  \Omega\right)  $. Then
for every $p\in\left(  1,\infty\right)  $, any $\Omega^{\prime}\Subset\Omega$,
there exists a constant $c$ depending on $X_{i},n,q,p,\mu$\thinspace
$,\Omega^{\prime},\Omega$ and the $VMO$ moduli of $a_{ij}$ and $a_{0}$, such
that for every $u$ $\in$ $S_{X}^{2,p}\left(  \Omega\right)  $,
\[
\left\Vert u\right\Vert _{S_{X}^{2,p}\left(  \Omega^{\prime}\right)  }\leq
c\left\{  \left\Vert \mathcal{L}u\right\Vert _{L^{p}(\Omega)}+\left\Vert
u\right\Vert _{L^{p}(\Omega)}\right\}  .
\]

\end{theorem}

\begin{remark}
\label{a0=1}Under the assumptions of the previous theorems, it is not
restrictive to assume $a_{0}\left(  x\right)  $ to be equal to $1$, for we can
always rewrite the equation%
\[
\sum_{i,j=1}^{q}a_{ij}X_{i}X_{j}+a_{0}X_{0}=f
\]
in the form%
\[
\sum_{i,j=1}^{q}\frac{a_{ij}}{a_{0}}X_{i}X_{j}+X_{0}=\frac{f}{a_{0}}%
\]
and apply the a-priori estimates to this equation, controlling $C^{\alpha}$ or
$VMO$ moduli of the new coefficients $\frac{a_{ij}}{a_{0}}$ in terms of the
analogous moduli of $a_{ij},a_{0}$ and the constant $\mu.$ Therefore
throughout the following we will always take $a_{0}\equiv1$.
\end{remark}

\section{Known results and preparatory material from real analysis and
geometry of vector fields\label{Section Known results}}

\subsection{Some known facts about H\"{o}rmander's vector fields, lifting and
approximation\label{subsection lifting}}

Let $X_{0},X_{1},\ldots,X_{q}$ be a system of real smooth vector fields,%
\[
X_{i}=\sum_{j=1}^{n}b_{ij}\left(  x\right)  \partial_{x_{j}},\text{
\ }i=0,1,2,...,q
\]
$\left(  q+1<n\right)  $ defined in some bounded, open and connected subset
$\Omega$ of $\mathbb{R}^{n}$. Let us assign to each $X_{i}$ a weight $p_{i}$,
saying that%

\[
p_{0}=2\text{ and }p_{i}=1\text{ for }i=1,2,...q.
\]
For any multiindex%

\[
I=(i_{1},i_{2},...,i_{k}),
\]
we define the weight of $I$ as%

\[
|I|=\underset{j=1}{\overset{k}{\sum}}p_{i_{j}}.
\]

For any couple of vector fields $X,Y$, let $\left[  X,Y\right]  =XY-YX$ be
their commutator. Now, for any multiindex $I=\left(  i_{1},i_{2}%
,...,i_{k}\right)  $ for $0\leq i_{k}\leq q$ we set:%

\[
X_{I}=X_{i_{1}}X_{i_{2}}...X_{i_{k}}%
\]
and%

\[
X_{[I]}=\left[  X_{i_{1}},\left[  X_{i_{2}},\ldots\left[  X_{i_{k-1}}%
,X_{i_{k}}\right]  \ldots\right]  \right]  .
\]
If $I=\left(  i_{1}\right)  $, then%

\[
X_{[I]}=X_{i_{1}}=X_{I}.
\]
We will say that $X_{[I]}$ is a \emph{commutator of weight }$|I|.$ As usual,
$X_{[I]}$ can be seen either as a differential operator or as a vector field.
We will write%

\[
X_{[I]}f
\]
to denote the differential operator $X_{[I]}$ acting on a function $f$, and%

\[
\left(  X_{[I]}\right)  _{x}%
\]
to denote the vector field $X_{[I]}$ evaluated at the point $x\in\Omega$.

We shall say that $X=\left\{  X_{0},X_{1},\ldots,X_{q}\right\}  $ satisfy
\emph{H\"{o}rmander's condition of weight} $s$ if these vector fields,
together with their commutators of weight $\leq s$, span the tangent space at
every point $x\in\Omega$.

Let $\mathcal{\ell}$ be the free Lie algebra of weight $s$ on $q+1$
generators, that is the quotient of the free Lie algebra with $q+1$ generators
by the ideal generated by the commutators of weight at least $s+1$. We say
that the vector fields $X_{0},\ldots,X_{q}$, which satisfy H\"{o}rmander's
condition of weight $s$ at some point $x_{0}\in\mathbb{R}^{n}$, are \emph{free
up to order }$s$\emph{ at} $x_{0}$ if $n=$dim$\,\mathcal{\ell}$, as a vector
space (note that inequality $\leq$ always holds). The famous Lifting Theorem
proved by Rothschild-Stein in \cite[p. 272]{rs} reads as follows:

\begin{theorem}
\label{Thm lifting}Let $X$ $=$ $(X_{0},X_{1},\ldots,X_{q})$ be $C^{\infty}$
real vector fields on a domain $\Omega\subset\mathbb{R}^{n}$ satisfying
H\"{o}rmander's condition of weight $s$ in $\Omega$. Then, for any
$\overline{x}\in\Omega$, in terms of new variables, $h_{n+1},\ldots,h_{N}$,
there exist smooth functions $\lambda_{il}(x,h)$ ($0\leq i\leq q,$ $n+1\leq
l\leq N$) defined in a neighborhood $\widetilde{U}$ of $\overline{\xi}=\left(
\overline{x},0\right)  \in\mathbb{R}^{N}$ such that the vector fields
$\widetilde{X}_{i}$ given by%
\[
\widetilde{X}_{i}=X_{i}+\underset{l=n+1}{\overset{N}{\sum}}\lambda
_{il}(x,h)\frac{\partial}{\partial h_{l}},\ \ i=0,\ldots,q
\]
satisfy H\"{o}rmander's condition of weight $s$ and are free up to weight $s$
at every point in $\widetilde{U}$.
\end{theorem}

Let $\widetilde{X}=\left(  \widetilde{X}_{0},\widetilde{X}_{1},\ldots
,\widetilde{X}_{q}\right)  $ be the lifted vector fields which are free up to
weight $s$ at some point $\xi\in\mathbb{R}^{N}$ and $\ell$ be the free Lie
algebra generated by $\widetilde{X}$. For each $j$, $1\leq j\leq s$, we can
select a family $\left\{  \widetilde{X}_{j,k}\right\}  _{k}$ of commutators of
weight $j$, with $\widetilde{X}_{1,k}=\widetilde{X}_{k},\widetilde{X}%
_{2,1}=\widetilde{X}_{0},k=1,2,\ldots,q$, such that $\left\{  \widetilde
{X}_{j,k}\right\}  _{jk}$ is a basis of $\ell$, that is to say, there exists a
set $A$ of double-indices $\alpha$\ such that $\left\{  \widetilde{X}_{\alpha
}\right\}  _{\alpha\in A}$ is a basis of $\ell$. Note that Card$A=N$, which
allows us to identify $\ell$ with $\mathbb{R}^{N}$.

Now, in $\mathbb{R}^{N}$ we can consider the group structure of $N(q+1,s)$,
which is the simply connected Lie group associated to $\mathcal{\ell}$. We
will write $\circ$ for the Lie group operation (which we think as a
\emph{translation}) and will assume that the group identity is the origin. It
is also possible to assume that $u^{-1}=-u$ (the group inverse is the
Euclidean opposite). We can naturally define \emph{dilations} in $N(q+1,s)$
by
\begin{equation}
D(\lambda)\left(  \left(  u_{\alpha}\right)  _{\alpha\in A}\right)  =\left(
\lambda^{\mid\alpha\mid}u_{\alpha}\right)  _{\alpha\in A}. \label{i1}%
\end{equation}
These are group automorphisms, hence $N(q+1,s)$ is a \emph{homogeneous group},
in the sense of Stein (see \cite[p. 618-622]{st}). We will call it
$\mathbb{G}$, leaving the numbers $q,s$ implicitly understood.

We can define in $\mathbb{G}$ a \emph{homogeneous norm} $\left\Vert
\cdot\right\Vert $ as follows. For any $u\in\mathbb{G}$, $u\neq0$, set
\[
\left\Vert u\right\Vert =r\text{ \ \ }\Leftrightarrow\text{ \ \ }\left\vert
D\left(  \frac{1}{r}\right)  u\right\vert =1\text{,}%
\]
where $\left\vert \cdot\right\vert $ denotes the Euclidean norm.

The function
\[
d_{\mathbb{G}}\left(  u,v\right)  =\left\Vert v^{-1}\circ u\right\Vert
\]
is a \emph{quasidistance}, that is:%

\begin{align}
d_{\mathbb{G}}\left(  u,v\right)   &  \geq0\text{ and }d_{\mathbb{G}}\left(
u,v\right)  =0\text{ if and only if }u=v;\nonumber\\
d_{\mathbb{G}}\left(  u,v\right)   &  =d_{\mathbb{G}}\left(  v,u\right)
;\label{quasid}\\
d_{\mathbb{G}}\left(  u,v\right)   &  \leq c\left(  d_{\mathbb{G}}\left(
u,z\right)  +d_{\mathbb{G}}\left(  z,v\right)  \right)  ,\nonumber
\end{align}
for every $u,v,z\in\mathbb{G}$ and some positive constant $c\left(
\mathbb{G}\right)  \geq1$. We define the balls with respect to $d_{\mathbb{G}%
}$ as
\[
B\left(  u,r\right)  \equiv\left\{  v\in\mathbb{R}^{N}:d_{\mathbb{G}}\left(
u,v\right)  <r\right\}  .
\]

It can be proved (see \cite[p.619]{st}) that the Lebesgue measure in
$\mathbb{R}^{N}$ is the Haar measure of $\mathbb{G}$. Therefore, by
(\ref{i1}),%
\[
\left\vert B\left(  u,r\right)  \right\vert =\left\vert B\left(  u,1\right)
\right\vert r^{Q},
\]
for every $u\in\mathbb{G}$ and $r>0$, where $Q=\sum_{\alpha\in A}\left\vert
\alpha\right\vert $. We will call $Q$ the \emph{homogeneous dimension} of
$\mathbb{G}$.

Next, we define the \emph{convolution} of two functions in $\mathbb{G}$ as%

\[
\left(  f\ast g\right)  \left(  u\right)  =\int_{\mathbb{R}^{N}}f\left(
u\circ v^{-1}\right)  g\left(  v\right)  dv=\int_{\mathbb{R}^{N}}g\left(
v^{-1}\circ u\right)  f\left(  v\right)  dv,
\]
for every couple of functions for which the above integrals make sense.

Let $\tau_{u}$ be the left translation operator acting on functions: $\left(
\tau_{u}f\right)  \left(  v\right)  =f\left(  u\circ v\right)  $. We say that
a differential operator $P$ on $\mathbb{G}$ is \emph{left invariant} if
$P\left(  \tau_{u}f\right)  =\tau_{u}\left(  Pf\right)  $ for every smooth
function $f$. From the above definition of convolution we read that if $P$ is
any left invariant differential operator,%

\begin{equation}
P\left(  f\ast g\right)  =f\ast Pg \label{left convolution}%
\end{equation}
(provided the integrals converge).

We say that \emph{a differential operator} $P$ on $\mathbb{G}$ is
\emph{homogeneous of degree} $\delta>0$ if
\[
P\left(  f\left(  D\left(  \lambda\right)  u\right)  \right)  =\lambda
^{\delta}\left(  Pf\right)  \left(  D\left(  \lambda\right)  u\right)
\]
for every test function $f$ and $\lambda>0,u\in\mathbb{G}$. Also, we say that
\emph{a function} $f$ is \emph{homogeneous of degree} $\delta\in\mathbb{R}$
if
\[
f\left(  D\left(  \lambda\right)  u\right)  =\lambda^{\delta}f\left(
u\right)  \text{ for every }\lambda>0,u\in\mathbb{G}.
\]

Clearly, if $P$ is a differential operator homogeneous of degree $\delta_{1}$
and $f$ is a homogeneous function of degree $\delta_{2}$, then $Pf$ is a
homogeneous function of degree $\delta_{2}-\delta_{1}$, while $fP$ is a
differential operator, homogeneous of degree $\delta_{1}-\delta_{2}$.

Let $Y_{\alpha}$ be the left invariant vector field which agrees with
$\frac{\partial}{\partial u_{\alpha}}$ at $0$ and set $Y_{1,k}=Y_{k}%
,k=1,\cdots,q,Y_{2,1}=Y_{0}$. $\ $The differential operator $Y_{i,k}$ is
homogeneous of degree $i$, and $\left\{  Y_{\alpha}\right\}  _{\alpha\in A}$
is a basis of the free Lie algebra $\ell$.

A differential operator on $\mathbb{G}$ is said to have \emph{local degree
less than or equal to }$\lambda$ if, after taking the Taylor expansion at $0$
of its coefficients, each term obtained is a differential operator homogeneous
of degree $\leq\lambda$.

Also, a function on $\mathbb{G}$ is said to have \emph{local degree greater
than or equal to }$\lambda$ if, after taking the Taylor expansion at $0$ of
its coefficients, each term obtained is a homogeneous function of degree
$\geq\lambda$.

For $\xi,\eta\in\widetilde{U}$, define the map
\[
\Theta_{\eta}(\xi)=(u_{\alpha})_{\alpha\in A}%
\]
with $\xi=\exp\left(  \sum_{\alpha\in A}u_{\alpha}\,\widetilde{X}_{\alpha
}\right)  \eta.$ We will also write $\Theta\left(  \eta,\xi\right)
=\Theta_{\eta}(\xi).$

We can \ now state Rothschild-Stein's approximation theorem (see \cite[p.
273]{rs}).

\begin{theorem}
\label{Rothschild-Stein's approximation Theorem}In the coordinates given by
$\Theta\left(  \eta,\cdot\right)  $ we can write $\widetilde{X}_{i}%
=Y_{i}+R_{i}^{\eta}$ on an open neighborhoods of $0$, where $R_{i}^{\eta}$ is
a vector field of local degree $\leq0$ for $i=1,\ldots,q\left(  \leq1\text{
for }i=0\right)  $ depending smoothly on $\eta$. Explicitly, this means that
for every $f\in C_{0}^{\infty}(\mathbb{G})$:%
\begin{equation}
\widetilde{X}_{i}\left[  f\left(  \Theta\left(  \eta,\cdot\right)  \right)
\right]  \left(  \xi\right)  =\left(  Y_{i}f+R_{i}^{\eta}f\right)  \left(
\Theta\left(  \eta,\xi\right)  \right)  . \label{approximation}%
\end{equation}
More generally, for every double-index $\left(  i,k\right)  \in A$, we can
write%
\begin{equation}
\widetilde{X}_{i,k}\left[  f\left(  \Theta\left(  \eta,\cdot\right)  \right)
\right]  \left(  \xi\right)  =\left(  Y_{i,k}f+R_{i,k}^{\eta}f\right)  \left(
\Theta\left(  \eta,\xi\right)  \right)  , \label{approximation format}%
\end{equation}
where $R_{i,k}^{\eta}$ is a vector field of local degree $\leq i-1$\ depending
smoothly on $\eta$.
\end{theorem}

This theorem says that the lifted vector fields $\widetilde{X}_{i}$ can be
locally approximated by the homogeneous, left invariant vector fields $Y_{i}$
on the group $\mathbb{G}$. Some other important properties of the map $\Theta$
are stated in the next theorem (see \cite[p. 284-287]{rs}):

\begin{theorem}
\label{metric}Let $\overline{\xi}\in\mathbb{R}^{N}$ and $\widetilde{U}$ be a
neighborhood of $\overline{\xi}$ such that for any $\eta\in\widetilde{U}$ the
map $\Theta\left(  \eta,\cdot\right)  $ is well defined in $\widetilde{U}%
$.\ For $\xi,\eta\in\widetilde{U}$, define%
\begin{equation}
\rho(\eta,\xi)=||\Theta\left(  \eta,\xi\right)  || \label{quasidistance}%
\end{equation}
where $||\cdot||$ is the homogeneous norm defined above. Then:

(a) \ \ \ $\Theta\left(  \eta,\xi\right)  =\Theta\left(  \xi,\eta\right)
^{-1}=-\Theta\left(  \xi,\eta\right)  $ for every $\xi,\eta\in\widetilde{U}$;

(b) $\ \ \rho$ is a quasidistance in $\widetilde{U}$ (that is satisfies the
three properties (\ref{quasid}));

(c) \ \ under the change of coordinates $u=$ $\Theta_{\xi}\left(  \eta\right)
$, the measure element becomes:%
\begin{equation}
d\eta=c(\xi)\cdot\left(  1+\omega\left(  \xi,u\right)  \right)  du,
\label{change of variables}%
\end{equation}
where $c(\xi)$ is a smooth function, bounded and bounded away from zero in
$\widetilde{U}$, $\omega\left(  \xi,u\right)  $ is a smooth function in both
variables, with
\[
\left\vert \omega\left(  \xi,u\right)  \right\vert \leq c\left\Vert
u\right\Vert ,
\]
and an analogous statement is true for the change of coordinates $u=$
$\Theta_{\eta}\left(  \xi\right)  $.
\end{theorem}

\begin{remark}
As we have recalled in the introduction, in the paper \cite{rs} detailed
proofs are given only when the drift term $X_{0}$ is lacking. A proof of the
lifting and approximation results explicitly covering the drift case can be
found in \cite{BBP2}, where the theory is also extended to the case of
nonsmooth H\"{o}rmander's vector fields. We refer to the introduction of
\cite{BBP2} for further bibliographic remarks about existing alternative
proofs of the lifting and approximation theorems.
\end{remark}

\subsection{Metric induced by vector fields\label{subsec metric}}

Let us start recalling the definition of control distance given by
Nagel-Stein-Wainger in \cite{nsw2} for H\"{o}rmander's vector fields with drift:

\begin{definition}
For any $\delta>0$, let $C\left(  \delta\right)  $ be the class of absolutely
continuous mappings $\varphi$: $[0,1]\rightarrow\Omega$ which satisfy%
\begin{equation}
\varphi^{\prime}(t)=\sum_{\left\vert I\right\vert \leq s}\lambda_{I}(t)\left(
X_{\left[  I\right]  }\right)  _{\varphi(t)}\text{ a.e. }t\in\left(
0,1\right)  \label{curve}%
\end{equation}
with $\left\vert \lambda_{I}(t)\right\vert \leq\delta^{\left\vert I\right\vert
}.$ We define%
\[
d(x,y)=\inf\left\{  \delta:\exists\varphi\in C\left(  \delta\right)  \text{
with }\varphi\left(  0\right)  =x,\varphi\left(  1\right)  =y\right\}  .
\]

\end{definition}

The finiteness of $d$ immediately follows by H\"{o}rmander's condition: since
the vector fields $\left\{  X_{\left[  I\right]  }\right\}  _{\left\vert
I\right\vert \leq s}$ span $\mathbb{R}^{n},$ we can always join any two points
$x,y$ with a curve $\varphi$ of the kind (\ref{curve}); moreover, $d$ turns
out to be a distance. Analogously to what Nagel-Stein-Wainger do in
\cite{nsw2} when $X_{0}$ is lacking, in \cite{BBP1} the following notion is introduced:

\begin{definition}
\label{c-c-distance}For any $\delta>0$, let $C_{1}\left(  \delta\right)  $ be
the class of absolutely continuous mappings $\varphi:[0,1]\rightarrow\Omega$
which satisfy%
\[
\varphi^{\prime}(t)=\overset{q}{\underset{i=0}{\sum}}\lambda_{i}(t)\left(
X_{i}\right)  _{\varphi\left(  t\right)  }\text{ a.e. }t\in\left(  0,1\right)
\]
with $\left\vert \lambda_{0}(t)\right\vert \leq\delta^{2}$ and $\left\vert
\lambda_{j}(t)\right\vert \leq\delta$ for $j=1,\cdots q.$

We define%
\[
d_{X}(x,y)=\inf\left\{  \delta:\exists\varphi\in C_{1}\left(  \delta\right)
\text{ with }\varphi\left(  0\right)  =x,\varphi\left(  1\right)  =y\right\}
.
\]

\end{definition}

Note that the finiteness of $d_{X}\left(  x,y\right)  $ for any two points
$x,y\in\Omega$ is not a trivial fact, but depends on a connectivity result
(\textquotedblleft Chow's theorem\textquotedblright); moreover, it can be
proved that $d$ and $d_{X}$ are equivalent, and that $d_{X}$ is still a
distance (see \cite{BBP1}, where these results are proved in the more general
setting of nonsmooth vector fields). From now on we will always refer to
$d_{X}$ as to the control distance, induced by the system of H\"{o}rmander's
vector fields $X$. It is well-known that this distance is topologically
equivalent to the Euclidean one. For any\ $x\in\Omega$, we set%

\[
B_{r}\left(  x\right)  =\left\{  y\in\Omega:d_{X}\left(  x,y\right)
<r\right\}  .
\]

The basic result about the measure of metric balls is the famous local
doubling condition proved by Nagel-Stein-Wainger \cite{nsw2}:

\begin{theorem}
\label{Thm NSW}For every $\Omega^{\prime}\Subset\Omega$ there exist positive
constants $c,r_{0}$ such that for any $x\in\Omega^{\prime},r\leq r_{0},$%
\[
\left\vert B\left(  x,2r\right)  \right\vert \leq c\left\vert B\left(
x,r\right)  \right\vert .
\]

\end{theorem}

As already pointed out in the introduction, the distance $d_{X}$ does
\emph{not} satisfy the segment property: given two points at distance $r$, it
is generally impossible to find a third point at distance $r/2$ from both. A
weaker property which this distance actually satisfies is contained in the
next lemma, and will be useful when dealing with the properties of H\"{o}lder
spaces $C^{\alpha}$:

\begin{lemma}
\label{non segment property}For any $x,y\in\Omega$ and any positive integer
$n,$ we can join $x$ to $y$ with a curve $\gamma$ and find $n+1$ points
$p_{0}=x,p_{1},p_{2},...,p_{n}=y$ on $\gamma,$ such that%
\[
d_{X}\left(  p_{j},p_{j+1}\right)  \leq\frac{d_{X}\left(  x,y\right)  }%
{\sqrt{n}}\text{ for }j=0,2,...,n-1.
\]

\end{lemma}

\begin{proof}
For any $x,y\in\Omega$ with $d_{X}\left(  x,y\right)  =R$, any $\varepsilon
>0$, by Definition \ref{c-c-distance} we can join $x$ and $y$ with a curve
$\gamma\left(  t\right)  $ satisfying
\[
\gamma\left(  0\right)  =y,\gamma\left(  1\right)  =x
\]
and
\[
\gamma^{\prime}\left(  t\right)  =\underset{i=0}{\overset{q}{\sum}}\lambda
_{i}\left(  t\right)  \left(  X_{i}\right)  _{\gamma\left(  t\right)  },
\]
with $\left\vert \lambda_{i}\left(  t\right)  \right\vert \leq R+\varepsilon$,
for $i=1,\ldots,q$ and $\left\vert \lambda_{0}\left(  t\right)  \right\vert
\leq\left(  R+\varepsilon\right)  ^{2}$.

Let $\gamma_{j}\left(  t\right)  =\gamma\left(  \frac{t+j}{n}\right)  $, for
$j=0,1,2,...,n-1$. Then $\gamma_{j}\left(  t\right)  $ satisfies%
\[
\gamma_{j}\left(  0\right)  =\gamma\left(  \frac{j}{n}\right)  \equiv
p_{j},\gamma_{j}\left(  1\right)  =\gamma\left(  \frac{j+1}{n}\right)
=p_{j+1};
\]
in particular, $p_{0}=x$ and $p_{n}=y;$ moreover,%
\[
\gamma_{j}^{\prime}\left(  t\right)  =\frac{1}{n}\underset{i=0}{\overset
{q}{\sum}}\lambda_{i}\left(  \frac{t+j}{n}\right)  \left(  X_{i}\right)
_{\gamma_{j}\left(  t\right)  },
\]
with
\[
\left\vert \frac{1}{n}\lambda_{0}\left(  \frac{t+j}{n}\right)  \right\vert
\leq\left(  \frac{R+\varepsilon}{\sqrt{n}}\right)  ^{2},\left\vert \frac{1}%
{n}\lambda_{i}\left(  \frac{t+j}{n}\right)  \right\vert <\frac{R+\varepsilon
}{\sqrt{n}}%
\]
for $i=1,\ldots,q,$ $j=0,2,...,n-1$. Thus%
\[
d_{X}\left(  p_{j},p_{j+1}\right)  \leq\frac{R+\varepsilon}{\sqrt{n}},
\]
for $j=0,2,...,n-1$ and any $\varepsilon>0,$ so we are done.
\end{proof}

The free lifted vector fields $\widetilde{X}_{i}$ induce, in the neighborhood
where they are defined, a control distance $d_{\widetilde{X}}$; we will denote
by $\widetilde{B}\left(  \xi,r\right)  $ the corresponding metric balls. In
this lifted setting we can also consider the quasidistance $\rho$ defined in
(\ref{quasidistance}). The two functions turn out to be equivalent:

\begin{lemma}
\label{equivalence of distance}Let $\overline{\xi},\widetilde{U}$ be as in
Thm. \ref{metric}. There exists $\widetilde{B}\left(  \overline{\xi},R\right)
\subset\widetilde{U}$ such that the distance $d_{\widetilde{X}}$ is equivalent
to the quasidistance $\rho$ in (\ref{quasidistance}) in $\widetilde{B}\left(
\overline{\xi},R\right)  ,$ and both are greater than the Euclidean distance;
namely there exist positive constants $c_{1},c_{2},c_{3}$ such that%
\[
c_{1}\left\vert \xi-\eta\right\vert \leq c_{2}\rho(\eta,\xi)\leq
d_{\widetilde{X}}(\eta,\xi)\leq c_{3}\rho(\eta,\xi)\text{ for every }\xi
,\eta\in\widetilde{B}\left(  \overline{\xi},R\right)  .
\]

\end{lemma}

This fact is proved in \cite{nsw2}, see also \cite[Proposition 22]{BBP2}.

\subsection{Some known results about locally homogeneous
spaces\label{subsec locally hom space}}

We are now going to recall the notion of \emph{locally homogeneous space},
introduced in \cite{bz}. This is the abstract setting which will allow us to
apply suitable results about singular integrals. Roughly speaking, a locally
homogeneous space is a set $\Omega$ endowed with a function $d$ which is a
quasidistance on any compact subset, and a measure $\mu$ which is locally
doubling, in a sense which will be made precise here below. In our concrete
situation, our set is endowed with a function $d$ which is a \emph{distance
}in\emph{ }$\Omega$, and a locally doubling measure. We can therefore give the
following definition, which is simpler than that given in \cite{bz}:

\begin{definition}
\label{Def loc hom space}Let $\Omega$ be a set, endowed with a distance $d$.
Let us denote by $B\left(  x,r\right)  $ the metric ball of center $x$ and
radius $r$. We will endow $\Omega$ with the topology induced by the metric.

Let $\mu$ be a positive regular Borel measure in $\Omega.$

Assume there exists an increasing sequence $\left\{  \Omega_{n}\right\}
_{n=1}^{\infty}$ of bounded measurable subsets of $\Omega,$ such that:%
\begin{equation}%
{\displaystyle\bigcup_{n=1}^{\infty}}
\Omega_{n}=\Omega\label{Hp 0}%
\end{equation}
and such for, any $n=1,2,3,...$:

(i) the closure of $\Omega_{n}$ in $\Omega$ is compact;

(ii) there exists $\varepsilon_{n}>0$ such that%
\begin{equation}
\left\{  x\in\Omega:d\left(  x,y\right)  <2\varepsilon_{n}\text{ for some
}y\in\Omega_{n}\right\}  \subset\Omega_{n+1}; \label{Hp 1}%
\end{equation}

(iii) there exists $C_{n}>1$ such that for any $x\in\Omega_{n},0<r\leq
\varepsilon_{n}$ we have%
\begin{equation}
\text{ }0<\mu\left(  B\left(  x,2r\right)  \right)  \leq C_{n}\mu\left(
B\left(  x,r\right)  \right)  <\infty. \label{Hp 3}%
\end{equation}
(Note that for $x\in\Omega_{n}$ and $r\leq\varepsilon_{n}$ we also have
$B\left(  x,2r\right)  \subset\Omega_{n+1}$).

We will say that $\left(  \Omega,\left\{  \Omega_{n}\right\}  _{n=1}^{\infty
},d,\mu\right)  $ is a \emph{(metric) locally homogeneous space }if the above
assumptions hold.
\end{definition}

Any space satisfying the above definition \emph{a fortiori }satisfies the
definition of locally homogeneous space given in \cite{bz}. In the following,
we will recall the statements of several results proved in \cite{bz}.

Next, we introduce the notion of local singular kernel.

\bigskip

\textbf{Assumption (K). }For fixed $\Omega_{n},\Omega_{n+1},$ and a fixed ball
$B\left(  \overline{x},R_{0}\right)  ,$ with $\overline{x}\in\Omega_{n}$ and
$R_{0}<2\varepsilon_{n}$ (hence $B\left(  \overline{x},R_{0}\right)
\subset\Omega_{n+1}$), let $K\left(  x,y\right)  $ be a measurable function
defined for $x,y\in B\left(  \overline{x},R_{0}\right)  $, $x\neq y$. Let
$R>0$ be any number satisfying%
\begin{equation}
cR\leq R_{0} \label{R R_0}%
\end{equation}
for some $c>1$; let $a,b\in C_{0}^{\alpha}\left(  \Omega_{n+1}\right)  ,$
$B\left(  \overline{x},c_{1}R\right)  \prec a\prec B\left(  \overline{x}%
,c_{2}R\right)  ,$ $B\left(  \overline{x},c_{3}R\right)  \prec b\prec B\left(
\overline{x},c_{4}R\right)  $ for some fixed constants $c_{i}\in\left(
0,1\right)  ,$ $i=1,...,4$ (the symbol $B_{1}\prec f\prec B_{2}$ means that
$f=1$ in $B_{1}$, vanishes outside $B_{2}$, and takes values in $\left[
0,1\right]  $). The new kernel%
\begin{equation}
\widetilde{K}\left(  x,y\right)  =a\left(  x\right)  K\left(  x,y\right)
b\left(  y\right)  \label{k tilde}%
\end{equation}
can be considered defined in the whole $\Omega_{n+1}\times\Omega
_{n+1}\setminus\left\{  x=y\right\}  $.

We now list a series of possible assumptions on the kernel $K$ which will be
recalled in the following theorems.

(i) We say that $K$ satisfies the \emph{standard estimates }for some $\nu
\in\lbrack0,1)$ if the following hold:%
\begin{equation}
\left\vert K\left(  x,y\right)  \right\vert \leq\frac{Ad\left(  x,y\right)
^{\nu}}{\mu\left(  B\left(  x,d\left(  x,y\right)  \right)  \right)  }
\label{standard 1}%
\end{equation}
for $x,y\in B\left(  \overline{x},R_{0}\right)  ,$ $x\neq y,$ and%
\begin{equation}
\left\vert K(x_{0},y)-K(x,y)\right\vert +\left\vert K(y,x_{0}%
)-K(y,x)\right\vert \leq\frac{Bd\left(  x_{0},y\right)  ^{\nu}}{\mu
(B(x_{0},d(x_{0},y)))}\left(  \frac{d(x_{0},x)}{d(x_{0},y)}\right)  ^{\beta}
\label{standard 2}%
\end{equation}
for any $x_{0},x,y\in B\left(  \overline{x},R_{0}\right)  $ with
$d(x_{0},y)>2d(x_{0},x)$, some $\beta>0$.

(ii) We say that $K$ satisfies the \emph{cancellation property }if the
following holds:

there exists $C>0$ such that for a.e. $x\in B\left(  \overline{x}%
,R_{0}\right)  $ and every $\varepsilon_{1},\varepsilon_{2}$ such that
$0<\varepsilon_{1}<\varepsilon_{2}$ and $B_{\rho}\left(  x,\varepsilon
_{2}\right)  \subset\Omega_{n+1}$%
\begin{equation}
\left\vert \int_{\Omega_{n+1},\varepsilon_{1}<\rho(x,y)<\varepsilon_{2}%
}K(x,y)\,d\mu(y)\right\vert +\left\vert \int_{\Omega_{n+1},\varepsilon
_{1}<\rho(x,z)<\varepsilon_{2}}K(z,x)\,d\mu(z)\right\vert \leq C,
\label{standard 3}%
\end{equation}
where $\rho$ is any \emph{quasidistance, }equivalent to $d$ in $\Omega_{n+1}$
and $B_{\rho}$ denotes $\rho$-balls. This means that $\rho$ satisfies the
axioms of distance, except for the triangle inequality, which is replaced by
the weaker%
\[
\rho\left(  x,y\right)  \leq c\left[  \rho\left(  x,z\right)  +\rho\left(
z,y\right)  \right]
\]
for any $x,y,z\in\Omega_{n+1}$ and some constant $c\geq1;$ moreover,%
\[
c_{1}d\left(  x,y\right)  \leq\rho\left(  x,y\right)  \leq c_{2}d\left(
x,y\right)
\]
for any $x,y$ and some positive constants $c_{1},c_{2}.$

(iii) We say that $K$ satisfies the \emph{convergence condition} if the
following holds: for a.e. $x\in B\left(  \overline{x},R_{0}\right)  $ such
that $B_{\rho}\left(  x,R\right)  \subset\Omega_{n+1}$ there exists%
\begin{equation}
h_{R}\left(  x\right)  \equiv\lim_{\varepsilon\rightarrow0}\int_{\Omega
_{n+1},\varepsilon<\rho(x,y)<R}K(x,y)d\mu(y), \label{convergence}%
\end{equation}
where $\rho$ is any quasidistance equivalent to $d$ in $\Omega_{n+1}$.

All the following results in this section have been proved in \cite{bz}. In
some statements we have introduced some slight simplifications (with respect
to \cite{bz}) due to the fact that our space is assumed to be metric.

\begin{theorem}
[$L^{p}$ and $C^{\eta}$ estimates for singular integrals]%
\label{Theorem L^p C^eta}Let $K,\widetilde{K}$ be as in Assumption (K), with
$K$ satisfying the standard estimates (i) with $\nu=0,$ the cancellation
property (ii) and the convergence condition (iii) stated above. If%
\[
Tf\left(  x\right)  =\lim_{\varepsilon\rightarrow0}\int_{B\left(  \overline
{x},R\right)  ,\rho(x,y)>\varepsilon}\widetilde{K}(x,y)f\left(  y\right)
d\mu(y),
\]
then for any $p\in\left(  1,\infty\right)  $%
\[
\left\Vert Tf\right\Vert _{L^{p}\left(  B\left(  \overline{x},R\right)
\right)  }\leq c\left\Vert f\right\Vert _{L^{p}\left(  B\left(  \overline
{x},R\right)  \right)  }.
\]
The constant $c$ depends on $p,n$ and the constants of $K$ involved in the
assumptions (but not on $R$).

Moreover, $T$ satisfies a weak 1-1 estimate:%
\[
\mu\left(  \left\{  x\in B\left(  \overline{x},R\right)  :\left\vert Tf\left(
x\right)  \right\vert >t\right\}  \right)  \leq\frac{c}{t}\left\Vert
f\right\Vert _{L^{1}\left(  B\left(  \overline{x},R\right)  \right)  }\text{
for any }t>0.
\]

Assume that, in addition, the kernel $K$ satisfies the condition
\begin{equation}
\widetilde{h}\left(  x\right)  \equiv\lim_{\varepsilon\rightarrow0}\int
_{\rho(x,y)>\varepsilon}\widetilde{K}(x,y)d\mu(y)\in C^{\gamma}\left(
\Omega_{n+1}\right)  \label{h tilde C^gamma}%
\end{equation}
for some $\gamma>0$ (where $\rho$ is the same appearing in the assumed
convergence condition (iii)). Then%
\begin{equation}
\left\Vert Tf\right\Vert _{C^{\eta}\left(  B\left(  \overline{x},R\right)
\right)  }\leq c\left\Vert f\right\Vert _{C^{\eta}\left(  B\left(
\overline{x},HR\right)  \right)  } \label{C_eta}%
\end{equation}
for any positive $\eta<\min\left(  \alpha,\beta,\gamma\right)  $ and some
constant $H>1$ independent of $R$. (Recall that $\alpha$ is the H\"{o}lder
exponent related to the cutoff functions defining $\widetilde{K},\beta$
appears in the standard estimates (i) and $\gamma$ is the number in
(\ref{h tilde C^gamma})).

The constant $c$ depends on $\eta,n,R,$ the constants involved in the
assumptions on $K,$ and the $C^{\gamma}$ norm of $\widetilde{h}.$
\end{theorem}

\begin{remark}
[Estimates for $C_{0}^{\eta}$ functions]\label{remark local holder}Applying
the H\"{o}lder continuity result to functions $f\in C_{0}^{\eta}\left(
B\left(  \overline{x},r\right)  \right)  $ with $r<R$ we can get a a bound%
\[
\left\Vert Tf\right\Vert _{C^{\eta}\left(  B\left(  \overline{x},r\right)
\right)  }\leq c\left\Vert f\right\Vert _{C^{\eta}\left(  B\left(
\overline{x},r\right)  \right)  }%
\]
with $c$ depending on $R$ but not on $r$.
\end{remark}

\begin{theorem}
[$L^{p}-L^{q}$ estimate for fractional integrals]\label{frac lp-lq}Let
$K,\widetilde{K}$ be as in Assumption (K), with $K$ satisfying the growth
condition
\begin{equation}
0\leq K\left(  x,y\right)  \leq\frac{c}{\mu\left(  B\left(  x,d\left(
x,y\right)  \right)  \right)  ^{1-\nu}} \label{Standard 1'}%
\end{equation}
for some $\nu\in\left(  0,1\right)  ,$ $c>0,$ any $x,y\in B\left(
\overline{x},R_{0}\right)  ,$ $x\neq y.$ If%
\[
I_{\nu}f\left(  x\right)  =\int_{B\left(  \overline{x},R\right)  }%
\widetilde{K}(x,y)f\left(  y\right)  d\mu(y)
\]
then, for any $p\in\left(  1,\frac{1}{\nu}\right)  ,\frac{1}{q}=\frac{1}%
{p}-\nu$ there exists $c\ $such that%
\[
\left\Vert I_{\nu}f\right\Vert _{L^{q}\left(  B\left(  \overline{x},R\right)
\right)  }\leq c\left\Vert f\right\Vert _{L^{p}\left(  B\left(  \overline
{x},R\right)  \right)  }%
\]
for any $f\in L^{p}\left(  B\left(  \overline{x},R\right)  \right)  $. The
constant $c$ depends on $p,n,$ and the constants of $K$ involved in the
assumptions (but not on $R$).
\end{theorem}

\begin{theorem}
[$C^{\eta}$ estimate for fractional integrals]\label{Thm frac C^eta}Let
$K,\widetilde{K}$ be as in Assumption (K), with $K$ satisfying
(\ref{standard 1}) and (\ref{standard 2}) for some $\nu\in\left(  0,1\right)
,\beta>0$. If%
\[
I_{\nu}f\left(  x\right)  =\int_{B\left(  \overline{x},R\right)  }%
\widetilde{K}(x,y)f\left(  y\right)  d\mu(y),
\]
then, for any $\eta<\min\left(  \alpha,\beta\,,\nu\right)  $%
\[
\left\Vert I_{\nu}f\right\Vert _{C^{\eta}\left(  B\left(  \overline
{x},R\right)  \right)  }\leq c\left\Vert f\right\Vert _{C^{\eta}\left(
B\left(  \overline{x},HR\right)  \right)  }.
\]
The constant $c$ depends on $\eta,n,R$ and the constants of $K$ involved in
the assumptions; the number $H$ only depends on $n$.
\end{theorem}

Reasoning as in Remark \ref{remark local holder}, we can also say that for
functions $f\in C_{0}^{\eta}\left(  B\left(  \overline{x},r\right)  \right)  $
with $r<R$ the following bound holds%
\[
\left\Vert I_{\nu}f\right\Vert _{C^{\eta}\left(  B\left(  \overline
{x},r\right)  \right)  }\leq c\left\Vert f\right\Vert _{C^{\eta}\left(
B\left(  \overline{x},r\right)  \right)  }%
\]
with $c$ depending on $R$ but not on $r$.

To state the commutator theorems that we will need, we have first to recall
the following

\begin{definition}
[Local $BMO$ and $VMO$ spaces]\label{definition local VMO}Let $\left(
\Omega,\left\{  \Omega_{n}\right\}  _{n=1}^{\infty},d,\mu\right)  $ be a
locally homogeneous space. For any function $u\in L^{1}\left(  \Omega
_{n+1}\right)  $, and $r>0$, with $r\leq\varepsilon_{n},$ set%
\[
\eta_{u,\Omega_{n},\Omega_{n+1}}^{\ast}(r)=\sup_{t\leq r}\sup_{x_{0}\in
\Omega_{n}}\frac{1}{\mu\left(  B\left(  x_{0},t\right)  \right)  }%
\int_{B\left(  x_{0},t\right)  }|u(x)-u_{B}|\,d\mu\left(  x\right)  ,
\]
where $u_{B}=\mu(B\left(  x_{0},t\right)  )^{-1}\int_{B\left(  x_{0},t\right)
}u$. We say that $u\in BMO_{loc}\left(  \Omega_{n},\Omega_{n+1}\right)  $ if%
\[
\left\Vert u\right\Vert _{BMO_{loc}\left(  \Omega_{n},\Omega_{n+1}\right)
}=\sup_{r\leq\varepsilon_{n}}\eta_{u,\Omega_{n},\Omega_{n+1}}^{\ast}\left(
r\right)  <\infty.
\]
We say that $u\in VMO_{loc}\left(  \Omega_{n},\Omega_{n+1}\right)  $ if $u\in
BMO_{loc}\left(  \Omega_{n},\Omega_{n+1}\right)  $ and
\[
\eta_{u,\Omega_{n},\Omega_{n+1}}^{\ast}(r)\rightarrow0\text{ as }%
r\rightarrow0.
\]
The function $\eta_{u,\Omega_{n},\Omega_{n+1}}^{\ast}$ will be called $VMO$
local modulus of $u$ in $\left(  \Omega_{n},\Omega_{n+1}\right)  $.
\end{definition}

Note that in the previous definition we integrate $u$ over balls centered at
points of $\Omega_{n}$ and enclosed in $\Omega_{n+1}$. This is a fairly
natural definition if we want to avoid integrating over the \emph{intersection
}$B\left(  x_{0},t\right)  \cap\Omega_{n}.$

\begin{theorem}
[Commutators of local singular integrals]\label{Thm commutator}Let
$K,\widetilde{K}$ be as in Assumption (K), with $K$ satisfying the standard
estimates (i) with $\nu=0,$ the cancellation property (ii) and the convergence
condition (iii). If%
\[
Tf\left(  x\right)  =\lim_{\varepsilon\rightarrow0}\int_{B\left(  \overline
{x},R\right)  ,\rho(x,y)>\varepsilon}\widetilde{K}(x,y)f\left(  y\right)
d\mu(y)
\]
and, for $a\in BMO_{loc}\left(  \Omega_{n+2},\Omega_{n+3}\right)  ,$ we set%
\[
C_{a}f\left(  x\right)  =T\left(  af\right)  \left(  x\right)  -a\left(
x\right)  Tf\left(  x\right)  ,
\]
then for any $p\in\left(  1,\infty\right)  $ there exists $c>0$ such that%
\[
\left\Vert C_{a}f\right\Vert _{L^{p}\left(  B\left(  \overline{x},R\right)
\right)  }\leq c\left\Vert a\right\Vert _{BMO_{loc}\left(  \Omega_{n+2}%
,\Omega_{n+3}\right)  }\left\Vert f\right\Vert _{L^{p}\left(  B\left(
\overline{x},R\right)  \right)  }.
\]

Moreover, if $a\in VMO_{loc}\left(  \Omega_{n+2},\Omega_{n+3}\right)  $ for
any $\varepsilon>0$ there exists $r>0$ such that for any $f\in L^{p}\left(
B\left(  \overline{x},r\right)  \right)  $ we have%
\[
\left\Vert C_{a}f\right\Vert _{L^{p}\left(  B\left(  \overline{x},r\right)
\right)  }\leq\varepsilon\left\Vert f\right\Vert _{L^{p}\left(  B\left(
\overline{x},r\right)  \right)  }.
\]
The constant $c$ depends on $p,n$ and the constants of $K$ involved in the
assumptions (but not on $R$); the constant $r$ also depends on the
$VMO_{loc}\left(  \Omega_{n+2},\Omega_{n+3}\right)  $ modulus of $a$.
\end{theorem}

\begin{theorem}
[Positive commutators of local fractional integrals]\label{Thm comm frac}Let
$K,\widetilde{K}$ be as in Assumption (K), with $K$ satisfying the growth
condition (\ref{Standard 1'}) for some $\nu>0.$ If%
\[
I_{\nu}f\left(  x\right)  =\int_{B\left(  \overline{x},R\right)  }%
\widetilde{K}(x,y)f\left(  y\right)  d\mu(y)
\]
and, for $a\in BMO_{loc}\left(  \Omega_{n+2},\Omega_{n+3}\right)  $, we set%
\begin{equation}
C_{\nu,a}f\left(  x\right)  =\int_{B\left(  \overline{x},R\right)  }%
\widetilde{K}(x,y)\left\vert a\left(  x\right)  -a\left(  y\right)
\right\vert f\left(  y\right)  d\mu(y) \label{positive comm}%
\end{equation}
then, for any $p\in\left(  1,\frac{1}{\nu}\right)  ,\frac{1}{q}=\frac{1}%
{p}-\nu$ there exists $c\ $such that%
\[
\left\Vert C_{\nu,a}f\right\Vert _{L^{q}\left(  B\left(  \overline
{x},R\right)  \right)  }\leq c\left\Vert a\right\Vert _{BMO_{loc}\left(
\Omega_{n+2},\Omega_{n+3}\right)  }\left\Vert f\right\Vert _{L^{p}\left(
B\left(  \overline{x},R\right)  \right)  }%
\]
for any $f\in L^{p}\left(  B\left(  \overline{x},R\right)  \right)  .$

Moreover, if $a\in VMO_{loc}\left(  \Omega_{n+2},\Omega_{n+3}\right)  $ for
any $\varepsilon>0$ there exists $r>0$ such that for any $f\in L^{p}\left(
B\left(  \overline{x},r\right)  \right)  $ we have%
\[
\left\Vert C_{\nu,a}f\right\Vert _{L^{q}\left(  B\left(  \overline
{x},r\right)  \right)  }\leq\varepsilon\left\Vert f\right\Vert _{L^{p}\left(
B\left(  \overline{x},r\right)  \right)  }.
\]
The constant $c$ depends on $p,\nu,n$ and the constants involved in the
assumptions on $K$ (but not on $R$); the constant $r$ also depends on the
$VMO_{loc}\left(  \Omega_{n+2},\Omega_{n+3}\right)  $ modulus of $a$.
\end{theorem}

\begin{theorem}
[Positive commutators of nonsingular integrals]\label{Thm comm pos}Let
$K,\widetilde{K}$ be as in Assumption (K), with $K$ satisfying condition
(\ref{standard 2}) with $\nu=0.$ Assume that the operator%
\[
Tf\left(  x\right)  =\int_{B\left(  \overline{x},R\right)  }\widetilde
{K}(x,y)f\left(  y\right)  d\mu(y)
\]
is continuous on $L^{p}\left(  B\left(  \overline{x},R\right)  \right)  $ for
any $p\in\left(  1,\infty\right)  $. For $a\in BMO_{loc}\left(  \Omega
_{n+2},\Omega_{n+3}\right)  ,$ set%
\begin{equation}
C_{a}f\left(  x\right)  =\int_{B\left(  \overline{x},R\right)  }\widetilde
{K}(x,y)\left\vert a\left(  x\right)  -a\left(  y\right)  \right\vert f\left(
y\right)  d\mu(y), \label{positive comm 2}%
\end{equation}
then%
\[
\left\Vert C_{a}f\right\Vert _{L^{p}\left(  B\left(  \overline{x},R\right)
\right)  }\leq c\left\Vert a\right\Vert _{BMO_{loc}\left(  \Omega_{n+2}%
,\Omega_{n+3}\right)  }\left\Vert f\right\Vert _{L^{p}\left(  B\left(
\overline{x},R\right)  \right)  }%
\]
for any $f\in L^{p}\left(  B\left(  \overline{x},R\right)  \right)
,p\in\left(  1,\infty\right)  $.

Moreover, if $a\in VMO_{loc}\left(  \Omega_{n+2},\Omega_{n+3}\right)  $ for
any $\varepsilon>0$ there exists $r>0$ such that for any $f\in L^{p}\left(
B\left(  \overline{x},r\right)  \right)  $ we have%
\[
\left\Vert C_{a}f\right\Vert _{L^{p}\left(  B\left(  \overline{x},r\right)
\right)  }\leq\varepsilon\left\Vert f\right\Vert _{L^{p}\left(  B\left(
\overline{x},r\right)  \right)  }.
\]
The constant $c$ depends on $n,$ the constants involved in the assumptions on
$K,$ and the $L^{p}$-$L^{p}$ norm of the operator $T$ (but not explicitly on
$R$); the constant $r$ also depends on the $VMO_{loc}\left(  \Omega
_{n+2},\Omega_{n+3}\right)  $ modulus of $a$.
\end{theorem}

\begin{remark}
\label{remark linear constants}In the statements of Theorems
\ref{Theorem L^p C^eta}, \ref{frac lp-lq}, \ref{Thm frac C^eta},
\ref{Thm commutator}, \ref{Thm comm frac}, \ref{Thm comm pos} we wrote that
the constant depends on the kernel only through the constants involved in the
assumptions. In the following we will need some additional information about
this dependence. A standard sublinearity argument allows us to say that if,
for example, our assumptions on the kernel are (\ref{standard 1}),
(\ref{standard 2}), (\ref{standard 3}), then the constant in our upper bound
will have the form%
\[
c\cdot\left(  A+B+C\right)
\]
where $A,B,C$ are the constants appearing in (\ref{standard 1}),
(\ref{standard 2}), (\ref{standard 3}), and $c$ does not depend on the kernel.
\end{remark}

We will also need the notion of \emph{local maximal operator }in locally
homogeneous spaces.

\begin{definition}
Fix $\Omega_{n},\Omega_{n+1}$ and, for any $f\in L^{1}\left(  \Omega
_{n+1}\right)  $ define the \emph{local maximal function}%
\[
M_{\Omega_{n},\Omega_{n+1}}f\left(  x\right)  =\sup_{r\leq r_{n}}\frac{1}%
{\mu\left(  B\left(  x,r\right)  \right)  }\int_{B\left(  x,r\right)
}\left\vert f\left(  y\right)  \right\vert d\mu\left(  y\right)  \text{ for
}x\in\Omega_{n}%
\]
where $r_{n}=2\varepsilon_{n}/5$.
\end{definition}

\begin{theorem}
\label{Thm maximal}Let $f$ be a measurable function defined on $\Omega_{n+1}.$
The following hold:

(a) If $f\in L^{p}\left(  \Omega_{n+1}\right)  $ for some $p\in\left[
1,\infty\right]  $, then $M_{\Omega_{n},\Omega_{n+1}}f$ is finite almost
everywhere in $\Omega_{n}$;

(b) if $f\in L^{1}\left(  \Omega_{n+1}\right)  $, then for every $t>0$,
\[
\mu\left(  \left\{  x\in\Omega_{n}:\left(  M_{\Omega_{n},\Omega_{n+1}%
}f\right)  \left(  x\right)  >t\right\}  \right)  \leq\frac{c_{n}}{t}%
\int_{\Omega_{n+1}}\left\vert f\left(  y\right)  \right\vert d\mu\left(
y\right)  ;
\]

(c) if $f\in L^{p}\left(  \Omega_{n+1}\right)  $, $1<p\leq\infty$, then
$M_{\Omega_{n},\Omega_{n+1}}f\in L^{p}\left(  \Omega_{n}\right)  $ and
\[
\left\Vert M_{\Omega_{n},\Omega_{n+1}}f\right\Vert _{L^{p}\left(  \Omega
_{n}\right)  }\leq c_{n,p}\left\Vert f\right\Vert _{L^{p}\left(  \Omega
_{n+1}\right)  }.
\]

\end{theorem}

Finally, we need to discuss an integral characterization of H\"{o}lder
continuous, analogous to the one classically introduced by Campanato
\cite{Campanato}, in our abstract and local setting.

\begin{definition}
[Local Campanato spaces]\label{def campanato}Let $\left(  \Omega,\left\{
\Omega_{n}\right\}  _{n=1}^{\infty},d,\mu\right)  $ be a locally homogeneous
space. For any function $u\in L^{1}\left(  \Omega_{n+1}\right)  ,$ $\alpha
\in\left(  0,1\right)  ,$ let%
\[
M_{\alpha,\Omega_{n},\Omega_{n+1}}u=\sup_{x\in\Omega_{n},r\leq\varepsilon_{n}%
}\inf_{c\in\mathbb{R}}\frac{1}{r^{\alpha}\left\vert B\left(  x,r\right)
\right\vert }\int_{B\left(  x,r\right)  }\left\vert u\left(  y\right)
-c\right\vert d\mu\left(  y\right)  .
\]
Set
\[
\mathcal{L}^{\alpha}\left(  \Omega_{n},\Omega_{n+1}\right)  =\left\{  u\in
L^{1}\left(  \Omega_{n+1}\right)  :M_{\alpha,\Omega_{n},\Omega_{n+1}}%
u<\infty\right\}  .
\]

\end{definition}

If $u\in C^{\alpha}\left(  \Omega_{n+1}\right)  $ then clearly%
\[
M_{\alpha,\Omega_{n},\Omega_{n+1}}u\leq\left\vert u\right\vert _{C^{\alpha
}\left(  \Omega_{n+1}\right)  }.
\]

A converse result is contained in the following:

\begin{theorem}
\label{Thm Campanato}For any $u\in\mathcal{L}^{\alpha}\left(  \Omega
_{n},\Omega_{n+1}\right)  ,$ there exists a function $u^{\ast},$ equal to $u$
a.e. in $\Omega_{n},$such that $u^{\ast}$ belongs to $C^{\alpha}\left(
\Omega_{n}\right)  $. Namely, for any $x,y\in\Omega_{n}$ with $2d\left(
x,y\right)  \leq\varepsilon_{n}$ we have%
\begin{equation}
\left\vert u^{\ast}\left(  x\right)  -u^{\ast}\left(  y\right)  \right\vert
\leq cM_{\alpha,\Omega_{n},\Omega_{n+1}}ud\left(  x,y\right)  ^{\alpha}.
\label{local holder}%
\end{equation}

If $2d\left(  x,y\right)  >\varepsilon_{n}$ then%
\begin{equation}
\left\vert u^{\ast}\left(  x\right)  -u^{\ast}\left(  y\right)  \right\vert
\leq c\left\{  M_{\alpha,\Omega_{n},\Omega_{n+1}}u+\left\Vert u\right\Vert
_{L^{1}\left(  \Omega_{n+1}\right)  }\right\}  d\left(  x,y\right)  ^{\alpha}.
\label{global holder}%
\end{equation}
The constant $c$ in (\ref{local holder}), (\ref{global holder}) depends on
$C_{n}$ but not on $\varepsilon_{n}$.
\end{theorem}

\bigskip

\textbf{Application of the abstract theory to our setting}

Let's now explain the way how this abstract setting will be used to describe
our concrete situation. The a-priori estimates we will prove in Theorems
\ref{schauder estimate}, \ref{lp estimate}, involve a fixed subdomain
$\Omega^{\prime}\Subset\Omega.$ Fix once and for all this $\Omega^{\prime}.$
For any $\overline{x}\in\Omega^{\prime}$ we can perform in a suitable
neighborhood of $\overline{x}$ the lifting and approximation procedure as
explained in \S \ \ref{subsection lifting}. Let $\overline{\xi}=\left(
\overline{x},0\right)  \in\mathbb{R}^{N}$ and $\widetilde{B}\left(
\overline{\xi},R\right)  $ be as in Lemma \ref{equivalence of distance}. We
can then choose%
\[
\widetilde{\Omega}=\widetilde{B}\left(  \overline{\xi},R\right)
;\widetilde{\Omega}_{k}=\widetilde{B}\left(  \overline{\xi},\frac{kR}%
{k+1}\right)  \text{ for }k=1,2,3,...
\]
By the properties of $d_{\widetilde{X}}$ that we have listed in
\S \ \ref{subsec metric}, and particularly Theorem \ref{Thm NSW},
\[
\left(  \widetilde{\Omega},\left\{  \widetilde{\Omega}_{k}\right\}
_{k=1}^{\infty},d_{\widetilde{X}},d\xi\right)
\]
is a metric\ locally homogeneous space. The function $\rho\left(  \xi
,\eta\right)  =\left\Vert \Theta\left(  \eta,\xi\right)  \right\Vert $ will
play the role of the quasidistance appearing in conditions (\ref{standard 3})
and (\ref{convergence}), in view of Lemma \ref{equivalence of distance}. This
will be the basic setting where we will apply singular integral estimates.

In the space of the original variables $\left(  \Omega,d_{X},dx\right)  ,$
instead, we will not apply singular integral estimates, but we will use again
the local doubling condition, when we will establish some important properties
of function spaces $C^{\alpha}$ and $VMO$ (see
\S \ \ref{subsec function spaces}). Note that, if $\Omega_{k}$ is an
increasing sequence of domains with $\Omega_{k}\Subset\Omega_{k+1}%
\Subset\Omega,$ we can say that%
\[
\left(  \Omega,\left\{  \Omega_{k}\right\}  _{k},d_{X},dx\right)
\]
is a metric locally homogeneous space.

\subsection{Function spaces\label{subsec function spaces}}

The aim of this section is twofold. First, we want to define the basic
function spaces we will need and point out their main properties; second, we
want to find a relation between function spaces defined over a ball $B\left(
\overline{x},r\right)  \subset\Omega\subset\mathbb{R}^{n}$ and on the
corresponding lifted ball $\widetilde{B}\left(  \overline{\xi},r\right)
\subset\mathbb{R}^{N}.$ More precisely, we need to know that $f\left(
x\right)  $ belongs to some function space on $B$ if and only $\widetilde
{f}\left(  x,h\right)  =f\left(  x\right)  $ belongs to the analogous function
space on $\widetilde{B}.$ This last fact relies on the following known result
(see \cite[Lemmas 3.1 and 3.2, p. 139]{nsw2}):

\begin{theorem}
\label{the volume of the ball}Let us denote by $B,\widetilde{B}$ the balls
defined with respect to $d_{X}$ in $\Omega$ and $d_{\widetilde{X}}$ in
$\widetilde{\Omega},$ respectively. There exist constants $\delta_{0}%
\in\left(  0,1\right)  ,r_{0},c_{1},c_{2}>0$ such that%
\begin{align}
c_{1}vol\left(  \widetilde{B}_{r}\left(  x,h\right)  \right)   &  \leq
vol\left(  B_{r}\left(  x\right)  \right)  \cdot vol\left\{  h^{\prime}%
\in\mathbb{R}^{N-n}:\left(  z,h^{\prime}\right)  \in\widetilde{B}_{r}\left(
x,h\right)  \right\} \label{metric ball equivalence}\\
&  \leq c_{2}vol\left(  \widetilde{B}_{r}\left(  x,h\right)  \right) \nonumber
\end{align}
for every $x\in\Omega,z\in$ $B_{\delta_{0}r}\left(  x\right)  $ and $r\leq
r_{0}$. (Here \textquotedblleft vol\textquotedblright\ stands for the Lebesgue
measure in the appropriate dimension, $x$ denotes a point in $\mathbb{R}^{n}$
and $h$ a point in $\mathbb{R}^{N-n}$). More precisely, the condition $z\in$
$B_{\delta_{0}r}\left(  x\right)  $ is needed only for the validity of the
first inequality in (\ref{metric ball equivalence}). Moreover:%
\begin{equation}
d_{\widetilde{X}}\left(  \left(  x,h\right)  ,\left(  x^{\prime},h^{\prime
}\right)  \right)  \geq d_{X}\left(  x,x^{\prime}\right)  .
\label{relation of distance}%
\end{equation}
Finally, the projection of the lifted ball $\widetilde{B}_{r}\left(
x,h\right)  $ on $\mathbb{R}^{n}$ is just the ball $B\left(  x,r\right)  ,$
and this projection is onto.
\end{theorem}

A consequence of the above theorem is the following

\begin{corollary}
\label{lifted integral and original integral}For any positive function $g$
defined in $B_{r}\left(  x\right)  \subset\Omega,r\leq r_{0}$, one has%
\begin{equation}
\frac{c_{1}}{\left\vert B_{\delta_{0}r}\left(  x\right)  \right\vert }%
\int_{B_{\delta_{0}r}\left(  x\right)  }g\left(  y\right)  dy\leq\frac
{1}{\left\vert \widetilde{B}_{r}\left(  x,h\right)  \right\vert }%
\int_{\widetilde{B}_{r}\left(  x,h\right)  }g\left(  y\right)  dydh^{\prime
}\leq\frac{c_{2}}{\left\vert B_{r}\left(  x\right)  \right\vert }\int
_{B_{r}\left(  x\right)  }g\left(  y\right)  dy.
\label{lifted BMO control BMO}%
\end{equation}
where $\delta_{0}$ is the constant in Theorem \ref{the volume of the ball}.
\end{corollary}

\begin{proof}
By (\ref{metric ball equivalence}) and the locally doubling condition, we
have, for some fixed $\delta_{0}<1$ as in Theorem \ref{the volume of the ball}%
,%
\begin{align*}
&  \frac{1}{\left\vert \widetilde{B}_{r}\left(  x,h\right)  \right\vert }%
\int_{\widetilde{B}_{r}\left(  x,h\right)  }g\left(  y\right)  dydh^{\prime}\\
&  =\frac{1}{\left\vert \widetilde{B}_{r}\left(  x,h\right)  \right\vert }%
\int_{B_{r}\left(  x\right)  }g\left(  y\right)  dy\int_{h^{\prime}%
\in\mathbb{R}^{N-n}:\left(  y,h^{\prime}\right)  \in\widetilde{B}_{r}\left(
x,h\right)  }dh^{\prime}\\
&  \geq\frac{c_{1}}{\left\vert \widetilde{B}_{r}\left(  x,h\right)
\right\vert }\int_{B_{\delta_{0}r}\left(  x\right)  }\frac{\left\vert
\widetilde{B}_{r}\left(  x,h\right)  \right\vert }{\left\vert B_{r}\left(
x\right)  \right\vert }g\left(  y\right)  dy\\
&  \geq\frac{c}{\left\vert B_{\delta_{0}r}\left(  x\right)  \right\vert }%
\int_{B_{\delta_{0}r}\left(  x\right)  }g\left(  y\right)  dy
\end{align*}
where in the last inequality we exploited the doubling condition $\left\vert
B_{r}\left(  x\right)  \right\vert \leq c\left\vert B_{\delta_{0}r}\left(
x\right)  \right\vert ,$ which holds because $B_{r}\left(  x\right)
\subset\Omega$ and $r\leq r_{0}.$ The proof of the second inequality in
(\ref{lifted BMO control BMO}) is analogous but easier, since it involves the
second inequality in (\ref{metric ball equivalence}), which does not require
the condition $y\in B_{\delta_{0}r}\left(  x\right)  .$
\end{proof}

\subsubsection{H\"{o}lder spaces\label{subsec holder}}

\begin{definition}
[H\"{o}lder spaces]For any $0<\alpha<1,u:\Omega\rightarrow\mathbb{R}$, let:%
\begin{align*}
\left\vert u\right\vert _{C_{X}^{\alpha}\left(  \Omega\right)  }  &
=\sup\left\{  \frac{\left\vert u\left(  x\right)  -u\left(  y\right)
\right\vert }{d_{X}\left(  x,y\right)  ^{\alpha}}:x,y\in\Omega,x\neq
y\right\}  ,\\
\left\Vert u\right\Vert _{C_{X}^{\alpha}\left(  \Omega\right)  }  &
=\left\vert u\right\vert _{C^{\alpha}\left(  \Omega\right)  }+\left\Vert
u\right\Vert _{L^{\infty}\left(  \Omega\right)  },\\
C_{X}^{\alpha}\left(  \Omega\right)   &  =\left\{  u:\Omega\rightarrow
\mathbb{R}:\left\Vert u\right\Vert _{C^{\alpha}\left(  \Omega\right)  }%
<\infty\right\}  .
\end{align*}
For any positive integer $k$ and $0<\alpha<0$, let%
\[
C_{X}^{k,\alpha}\left(  \Omega\right)  =\left\{  u:\Omega\rightarrow
\mathbb{R}:\left\Vert u\right\Vert _{C^{k,\alpha}\left(  \Omega\right)
}<\infty\right\}  ,\text{ \ }%
\]
with%
\[
\left\Vert u\right\Vert _{C_{X}^{k,\alpha}\left(  \Omega\right)  }%
=\underset{\left\vert I\right\vert =1}{\overset{k}{\sum}}\underset{j_{i}%
=0}{\overset{q}{\sum}}\left\Vert X_{j_{1}}\ldots X_{j_{l}}u\right\Vert
_{C^{\alpha}\left(  \Omega\right)  }+\left\Vert u\right\Vert _{C^{\alpha
}\left(  \Omega\right)  }%
\]
where $I=\left(  j_{1},j_{2},...,j_{l}\right)  $.

We will set $C_{X,0}^{\alpha}\left(  \Omega\right)  $ and $C_{X,0}^{k,\alpha
}\left(  \Omega\right)  $ for the subspaces of $C_{X}^{\alpha}\left(
\Omega\right)  $ and $C_{X}^{k,\alpha}\left(  \Omega\right)  $ of functions
which are compactly supported in $\Omega$, and set $C_{\widetilde{X}}^{\alpha
}\left(  \widetilde{\Omega}\right)  $, $C_{\widetilde{X}}^{k,\alpha}\left(
\widetilde{\Omega}\right)  $, $C_{\widetilde{X},0}^{\alpha}\left(
\widetilde{\Omega}\right)  $ and $C_{\widetilde{X},0}^{k,\alpha}\left(
\widetilde{\Omega}\right)  $ for the analogous function spaces over
$\widetilde{\Omega}$ defined by the $\widetilde{X}_{i}$'s.

We will also write $C_{X}^{k,0}\left(  \Omega\right)  $ to denote the space of
functions with continuous $X$-derivatives up to weight $k.$

Finally, whenever there is no risk of confusion, we will drop the index $X$,
writing $C^{\alpha}\left(  \Omega\right)  $ instead of $C_{X}^{\alpha}\left(
\Omega\right)  $, and so on.
\end{definition}

The next Proposition, adapted from \cite[Proposition 4.2]{bb4}, collects some
properties of $C^{\alpha}$ functions which will be useful later. We will apply
these properties mainly in the context of lifted variables, that is for the
vector fields $\widetilde{X}_{i}$ on a ball $\widetilde{B}\left(
\overline{\xi},R\right)  .$

\begin{proposition}
\label{basic properties for holder norm}Let $B\left(  \overline{x},2R\right)
$ be a fixed ball where the vector fields $X_{i}$ and the control distance $d$
are well defined.

(i) For any $\delta\in\left(  0,1\right)  ,$ for any $f\in C^{1}\left(
B\left(  \overline{x},\left(  1+\delta\right)  R\right)  \right)  $, one has
\begin{equation}
\left\vert f(x)-f(y)\right\vert \leq\frac{c}{\delta}d_{X}\left(  x,y\right)
\left(  \overset{q}{\underset{i=1}{\sum}}\underset{B\left(  \overline
{x},\left(  1+\delta\right)  R\right)  }{\sup}\left\vert X_{i}f\right\vert
+d_{X}\left(  x,y\right)  \underset{B\left(  \overline{x},\left(
1+\delta\right)  R\right)  }{\sup}\left\vert X_{0}f\right\vert \right)
\label{4.6}%
\end{equation}
for any $x,y\in B\left(  \overline{x},R\right)  $.

If $f\in C_{0}^{1}(B\left(  \overline{x},R\right)  )$, one can simply write,
for any $x,y\in B\left(  \overline{x},R\right)  ,$%
\begin{equation}
\left\vert f(x)-f(y)\right\vert \leq cd_{X}\left(  x,y\right)  \left(
\overset{q}{\underset{i=1}{\sum}}\underset{B\left(  \overline{x},R\right)
}{\sup}\left\vert X_{i}f\right\vert +d_{X}\left(  x,y\right)  \underset
{B\left(  \overline{x},R\right)  }{\sup}\left\vert X_{0}f\right\vert \right)
. \label{4.7}%
\end{equation}
In particular, for $f\in C_{0}^{1}(B\left(  \overline{x},R\right)  ,$
\begin{equation}
\left\vert f\right\vert _{C^{\alpha}\left(  B\left(  \overline{x},R\right)
\right)  }\leq cR^{1-\alpha}\cdot\left(  \overset{q}{\underset{i=1}{\sum}%
}\underset{B\left(  \overline{x},R\right)  }{\sup}\left\vert X_{i}f\right\vert
+R\underset{B\left(  \overline{x},R\right)  }{\sup}\left\vert X_{0}%
f\right\vert \right)  . \label{4.8}%
\end{equation}
Here $C^{1}$ (and $C_{0}^{1}$ ) stands for the classical space of (compactly
supported) continuously differentiable functions. The assumption $f\in C^{1}$
(or $C_{0}^{1}$ ) can be replaced by $f\in C_{X}^{2}$ (or $C_{X,0}^{2}$, respectively).

(ii) For any couple of functions $f,g\in C_{X}^{\alpha}\left(  B\left(
\overline{x},R\right)  \right)  $, one has
\[
\left\vert fg\right\vert _{C_{X}^{\alpha}\left(  B\left(  \overline
{x},R\right)  \right)  }\leq\left\vert f\right\vert _{C_{X}^{\alpha}\left(
B\left(  \overline{x},R\right)  \right)  }\left\Vert g\right\Vert _{L^{\infty
}\left(  B\left(  \overline{x},R\right)  \right)  }+\left\vert g\right\vert
_{C_{X}^{\alpha}\left(  B\left(  \overline{x},R\right)  \right)  }\left\Vert
f\right\Vert _{L^{\infty}\left(  B\left(  \overline{x},R\right)  \right)
}\text{ }%
\]
and%
\begin{equation}
\left\Vert fg\right\Vert _{C_{X}^{\alpha}\left(  B\left(  \overline
{x},R\right)  \right)  }\leq2\left\Vert f\right\Vert _{C_{X}^{\alpha}\left(
B\left(  \overline{x},R\right)  \right)  }\left\Vert g\right\Vert
_{C_{X}^{\alpha}\left(  B\left(  \overline{x},R\right)  \right)  }.
\label{4.3}%
\end{equation}
Moreover, if both $f$ and $g$ vanish at least at a point of $B\left(
\overline{x},R\right)  $, then
\begin{equation}
\left\vert fg\right\vert _{C_{X}^{\alpha}\left(  B\left(  \overline
{x},R\right)  \right)  }\leq cR^{\alpha}\left\vert f\right\vert _{C_{X}%
^{\alpha}\left(  B\left(  \overline{x},R\right)  \right)  }\left\vert
g\right\vert _{C_{X}^{\alpha}\left(  B\left(  \overline{x},R\right)  \right)
}. \label{4.4}%
\end{equation}

(iii) Let $B\left(  x_{i},r\right)  $ $(i=1,2,\cdots,k)$ be a finite family of
balls of the same radius $r$, such that $\cup_{i=1}^{k}B\left(  x_{i}%
,2r\right)  \subset$ $\Omega$. Then for any $f\in C_{X}^{\alpha}(\Omega)$,%
\begin{equation}
\left\Vert f\right\Vert _{C_{X}^{\alpha}(\cup_{i=1}^{k}B\left(  x_{i}%
,r\right)  )}\leq c\sum_{i=1}^{k}\left\Vert f\right\Vert _{C_{X}^{\alpha
}(B\left(  x_{i},2r\right)  )} \label{4.10}%
\end{equation}
with $c$ depending on the family of balls, but not on $f$.

(iv) There exists $r_{0}>0$ such that for any $f\in C_{X,0}^{2,\alpha}\left(
B\left(  \overline{x},R\right)  \right)  $ and $0<r\leq r_{0}$, we have the
following interpolation inequality:%
\begin{equation}
\left\Vert X_{0}f\right\Vert _{L^{\infty}\left(  B\left(  \overline
{x},R\right)  \right)  }\leq r^{\alpha/2}\left\vert X_{0}f\right\vert
_{C_{X}^{\alpha}\left(  B\left(  \overline{x},R\right)  \right)  }+\frac{2}%
{r}\left\Vert f\right\Vert _{L^{\infty}\left(  B\left(  \overline{x},R\right)
\right)  }. \label{4.11}%
\end{equation}

\end{proposition}

\begin{proof}
The proof for (ii)-(iii) is similar to that in \cite[Proposition 4.2]{bb4},
hence we will only prove (i) and (iv).

Throughout the proof we will write $d$ for $d_{X}.$

(i) Fix $\delta\in\left(  0,1\right)  $ and let $R^{\prime}=\left(
1+\delta\right)  R.$ Let us distinguish two cases:

(a) $d\left(  x,y\right)  <R^{\prime}-\max\left(  d\left(  \overline
{x},x\right)  ,d\left(  \overline{x},y\right)  \right)  .$ Let $\varepsilon>0$
such that also
\begin{equation}
d\left(  x,y\right)  +\varepsilon<R^{\prime}-\max\left(  d\left(  \overline
{x},x\right)  ,d\left(  \overline{x},y\right)  \right)  , \label{d+epsi}%
\end{equation}
hence by Definition \ref{c-c-distance} there exists a curve $\varphi(t)$, such
that $\varphi(0)=x,\varphi(1)=y$, and
\[
\varphi^{\prime}(t)=\overset{q}{\underset{i=0}{\sum}}\lambda_{i}(t)\left(
X_{i}\right)  _{\varphi(t)}%
\]
with $\left\vert \lambda_{i}(t)\right\vert \leq\left(  d\left(  x,y\right)
+\varepsilon\right)  ,\left\vert \lambda_{0}(t)\right\vert \leq\left(
d\left(  x,y\right)  +\varepsilon\right)  ^{2}$ for $i=1,\ldots q$. By
(\ref{d+epsi}),%
\[
B\left(  x,d\left(  x,y\right)  +\varepsilon\right)  \subset B\left(
\overline{x},R^{\prime}\right)
\]
hence every point $\gamma\left(  t\right)  $ for $t\in\left(  0,1\right)  $
belongs to $B\left(  \overline{x},R^{\prime}\right)  .$ Then we can write:
\begin{align*}
\left\vert f(x)-f(y)\right\vert  &  =\left\vert \int_{0}^{1}\frac{d}%
{dt}f(\varphi(t))\,dt\right\vert =\left\vert \int_{0}^{1}\overset{q}%
{\underset{i=0}{\sum}}\lambda_{i}(t)\left(  X_{i}f\right)  _{\varphi
(t)}\,dt\right\vert \\
&  \leq\left(  d\left(  x,y\right)  +\varepsilon\right)  \overset{q}%
{\underset{i=1}{\sum}}\underset{B\left(  \overline{x},R^{\prime}\right)
}{\sup}\left\vert X_{i}f\right\vert \,+\left(  d\left(  x,y\right)
+\varepsilon\right)  ^{2}\underset{B\left(  \overline{x},R^{\prime}\right)
}{\sup}\left\vert X_{0}f\right\vert ,
\end{align*}
and since $\varepsilon$ is arbitrary this implies%
\[
\left\vert f(x)-f(y)\right\vert \leq d\left(  x,y\right)  \left\{  \overset
{q}{\underset{i=1}{\sum}}\underset{B\left(  \overline{x},R^{\prime}\right)
}{\sup}\left\vert X_{i}f\right\vert \,+d\left(  x,y\right)  \underset{B\left(
\overline{x},R^{\prime}\right)  }{\sup}\left\vert X_{0}f\right\vert \right\}
.
\]
Note that the above argument relies on the differentiability of $f$ along the
curve $\varphi,$ which holds under either the assumption $f\in C^{1}\left(
B\left(  \overline{x},\left(  1+\delta\right)  R\right)  \right)  $ or $f\in
C_{X}^{2}\left(  B\left(  \overline{x},\left(  1+\delta\right)  R\right)
\right)  $ (since $X_{0}$ has weight two).

(b) Let now $d\left(  x,y\right)  \geq R^{\prime}-\max\left(  d\left(
\overline{x},x\right)  ,d\left(  \overline{x},y\right)  \right)  ,$ and let us
write%
\[
\left\vert f(x)-f(y)\right\vert \leq\left\vert f(x)-f\left(  \overline
{x}\right)  \right\vert +\left\vert f\left(  \overline{x}\right)
-f(y)\right\vert =A+B.
\]
Each of the terms $A,B$ can be bounded by an argument similar to that in case
(a) (since both $x$ and $y$ can be joined to $\overline{x}$ by curves
contained in $B\left(  \overline{x},R\right)  $) , getting%
\[
\left\vert f(x)-f(y)\right\vert \leq\left[  d\left(  x,\overline{x}\right)
+d\left(  y,\overline{x}\right)  \right]  \left\{  \overset{q}{\underset
{i=1}{\sum}}\underset{B\left(  \overline{x},R\right)  }{\sup}\left\vert
X_{i}f\right\vert \,+\left[  d\left(  x,\overline{x}\right)  +d\left(
y,\overline{x}\right)  \right]  \underset{B\left(  \overline{x},R\right)
}{\sup}\left\vert X_{0}f\right\vert \right\}  .
\]
Now it is enough to show that%
\[
d\left(  x,\overline{x}\right)  +d\left(  y,\overline{x}\right)  \leq\frac
{c}{\delta}d\left(  x,y\right)  .
\]
To show this, let $r\equiv\max\left(  d\left(  \overline{x},x\right)
,d\left(  \overline{x},y\right)  \right)  $. Then:%
\[
d\left(  x,\overline{x}\right)  +d\left(  y,\overline{x}\right)  \leq
2r\leq\frac{2}{\delta}\left(  R^{\prime}-r\right)  \leq\frac{2}{\delta
}d\left(  x,y\right)
\]
where the second inequality holds since $r<R$ and $R^{\prime}=\left(
1+\delta\right)  R,$ and the last inequality is assumption (b). This completes
the proof of (\ref{4.6}), which immediately implies (\ref{4.7}) and (\ref{4.8}).

Let us now prove\ (vi). Let $f\in C_{X,0}^{2,\alpha}\left(  B\left(
\overline{x},R\right)  \right)  $. For any $x\in B\left(  \overline
{x},R\right)  ,$ let $\gamma(t)$ be the curve such that%
\[
\gamma^{\prime}(t)=\left(  X_{0}\right)  _{\gamma(t)},\gamma(0)=x.
\]
This $\gamma\left(  t\right)  $ will be defined at least for $t\in\left[
0,r_{0}\right]  $ where $r_{0}>0$ is a number only depending on $B\left(
\overline{x},R\right)  $ and $X_{0}.$ Then, for any $r\in\left(
0,r_{0}\right)  $ we can write, for some $\theta\in\left(  0,1\right)  $:%
\begin{align*}
\left(  X_{0}f\right)  \left(  x\right)   &  =\left(  X_{0}f\right)  \left(
\gamma(0)\right)  =\frac{d}{dt}\left[  f\left(  \gamma(t)\right)  \right]
_{t=0}\\
&  =\frac{d}{dt}\left[  f\left(  \gamma(t)\right)  \right]  _{t=0}-\left[
f\left(  \gamma(r)\right)  -f\left(  \gamma(0)\right)  \right]  +\left[
f\left(  \gamma(r)\right)  -f\left(  \gamma(0)\right)  \right] \\
&  =\frac{d}{dt}\left[  f\left(  \gamma(t)\right)  \right]  _{t=0}-r\frac
{d}{dt}\left[  f\left(  \gamma(t)\right)  \right]  _{t=\theta r}+\left[
f\left(  \gamma(r)\right)  -f\left(  \gamma(0)\right)  \right] \\
&  =\frac{d}{dt}\left[  f\left(  \gamma(t)\right)  \right]  _{t=0}\left(
1-r\right)  +r\left(  \frac{d}{dt}\left[  f\left(  \gamma(t)\right)  \right]
_{t=0}-\frac{d}{dt}\left[  f\left(  \gamma(t)\right)  \right]  _{t=\theta
r}\right)  +\\
&  +\left[  f\left(  \gamma(r)\right)  -f\left(  \gamma(0)\right)  \right] \\
&  =\left(  1-r\right)  \left(  X_{0}f\right)  \left(  x\right)  +r\left[
\left(  X_{0}f\right)  \left(  \gamma(0)\right)  -\left(  X_{0}f\right)
\left(  \gamma(\theta r)\right)  \right] \\
&  +\left[  f\left(  \gamma(r)\right)  -f\left(  \gamma(0)\right)  \right]  ,
\end{align*}
hence
\begin{align*}
r\left\vert \left(  X_{0}f\right)  \left(  x\right)  \right\vert  &  \leq
r\left\vert \left(  X_{0}f\right)  \left(  \gamma(0)\right)  -\left(
X_{0}f\right)  \left(  \gamma(\theta r)\right)  \right\vert +2\left\Vert
f\right\Vert _{L^{\infty}}\\
&  =r\frac{\left(  \theta r\right)  ^{\alpha/2}\left\vert \left(
X_{0}f\right)  \left(  \gamma(0)\right)  -\left(  X_{0}f\right)  \left(
\gamma(\theta r)\right)  \right\vert }{\left(  \theta r\right)  ^{\alpha/2}%
}+2\left\Vert f\right\Vert _{L^{\infty}}.
\end{align*}
Since, by definition of $\gamma$ and $d,$ $d(\gamma(0),\gamma(\theta
r))\leq\left(  \theta r\right)  ^{1/2},$
\begin{align*}
\left\vert \left(  X_{0}f\right)  \left(  x\right)  \right\vert  &
\leq\left(  \theta r\right)  ^{\alpha/2}\left\vert X_{0}f\right\vert
_{C_{X}^{\alpha}\left(  B\left(  \overline{x},R\right)  \right)  }+\frac{2}%
{r}\left\Vert f\right\Vert _{L^{\infty}\left(  B\left(  \overline{x},R\right)
\right)  }\\
&  \leq r^{\alpha/2}\left\vert X_{0}f\right\vert _{C_{X}^{\alpha}\left(
B\left(  \overline{x},R\right)  \right)  }+\frac{2}{r}\left\Vert f\right\Vert
_{L^{\infty}\left(  B\left(  \overline{x},R\right)  \right)  },
\end{align*}
and we are done.
\end{proof}

Next, we are going to study the relation between the spaces $C_{X}^{\alpha
}\left(  B_{R}\right)  $ and $C_{\widetilde{X}}^{\alpha}\left(  \widetilde
{B}_{R}\right)  .$

\begin{proposition}
\label{original holder norm and lifted holder norm}Let $\widetilde{B}\left(
\overline{\xi},R\right)  $ be a lifted ball (as described at the end of
\S \ \ref{subsec locally hom space}), with $\overline{\xi}=\left(
\overline{x},0\right)  .$ If $f$ is a function defined in $B\left(
\overline{x},R\right)  $ and $\widetilde{f}\left(  x,h\right)  =f\left(
x\right)  $ is regarded as a function defined on $\widetilde{B}_{R}\left(
\overline{\xi},R\right)  $, then the following inequalities hold (whenever the
right-hand side is finite):%
\begin{align}
\left\vert \widetilde{f}\right\vert _{C_{\widetilde{X}}^{\alpha}\left(
\widetilde{B}\left(  \overline{\xi},R\right)  \right)  }  &  \leq\left\vert
f\right\vert _{C_{X}^{\alpha}\left(  B\left(  \overline{x},R\right)  \right)
},\nonumber\\
\left\vert f\right\vert _{C_{X}^{\alpha}\left(  B\left(  \overline
{x},s\right)  \right)  }  &  \leq\frac{c}{\left(  t-s\right)  ^{2}}\left\vert
\widetilde{f}\right\vert _{C_{\widetilde{X}}^{\alpha}\left(  \widetilde
{B}\left(  \overline{\xi},t\right)  \right)  }\text{ for }0<s<t<R
\label{second}%
\end{align}
where $c$ also depends on $R$. Moreover,%
\begin{align}
\left\vert \widetilde{X}_{i_{1}}\widetilde{X}_{i_{2}}\cdots\widetilde
{X}_{i_{k}}\widetilde{f}\right\vert _{C_{\widetilde{X}}^{\alpha}\left(
\widetilde{B}\left(  \overline{\xi},R\right)  \right)  }  &  \leq\left\vert
X_{i_{1}}X_{i_{2}}\cdots X_{i_{k}}f\right\vert _{C_{X}^{\alpha}\left(
B\left(  \overline{x},R\right)  \right)  }, \label{higher order holder norm 1}%
\\
\left\vert X_{i_{1}}X_{i_{2}}\cdots X_{i_{k}}f\right\vert _{C_{X}^{\alpha
}\left(  B\left(  \overline{x},s\right)  \right)  }  &  \leq\frac{c}{\left(
t-s\right)  ^{2}}\left\vert \widetilde{X}_{i_{1}}\widetilde{X}_{i_{2}}%
\cdots\widetilde{X}_{i_{k}}\widetilde{f}\right\vert _{C_{\widetilde{X}%
}^{\alpha}\left(  \widetilde{B}\left(  \overline{\xi},t\right)  \right)  }
\label{higher order holder norm 2}%
\end{align}
for $0<s<t<R$ and $i_{j}=0,1,2,\ldots,q$.
\end{proposition}

As already done in \cite[Proposition 8.3]{bb4}, to prove the above relation
between H\"{o}lder spaces over $B$ and $\widetilde{B}$ we have to exploit an
equivalent integral characterization of H\"{o}lder continuous functions,
analogous to the one established in the classical case by Campanato in
\cite{Campanato}. However, to avoid integration over sets of the kind
$\Omega\cap B\left(  x,r\right)  $ (with the related problem of assuring a
suitable doubling condition) we need to apply the local version of this result
which has been established in \cite{bz} and recalled in
\S \ \ref{subsec locally hom space}. We are going to apply Definition
\ref{def campanato} in our context.

\begin{definition}
For $\overline{x}\in\Omega^{\prime},B\left(  \overline{x},R\right)
\subset\Omega,f\in L^{1}\left(  B\left(  \overline{x},R\right)  \right)
,\alpha\in\left(  0,1\right)  ,$ $0<s<t\leq1,$ let%
\[
M_{\alpha,B_{sR},B_{tR}}\left(  f\right)  =\sup_{x\in B\left(  \overline
{x},sR\right)  ,r\leq\left(  t-s\right)  R}\inf_{c\in\mathbb{R}}\frac
{1}{r^{\alpha}\left\vert B_{r}\left(  x\right)  \right\vert }\int
_{B_{r}\left(  x\right)  }\left\vert f\left(  y\right)  -c\right\vert dy.
\]

\end{definition}

If $f\in C_{X}^{\alpha}\left(  B\left(  \overline{x},R\right)  \right)  $ then%
\[
M_{\alpha,B_{sR},B_{tR}}\left(  f\right)  \leq\left\vert f\right\vert
_{C^{\alpha}\left(  B_{R}\left(  x_{0}\right)  \right)  }.
\]
Moreover:

\begin{lemma}
\label{holder norm and Campanato}For $\overline{x}\in\Omega^{\prime},B\left(
\overline{x},2R_{0}\right)  \subset\Omega,$ $R<R_{0},$ $\alpha\in\left(
0,1\right)  ,$ $0<s<t\leq1,$ if $f\in L^{1}\left(  B\left(  \overline
{x},tR\right)  \right)  $ is a function such that $M_{\alpha,B_{sR},B_{tR}%
}\left(  f\right)  <\infty$, then there exists a function $f^{\ast}$, a.e.
equal to $f$, such that $f^{\ast}\in C_{X}^{\alpha}\left(  B\left(
\overline{x},sR\right)  \right)  $ and
\[
\left\vert f^{\ast}\right\vert _{C_{X}^{\alpha}\left(  B\left(  \overline
{x},sR\right)  \right)  }\leq\frac{c}{\left(  t-s\right)  ^{2}}M_{\alpha
,B_{sR},B_{tR}}\left(  f\right)
\]
for some $c$ independent of $f,s,t.$
\end{lemma}

\begin{proof}
We can apply Theorem \ref{Thm Campanato} choosing $\Omega_{k}=B\left(
\overline{x},sR\right)  ,\Omega_{k+1}=B\left(  \overline{x},tR\right)
,\varepsilon_{n}=R\left(  t-s\right)  $. The locally doubling constant can be
chosen independently of $R,$ since $B\left(  \overline{x},2R_{0}\right)
\subset\Omega,$ $R<R_{0}.$ We conclude there exists a function $f^{\ast}$,
a.e. equal to $f$, such that%
\[
\left\vert f^{\ast}\left(  x\right)  -f^{\ast}\left(  y\right)  \right\vert
\leq cM_{\alpha,B_{sR},B_{tR}}\left(  f\right)  d_{X}\left(  x,y\right)
^{\alpha}%
\]
for any $x,y\in B\left(  \overline{x},sR\right)  $ with $d_{X}\left(
x,y\right)  \leq R\left(  t-s\right)  /2$

If now $x,y$ are any two points in $B_{sR}\left(  x_{0}\right)  $, and
$r=d_{X}\left(  x,y\right)  ,$ by Lemma \ref{non segment property} we can find
$n+1$ points $x_{0}=x,x_{1},x_{2},...,x_{n}=y$ in $B_{sR}\left(  x_{0}\right)
$ such that
\[
d_{X}\left(  x_{i},x_{i-1}\right)  \leq\frac{r}{\sqrt{n}}.
\]
Let $n$ be the least integer such that $\frac{r}{\sqrt{n}}\leq R\left(
t-s\right)  /2,$ then%
\begin{align*}
\left\vert f^{\ast}\left(  x\right)  -f^{\ast}\left(  y\right)  \right\vert
&  \leq\sum_{i=1}^{n}\left\vert f^{\ast}\left(  x_{i}\right)  -f^{\ast}\left(
x_{i-1}\right)  \right\vert \leq\sum_{i=1}^{n}cM_{\alpha,B_{sR},B_{tR}}\left(
f\right)  d_{X}\left(  x_{i},x_{i-1}\right)  ^{\alpha}\\
&  \leq ncM_{\alpha,B_{sR},B_{tR}}\left(  f\right)  d_{X}\left(  x,y\right)
^{\alpha}.
\end{align*}
Let us find an upper bound on $n.$ We know that
\[
\sqrt{n}\leq c\frac{d_{X}\left(  x,y\right)  }{R\left(  t-s\right)  }\leq
\frac{c}{t-s}%
\]
since $d_{X}\left(  x,y\right)  \leq2R$ for $x,y\in B_{tR}\left(
x_{0}\right)  .$ Hence $n\leq c/\left(  t-s\right)  ^{2}$ and the lemma is proved.
\end{proof}

\bigskip

\begin{proof}
[Proof of Proposition \ref{original holder norm and lifted holder norm}]The
first inequality immediately follows by (\ref{relation of distance}), so let
us prove the second one.

Let $0<s<t<1,$ $x\in B\left(  \overline{x},\delta_{0}sR\right)  ,$ where
$\delta_{0}$ is the number in Theorem \ref{the volume of the ball}, $r\leq
R\left(  t-s\right)  $, $\overline{\xi}=\left(  \overline{x},0\right)  $.
Since the projection $\pi:\widetilde{B}\left(  \left(  x,s\right)
,\delta\right)  \rightarrow B\left(  x,\delta\right)  $ is onto (see Theorem
\ref{the volume of the ball}), there exists $h\in\mathbb{R}^{N-n}$ such that
$\xi=\left(  x,h\right)  \in$ $\widetilde{B}\left(  \overline{\xi},\delta
_{0}sR\right)  .$ Then we have the following inequalities:%
\begin{align}
&  \frac{1}{r^{\alpha}}\frac{c}{\left\vert B_{\delta_{0}r}\left(  x\right)
\right\vert }\int_{B_{\delta_{0}r}\left(  x\right)  }\left\vert f\left(
y\right)  -k\right\vert dy\nonumber\\
&  \text{(by Corollary \ref{lifted integral and original integral}%
)}\nonumber\\
&  \leq\frac{c}{r^{\alpha}}\frac{1}{\left\vert \widetilde{B}\left(
\xi,r\right)  \right\vert }\int_{\widetilde{B}\left(  \xi,r\right)
}\left\vert \widetilde{f}\left(  \eta\right)  -k\right\vert d\eta\nonumber\\
&  \text{choosing }k=f\left(  x\right)  =\widetilde{f}\left(  \xi\right)
\nonumber\\
&  \leq\frac{c}{r^{\alpha}}\left\vert \widetilde{f}\right\vert _{C_{\widetilde
{X}}^{\alpha}\left(  \widetilde{B}\left(  \xi,r\right)  \right)  }r^{\alpha
}=c\left\vert \widetilde{f}\right\vert _{C_{\widetilde{X}}^{\alpha}\left(
\widetilde{B}\left(  \xi,r\right)  \right)  }. \label{lifted holder}%
\end{align}
Since $r\leq R\left(  t-s\right)  $ and $d\left(  \xi,\overline{\xi}\right)
<\delta_{0}sR,$ we have the inclusion
\[
\widetilde{B}\left(  \xi,r\right)  \subset\widetilde{B}\left(  \overline{\xi
},\delta_{0}sR+R\left(  t-s\right)  \right)  \equiv\widetilde{B}\left(
\overline{\xi},R^{\prime}\right)
\]
so that (\ref{lifted holder}) implies%
\[
M_{\alpha,B\left(  \overline{x},\delta_{0}sR\right)  ,B\left(  \overline
{x},\delta_{0}tR\right)  }\left(  f\right)  \leq c\left\vert \widetilde
{f}\right\vert _{C_{\widetilde{X}}^{\alpha}\left(  \widetilde{B}\left(
\overline{\xi},R^{\prime}\right)  \right)  },
\]
and by Lemma \ref{holder norm and Campanato}, we conclude%
\[
\left\vert f^{\ast}\right\vert _{C_{X}^{\alpha}\left(  B\left(  \overline
{x},\delta_{0}sR\right)  \right)  }\leq\frac{c}{\left(  t-s\right)  ^{2}%
}\left\vert \widetilde{f}\right\vert _{C_{\widetilde{X}}^{\alpha}\left(
\widetilde{B}\left(  \overline{\xi},R^{\prime}\right)  \right)  }.
\]

Note that $R^{\prime}-\delta_{0}sR=R\left(  t-s\right)  ,$ hence changing our
notation as%
\begin{align*}
\delta_{0}sR  &  =s^{\prime}\\
R^{\prime}  &  =t^{\prime}%
\end{align*}
we get%
\[
\left\vert f^{\ast}\right\vert _{C_{X}^{\alpha}\left(  B\left(  \overline
{x},s^{\prime}\right)  \right)  }\leq\frac{c}{\left(  t^{\prime}-s^{\prime
}\right)  ^{2}}\left\vert \widetilde{f}\right\vert _{C_{\widetilde{X}}%
^{\alpha}\left(  \widetilde{B}\left(  \overline{\xi},t^{\prime}\right)
\right)  }%
\]
for $0<s^{\prime}<t^{\prime}<R,$ with $c$ also depending on $R$. This is
(\ref{second}).

Now, inequalities (\ref{higher order holder norm 1}) and
(\ref{higher order holder norm 2}) are also consequences of what we have
proved because $\widetilde{X}_{i}\widetilde{f}=\widetilde{X_{i}f}$, hence the
same reasoning can be iterated to higher order derivatives.
\end{proof}

\subsubsection{$L^{p}$ and Sobolev spaces\label{subsec sobolev spaces}}

We are going to define the Sobolev spaces $S_{X}^{k,p}\left(  \Omega\right)  $
in the present context as in \cite{rs}.

\begin{definition}
[Sobolev spaces]If $X=\left(  X_{0},X_{1},\ldots,X_{q}\right)  $ is any system
of smooth vector fields satisfying H\"{o}rmander's condition in a domain
$\Omega\subset\mathbb{R}^{n}$, the Sobolev space $S_{X}^{k,p}\left(
\Omega\right)  $ ($1\leq p\leq\infty,k$ positive integer) consists of $L^{p}%
$-functions with $k$ (weighted) derivatives with respect to the vector fields
$X_{i}$'s, in $L^{p}$. Explicitly,%
\begin{align*}
\left\Vert u\right\Vert _{S_{X}^{k,p}\left(  \Omega\right)  }  &  =\left\Vert
u\right\Vert _{L^{p}\left(  \Omega\right)  }+\underset{i=1}{\overset{k}{\sum}%
}\left\Vert D^{i}u\right\Vert _{L^{p}\left(  \Omega\right)  },\text{ where}\\
\left\Vert D^{k}u\right\Vert _{L^{p}\left(  \Omega\right)  }  &
=\sum_{\left\vert I\right\vert =k}\left\Vert X_{I}u\right\Vert _{L^{p}\left(
\Omega\right)  }.
\end{align*}

Also, we can define the spaces of functions vanishing at the boundary saying
that $u\in$ $S_{0,X}^{k,p}\left(  \Omega\right)  $ if there exists a sequence
$\left\{  u_{k}\right\}  $ of $C_{0}^{\infty}\left(  \Omega\right)  $
functions converging to $u$ in $S_{X}^{k,p}\left(  \Omega\right)  $.
Similarly, we can define the Sobolev spaces $S_{\widetilde{X}}^{k,p}\left(
\widetilde{B}\right)  $, $S_{\widetilde{X},0}^{k,p}\left(  \widetilde
{B}\right)  $ over a lifted ball $\widetilde{B},$ induced by the
$\widetilde{X}$'s.
\end{definition}

It can be proved (see \cite[Proposition 3.5]{bb2}) that:

\begin{proposition}
\label{Prop Sobolev vanishing}If $u\in S_{X}^{2,p}\left(  \Omega\right)  $ and
$\varphi\in C_{0}^{\infty}\left(  \Omega\right)  ,$ then $u\varphi\in
S_{0,X}^{2,p}\left(  \Omega\right)  ,$ and an analogous property holds for the
space $S_{0,\widetilde{X}}^{2,p}\left(  \widetilde{B}\right)  $.
\end{proposition}

Moreover:

\begin{theorem}
\label{Theorem lifted sobolev}Let $f\in L^{p}\left(  B\left(  x,r\right)
\right)  ,\widetilde{f}\left(  x,h\right)  =f\left(  x\right)  ,$
$\widetilde{B}\left(  \xi,r\right)  $ be the lifted ball of $B\left(
x,r\right)  ,$ with $\xi=\left(  x,0\right)  \in\mathbb{R}^{N}.$ Then%
\begin{align*}
c_{1}\left\Vert f\right\Vert _{L^{p}\left(  B\left(  x,\delta_{0}r\right)
\right)  }  &  \leq\left\Vert \widetilde{f}\right\Vert _{L^{p}\left(
\widetilde{B}\left(  \xi,r\right)  \right)  }\leq c_{2}\left\Vert f\right\Vert
_{L^{p}\left(  B\left(  x,r\right)  \right)  }\\
c_{1}\left\Vert f\right\Vert _{S_{X}^{2,p}\left(  B\left(  x,\delta
_{0}r\right)  \right)  }  &  \leq\left\Vert \widetilde{f}\right\Vert
_{S_{\widetilde{X}}^{2,p}\left(  \widetilde{B}\left(  \xi,r\right)  \right)
}\leq c_{2}\left\Vert f\right\Vert _{S_{X}^{2,p}\left(  B\left(  x,r\right)
\right)  }%
\end{align*}
where $\delta_{0}<1$ is the number appearing in Theorem
\ref{the volume of the ball}.
\end{theorem}

\begin{proof}
The first inequality follows by Theorem \ref{the volume of the ball}; the
second follows by the first one, since
\[
\widetilde{X}_{i}\widetilde{f}=X_{i}\widetilde{f}=\widetilde{\left(
X_{i}f\right)  }.
\]

\end{proof}

\subsubsection{Vanishing mean oscillation\label{subsec VMO}}

The definition of $VMO_{loc}\left(  \Omega_{k},\Omega_{k+1}\right)  $ in an
abstract locally homogeneous space has been recalled in
\S \ \ref{subsec locally hom space} (see Definition \ref{definition local VMO}%
); let us endow our domain $\Omega$ with the structure
\[
\left(  \Omega,\left\{  \Omega_{k}\right\}  _{k},d_{X},dx\right)
\]
of locally homogeneous space described at the end of
\S \ \ref{subsec locally hom space}. Then:

\begin{definition}
[Local $VMO$]We say that $a\in VMO_{X,loc}\left(  \Omega\right)  $ if%
\[
a\in VMO_{loc}\left(  \Omega_{k},\Omega_{k+1}\right)  \text{ for every }k.
\]

\end{definition}

More explicitly, this means that for any fixed $\Omega^{\prime}\Subset\Omega,$
the function%
\[
\eta_{u,\Omega^{\prime},\Omega}^{\ast}(r)=\sup_{t\leq r}\sup_{x_{0}\in
\Omega^{\prime}}\frac{1}{\left\vert B_{t}\left(  x_{0}\right)  \right\vert
}\int_{B_{t}\left(  x_{0}\right)  }|u(x)-u_{B_{t}\left(  x_{0}\right)
}|\,dx,
\]
is finite for $r\leq r_{0}$ and vanishes for $r\rightarrow0$, where $r_{0}$ is
the number such that the local doubling condition of Theorem \ref{Thm NSW}
holds:%
\[
\left\vert B\left(  x,2r\right)  \right\vert \leq c\left\vert B\left(
x,r\right)  \right\vert \text{ for any }x\in\Omega^{\prime},r\leq r_{0}.
\]

As for H\"{o}lder continuous and Sobolev functions, we need a comparison
result for $VMO$ functions in the original variables and the lifted ones. By
Corollary \ref{lifted integral and original integral} we immediately have the following:

\begin{proposition}
\label{Prop lifted VMO}Let $a\in VMO_{X,loc}\left(  \Omega\right)  $ then for
any $\Omega^{\prime}\Subset\Omega,x_{0}\in\Omega^{\prime},B\left(
x_{0},R\right)  $ and $\widetilde{\Omega}_{k}=\widetilde{B}\left(  \xi
_{0},\frac{kR}{k+1}\right)  $ as before, we have that $\widetilde{a}\left(
x,h\right)  =a\left(  x\right)  $ belongs to the class $VMO_{loc}\left(
\widetilde{\Omega}_{k},\widetilde{\Omega}_{k}\right)  $ for every $k,$ with%
\[
\eta_{\widetilde{a},\widetilde{\Omega}_{k},\widetilde{\Omega}_{k+1}}^{\ast
}(r)\leq c\eta_{a,\Omega^{\prime},\Omega}^{\ast}(r).
\]

\end{proposition}

In other words, the $VMO_{loc}$ modulus of the original function $a$ controls
the $VMO_{loc}$ modulus of its lifted version.

\section{Operators of type $\lambda$ and representation formulas}

\subsection{Differential operators and fundamental
solutions\label{subsec diff operators}}

We now define various differential operators that we will handle in the
following. Our main interest is to study the operator%
\[
\mathcal{L}=\sum_{i,j=1}^{q}a_{ij}(x)X_{i}X_{j}+X_{0},
\]
under the Assumption (H) in \S \ \ref{main result}. Recall that in view of
Remark \ref{a0=1} we have set $a_{0}\left(  x\right)  \equiv1.$

For any $\overline{x}\in\Omega$ we can apply the \textquotedblleft lifting
theorem\textquotedblright\ to the vector fields $X_{i}$ (see
\S \ \ref{subsection lifting} for the statement and notation), obtaining new
vector fields $\widetilde{X}_{i}$ which are free up to weight $s$ and satisfy
H\"{o}rmander's condition of weight $s$ in a neighborhood of $\overline{\xi
}=\left(  \overline{x},0\right)  \in\mathbb{R}^{N}$. For $\xi=(x,t)\in
\widetilde{B}\left(  \overline{\xi},R\right)  $, with $\widetilde{B}\left(
\overline{\xi},R\right)  $ as in Lemma \ref{equivalence of distance}, set
\[
\widetilde{a}_{ij}(x,t)=a_{ij}(x),
\]
and let%
\begin{equation}
\widetilde{\mathcal{L}}=\sum_{i,j=1}^{q}\widetilde{a}_{ij}(\xi)\widetilde
{X}_{i}\widetilde{X}_{j}+\widetilde{X}_{0} \label{basic equation}%
\end{equation}
be the lifted operator, defined in $\widetilde{B}\left(  \overline{\xi
},R\right)  $. Next, we freeze $\widetilde{\mathcal{L}}$ at some point
$\xi_{0}\in\widetilde{B}\left(  \overline{\xi},R\right)  $, and consider the
frozen lifted operator:
\begin{equation}
\widetilde{\mathcal{L}}_{0}=\sum_{i,j=1}^{q}\widetilde{a}_{ij}(\xi
_{0})\widetilde{X}_{i}\widetilde{X}_{j}+\widetilde{X}_{0}.
\label{frozen operator}%
\end{equation}
To study $\widetilde{\mathcal{L}}_{0}$, in view of the \textquotedblleft
approximation theorem\textquotedblright\ (Thm.
\ref{Rothschild-Stein's approximation Theorem}), we will consider the
approximating operator, defined on the homogeneous group $\mathbb{G}$:
\[
\mathcal{L}_{0}^{\ast}=\sum_{i,j=1}^{q}\widetilde{a}_{ij}(\xi_{0})Y_{i}%
Y_{j}+Y_{0}%
\]
and its transpose:
\[
\mathcal{L}_{0}^{\ast T}=\sum_{i,j=1}^{q}\widetilde{a}_{ij}(\xi_{0})Y_{i}%
Y_{j}-Y_{0}%
\]
where $\left\{  Y_{i}\right\}  $ are the left invariant vector fields on the
group $\mathbb{G}$ defined in \S \ \ref{subsection lifting}.

We will apply to $\mathcal{L}_{0}^{\ast}$ and $\mathcal{L}_{0}^{\ast T}$
several results proved in \cite{bb1}, which in turn are based on results due
to Folland \cite[Thm. 2.1 and Corollary 2.8]{fo} and Folland-Stein
\cite[Proposition 8.5]{fs}. They are collected in the following theorem:

\begin{theorem}
\label{ith:fundsol:lzero}Assume that the homogeneous dimension of $\mathbb{G}$
is $Q\geq3$. For every $\xi_{0}\in\widetilde{B}\left(  \overline{\xi
},R\right)  $ the operator $\mathcal{L}_{0}^{\ast}$ has a unique fundamental
solution $\Gamma\left(  \xi_{0};\cdot\right)  $ such that: :

$(a)$ $\Gamma\left(  \xi_{0};\cdot\right)  \in C^{\infty}\left(
\mathbb{R}^{N}\setminus\left\{  0\right\}  \right)  ;$

$(b)$ $\Gamma\left(  \xi_{0};\cdot\right)  $ is homogeneous of degree $(2-Q);$

$(c)$ for every test function $f$ and every $v\in\mathbb{R}^{N}$,
\[
f(v)\text{ }=\left(  \mathcal{L}_{0}^{\ast}f\,\ast\,\Gamma\left(  \xi
_{0}\,;\cdot\right)  \right)  (v)=\int_{\mathbb{R}^{N}}\text{ }\Gamma\left(
\xi_{0}\,;\,u^{-1}\circ v\right)  \,\mathcal{L}_{0}^{\ast}f(u)\,du;
\]
moreover, for every $i,j=1,\ldots,q$, there exist constants $\alpha_{ij}%
(\xi_{0})$ such that
\begin{equation}
Y_{i}Y_{j}\,f(v)\text{ }=\text{ }P.V.\int_{\mathbb{R}^{N}}\text{ }Y_{i}%
Y_{j}\Gamma\left(  \xi_{0};\,u^{-1}\circ v\right)  \,\mathcal{L}_{0}^{\ast
}f(u)du\text{ }+\alpha_{ij}(\xi_{0})\cdot\mathcal{L}_{0}^{\ast}f(v);
\label{i10}%
\end{equation}

$(d)$ $Y_{i}Y_{j}\Gamma\left(  \xi_{0}\,;\cdot\right)  \in C^{\infty}\left(
\mathbb{R}^{N}\setminus\left\{  0\right\}  \right)  ;$

$(e)$ $Y_{i}Y_{j}\Gamma\left(  \xi_{0}\,;\cdot\right)  $ is homogeneous of
degree $-Q;$

$(f)$
\[
\int_{r<\left\Vert u\right\Vert <R}\text{ }Y_{i}Y_{j}\Gamma\left(  \xi
_{0};\,u\right)  \,du=\int_{\left\Vert u\right\Vert =1}\text{ }Y_{i}%
Y_{j}\Gamma\left(  \xi_{0};\,u\right)  \,d\sigma(u)=0\text{ \ for every
}R>r>0.
\]

\end{theorem}

In (\ref{i10}) the notation $P.V.\int_{\mathbb{R}^{N}}\left(  ...\right)  du$
stands for $\lim_{\varepsilon\rightarrow0}\int_{\left\Vert u^{-1}\circ
v\right\Vert >\varepsilon}\left(  ...\right)  du.$

\begin{remark}
\label{iioss:selfadj:ker} By \cite[Remark on p.174]{fo}, we know that the
fundamental solution of the transposed operator $\mathcal{L}_{0}^{\ast T}$ is
\[
\Gamma^{T}\left(  \xi_{0};u\right)  =\Gamma\left(  \xi_{0};u^{-1}\right)
=\Gamma\left(  \xi_{0};-u\right)  .
\]
(However, beware that $Y_{i}\Gamma^{T}\left(  \xi_{0};u\right)  \neq
Y_{i}\Gamma\left(  \xi_{0};-u\right)  $).
\end{remark}

Throughout the following, we will set, for $i,j=1,\ldots,q$,
\begin{align*}
\Gamma_{ij}(\xi_{0};u)  &  =Y_{i}Y_{j}\,\left[  \Gamma(\xi_{0};\cdot)\right]
(u);\\
\text{ }\Gamma_{ij}^{T}(\xi_{0};u)  &  =Y_{i}Y_{j}\,\left[  \Gamma^{T}(\xi
_{0};\cdot)\right]  (u).
\end{align*}

A second fundamental result we need contains a bound on the derivatives of
$\Gamma$, uniform with respect to $\xi_{0}$, and is proved in \cite[Thm.
12]{bb1}:

\begin{theorem}
\label{ith:stima:unif:ipo}For every multi-index $\beta$, there exists a
constant $c=c(\beta,\mathbb{G},\mu)$ such that
\[
\underset{\xi\in\widetilde{B}\left(  \overline{\xi},R\right)  }{\underset
{\left\Vert u\right\Vert =1}{{\sup}}}\,\left\vert \left(  \frac{\partial
}{\partial u}\right)  ^{\beta}\Gamma_{ij}\left(  \xi\,;\,u\right)  \right\vert
\,\leq\text{ }c\text{,}%
\]
for any $i,j=1,\ldots,q$; moreover, for the $\alpha_{ij}$'s appearing in
(\ref{i10}), the uniform bound
\[
\underset{\xi\in\widetilde{B}\left(  \overline{\xi},R\right)  }{{\sup}%
}\left\vert \alpha_{ij}(\xi)\right\vert \leq c_{2}%
\]
holds for some constant $c_{2}=c_{2}\left(  \mathbb{G},\mu\right)  $.
\end{theorem}

\begin{remark}
Theorems \ref{ith:fundsol:lzero} and \ref{ith:stima:unif:ipo} still hold when
we replace $\Gamma$ by $\Gamma^{T}$ and $\Gamma_{ij}$ by $\Gamma_{ij}^{T}$.
\end{remark}

\subsection{Operators of type $\lambda$\label{Operators of type}}

As in \cite{rs} and \cite{bb2}, we are going to build a parametrix for
$\widetilde{\mathcal{L}}$ shaped on the homogeneous fundamental solution of
$\mathcal{L}_{0}^{\ast}$. More generally, we need to define a class of
integral operators with different degrees of singularity. The next definition
is adapted from \cite{bb2}, the difference being the necessity, in the present
case, to consider integral kernels shaped on the fundamental solutions of both
$\mathcal{L}_{0}^{\ast}$ and $\mathcal{L}_{0}^{\ast T}$.

\begin{definition}
\label{iidef:kernels:type}For any $\xi_{0}\in\widetilde{B}\left(
\overline{\xi},R\right)  $, we say that $k(\xi_{0};\xi,\eta)$ is a
\emph{frozen kernel of type} $\lambda$ (over the ball $\widetilde{B}\left(
\overline{\xi},R\right)  $), for some nonnegative integer $\lambda$, if for
every positive integer $m$ we can write, for $\xi,\eta\in\widetilde{B}\left(
\overline{\xi},R\right)  ,$%
\begin{align*}
k(\xi_{0};\xi,\eta)  &  =k^{\prime}(\xi_{0};\xi,\eta)+k^{\prime\prime}(\xi
_{0};\xi,\eta)\\
&  =\left\{  \sum_{i=1}^{H_{m}}a_{i}(\xi)b_{i}(\eta)D_{i}\Gamma(\xi_{0}%
;\cdot)+a_{0}(\xi)b_{0}(\eta)D_{0}\Gamma(\xi_{0};\cdot)\right\}  \left(
\Theta(\eta,\xi)\right) \\
&  +\left\{  \sum_{i=1}^{H_{m}}a_{i}^{\prime}(\xi)b_{i}^{\prime}(\eta
)D_{i}^{\prime}\Gamma^{T}(\xi_{0};\cdot)+a_{0}^{\prime}(\xi)b_{0}^{\prime
}(\eta)D_{0}^{\prime}\Gamma^{T}(\xi_{0};\cdot)\right\}  \left(  \Theta
(\eta,\xi)\right)
\end{align*}
where $a_{i},b_{i},a_{i}^{\prime},b_{i}^{\prime}\in C_{0}^{\infty}\left(
\widetilde{B}\left(  \overline{\xi},R\right)  \right)  $ $(i=0,1,\ldots
H_{m})$, $D_{i}$ and $D_{i}^{\prime}$ are differential operators such that:
for $i=1,\ldots,H_{m}\,$, $D_{i}$ and $D_{i}^{\prime}$ are homogeneous of
degree $\leq2-\lambda$ (so that $D_{i}\Gamma(\xi_{0};\cdot)$ and
$D_{i}^{\prime}\Gamma^{T}(\xi_{0};\cdot)$ are homogeneous functions of degree
$\geq\lambda-Q$); $D_{0}$ and $D_{0}^{\prime}$ are differential operators such
that $D_{0}\Gamma(\xi_{0};\cdot)$ and\ $D_{0}^{\prime}\Gamma^{T}(\xi_{0}%
;\cdot)$ have $m$ (weighted) derivatives with respect to the vector fields
$Y_{i}$ $(i=0,1,\ldots,q)$. Moreover, the coefficients of the differential
operators $D_{i},D_{i}^{\prime}$ for $i=0,1,...,H_{m}$ possibly depend also on
the variables $\xi,\eta,$ in such a way that the joint dependence on $\left(
\xi,\eta,u\right)  $ is smooth.
\end{definition}

In order to simplify notation, we will not always express explicitly this
dependence of the coefficients of $D_{i}$ on $\xi,\eta.$ Only when it is
necessary we will write, for instance, $a_{i}(\xi)b_{i}(\eta)D_{i}^{\xi,\eta
}\Gamma(\xi_{0};\Theta(\eta,\xi))$ to recall this dependence.

\begin{remark}
\label{Remark parameter derivative}Note that if a smooth function $c\left(
\xi,\eta,u\right)  $ is $D\left(  \lambda\right)  $-homogeneous of some degree
$\beta$ with respect to $u,$ then any $\xi$ or $\eta$ derivative of $c$ has
the same homogeneity with respect to $u,$ since%
\[
c\left(  \xi,\eta,D\left(  \lambda\right)  u\right)  =\lambda^{\beta}c\left(
\xi,\eta,u\right)  \text{ implies }\frac{\partial c}{\partial\xi_{i}}\left(
\xi,\eta,D\left(  \lambda\right)  u\right)  =\lambda^{\beta}\frac{\partial
c}{\partial\xi_{i}}\left(  \xi,\eta,u\right)  .
\]
Hence any derivative
\[
\left(  \frac{\partial}{\partial\xi_{i}}D_{i}^{\xi,\eta}\right)  \Gamma
(\xi_{0};\cdot),\left(  \frac{\partial}{\partial\eta_{i}}D_{i}^{\xi,\eta
}\right)  \Gamma(\xi_{0};\cdot)
\]
has the same homogeneity as%
\[
D_{i}^{\xi,\eta}\Gamma(\xi_{0};\cdot).
\]

\end{remark}

\begin{definition}
For any $\xi_{0}\in\widetilde{B}\left(  \overline{\xi},R\right)  $, we say
that $T(\xi_{0})$ is a \emph{frozen operator of type} $\lambda\geq1$ (over the
ball $\widetilde{B}\left(  \overline{\xi},R\right)  $) if $k(\xi_{0};\xi
,\eta)$ is a frozen kernel of type $\lambda$ and
\[
T(\xi_{0})f(\xi)=\int_{\widetilde{B}}k(\xi_{0};\xi,\eta)\,f(\eta)\,d\eta
\]
for $f\in C_{0}^{\infty}\left(  \widetilde{B}\left(  \overline{\xi},R\right)
\right)  $. We say that $T(\xi_{0})$ is a \emph{frozen operator of type }$0$
if $k(\xi_{0};\xi,\eta)$ is a frozen kernel of type $0$ and
\[
T(\xi_{0})f(\xi)=P.V.\int_{\widetilde{B}}k(\xi_{0};\xi,\eta)\,f(\eta
)\,d\eta+\alpha\left(  \xi_{0},\xi\right)  f\left(  \xi\right)  ,
\]
where $\alpha$ is a bounded measurable function, smooth in $\xi,$ and the
principal value integral exists. Explicitly, this principal value is defined
by:%
\[
P.V.\int_{\widetilde{B}}k(\xi_{0};\xi,\eta)\,f(\eta)\,d\eta=\lim
_{\varepsilon\rightarrow0}\int_{\left\Vert \Theta\left(  \eta,\xi\right)
\right\Vert >\varepsilon}k(\xi_{0};\xi,\eta)\,f(\eta)\,d\eta.
\]

\end{definition}

\begin{definition}
If $k(\xi_{0};\xi,\eta)$ is a frozen kernel of type $\lambda\geq0$, we say
that $k(\xi;\xi,\eta)$ is a \emph{variable kernel of type }$\lambda$ (over the
ball $\widetilde{B}\left(  \overline{\xi},R\right)  $), and
\[
Tf(\xi)=\int_{\widetilde{B}}k(\xi;\xi,\eta)f(\eta)\,d\eta
\]
is a \emph{variable operator of type} $\lambda$. If $\lambda=0$, the integral
must be taken in principal value sense and a term $\alpha\left(  \xi
,\xi\right)  f\left(  \xi\right)  $ must be added.
\end{definition}

\bigskip

With reference to Definition \ref{iidef:kernels:type}, we will call the
$k^{\prime},k^{\prime\prime}$ parts of $k$ \textquotedblleft frozen kernel of
type $\lambda$ modeled on $\Gamma,\Gamma^{T}$\textquotedblright, respectively.
Analogously we will sometimes speak of frozen operators of type $\lambda$
modeled on $\Gamma$ or $\Gamma^{T},$ to denote that the kernel has this
special form.

A common operation on frozen operators is \emph{transposition}:

\begin{definition}
If $T\left(  \xi_{0}\right)  $ is a frozen operator of type $\lambda\geq0$
over $\widetilde{B}\left(  \overline{\xi},R\right)  ,$we will denote by
$T\left(  \xi_{0}\right)  ^{T}$ the transposed operator, formally defined by%
\[
\int_{\widetilde{B}}f\left(  \xi\right)  T\left(  \xi_{0}\right)  ^{T}g\left(
\xi\right)  d\xi=\int_{\widetilde{B}}g\left(  \xi\right)  T\left(  \xi
_{0}\right)  f\left(  \xi\right)  d\xi
\]
for any $f,g\in C_{0}^{\infty}\left(  \widetilde{B}\left(  \overline{\xi
},R\right)  \right)  .$
\end{definition}

Clearly, if $k\left(  \xi_{0},\xi,\eta\right)  $ is the kernel of $T\left(
\xi_{0}\right)  ,$ then $k\left(  \xi_{0},\eta,\xi\right)  $ is the kernel of
$T\left(  \xi_{0}\right)  ^{T}.$ It is useful to note that:

\begin{proposition}
\label{Prop transposed operators}If $T\left(  \xi_{0}\right)  $ is a frozen
operator of type $\lambda\geq0$ over $\widetilde{B}\left(  \overline{\xi
},R\right)  ,$ modeled on $\Gamma$ or $\Gamma^{T},$ then $T\left(  \xi
_{0}\right)  ^{T}$ is a frozen operator of type $\lambda$, modeled on
$\Gamma^{T},\Gamma,$ respectively. In particular, the transposed of a frozen
operator of type $\lambda$ is still a frozen operator of type $\lambda.$
\end{proposition}

\begin{proof}
Let $D$ be any differential operator on the group $\mathbb{G}.$ For any $f\in
C_{0}^{\infty}\left(  \widetilde{B}\left(  \overline{\xi},R\right)  \right)
,$ let $f^{\prime}\left(  u\right)  =f\left(  -u\right)  .$ Let $D^{\prime}$
be the differential operator defined by the identity%
\[
D^{\prime}f=\left(  D\left(  f^{\prime}\right)  \right)  ^{\prime}.
\]
Clearly, if $D$ is homogeneous of some degree $\beta,$ the same is true for
$D^{\prime};$ if $D\Gamma(\xi_{0};\cdot)$ or\ $D\Gamma^{T}(\xi_{0};\cdot)$
have $m$ (weighted) derivatives with respect to the vector fields $Y_{i}$
$(i=0,1,\ldots,q)$, the same is true for $D^{\prime}\Gamma(\xi_{0};\cdot)$
or\ $D^{\prime}\Gamma^{T}(\xi_{0};\cdot).$ Also, recalling that $\Gamma
^{T}(\xi_{0};u)=\Gamma(\xi_{0};-u),$ we have%
\begin{align*}
\left(  D^{\prime}\Gamma\right)  \left(  u\right)   &  =\left(  D\Gamma
^{T}\right)  \left(  -u\right) \\
\left(  D^{\prime}\Gamma^{T}\right)  \left(  u\right)   &  =\left(
D\Gamma\right)  \left(  -u\right)  .
\end{align*}
Moreover, these identities can be iterated, for instance:%
\[
\left(  D_{1}D_{2}\Gamma\right)  \left(  -u\right)  =\left(  D_{1}\left(
D_{2}\Gamma\right)  \right)  \left(  -u\right)  =\left(  D_{1}^{\prime}\left(
D_{2}\Gamma\right)  ^{\prime}\right)  \left(  u\right)  =\left(  D_{1}%
^{\prime}D_{2}^{\prime}\Gamma^{T}\right)  \left(  u\right)  .
\]
Then, if%
\[
k^{\prime}\left(  \xi_{0},\xi,\eta\right)  =\left\{  \sum_{i=1}^{H_{m}}%
a_{i}(\xi)b_{i}(\eta)D_{i}\Gamma(\xi_{0};\cdot)+a_{0}(\xi)b_{0}(\eta
)D_{0}\Gamma(\xi_{0};\cdot)\right\}  \left(  \Theta(\eta,\xi)\right)
\]
is a frozen kernel of type $\lambda$ modeled on $\Gamma,$%
\begin{align*}
k^{\prime}\left(  \xi_{0},\eta,\xi\right)   &  =\left\{  \sum_{i=1}^{H_{m}%
}a_{i}(\eta)b_{i}(\xi)D_{i}\Gamma(\xi_{0};\cdot)+a_{0}(\xi)b_{0}(\eta
)D_{0}\Gamma(\xi_{0};\cdot)\right\}  \left(  -\Theta(\eta,\xi)\right) \\
&  =\left\{  \sum_{i=1}^{H_{m}}a_{i}(\eta)b_{i}(\xi)D_{i}^{\prime}\Gamma
^{T}(\xi_{0};\cdot)+a_{0}(\xi)b_{0}(\eta)D_{0}^{\prime}\Gamma^{T}(\xi
_{0};\cdot)\right\}  \left(  \Theta(\eta,\xi)\right)
\end{align*}
is a frozen kernel of type $\lambda$ modeled on $\Gamma^{T}.$ Analogously one
can prove the converse.
\end{proof}

\bigskip

We have now to deal with the relations between operators of type $\lambda$ and
the differential operators represented by the vector fields $\widetilde{X}%
_{i}$. This is a study which has been carried out in \cite[\S \ 14]{rs}, and
adapted to nonvariational operators in \cite{bb2}. We are interested in two
main results. Roughly speaking, the first says that the composition, in any
order, of an operator of type $\lambda$ with the $\widetilde{X}_{i}$ or
$\widetilde{X}_{0}$ derivative is an operator of type $\lambda-1$ or
$\lambda-2,$ respectively. The second says that the $\widetilde{X}_{i}$
derivative of an operator of type $\lambda$ can be rewritten as the sum of
other operators of type $\lambda,$ each acting on a different $\widetilde
{X}_{j}$ derivative, plus a suitable remainder. In \cite{rs} these results are
proved only for a system of H\"{o}rmander's vector fields of weight one (that
is, without the drift), and some proofs are quite condensed. Hence we need to
extend and modify some arguments in \cite[\S \ 14]{rs} to cover the present
situation. Moreover, as in \cite{bb2}, we need to keep under careful control
the dependence of any quantity on the frozen point $\xi_{0}$ appearing in
$\Gamma\left(  \xi_{0},\cdot\right)  $. For these and other technical reasons,
we prefer to write complete proofs of these properties, even though they are
not so different from known results. The first result is the following:

\begin{theorem}
\label{main theorem}(See \cite[Thm. 8]{rs}). Suppose $T\left(  \xi_{0}\right)
$ is a frozen operator of type $\lambda\geq1$. Then $\widetilde{X}_{k}T\left(
\xi_{0}\right)  $ and $T\left(  \xi_{0}\right)  \widetilde{X}_{k}$
($k=1,2,...,q$) are operators of type $\lambda-1$. If $\lambda\geq2$, then
$\widetilde{X}_{0}T\left(  \xi_{0}\right)  $ and $T\left(  \xi_{0}\right)
\widetilde{X}_{0}$ are operators of type $\lambda-2$.
\end{theorem}

To prove this, we begin by stating the following two lemmas:

\begin{lemma}
\label{iilem:der:optype2}If $k(\xi_{0};\xi,\eta)$ is a frozen kernel of type
$\lambda\geq1$ over $\widetilde{B}\left(  \overline{\xi},R\right)  $, then
$\left(  \widetilde{X}_{j}k\right)  (\xi_{0};\cdot,\eta)\left(  \xi\right)  $
$(j=1,2,...,q)$ is a frozen kernel of type $\lambda-1$. If \ $\lambda\geq2$,
then $\left(  \widetilde{X}_{0}k\right)  (\xi_{0};\cdot,\eta)\left(
\xi\right)  $ is a frozen kernel of type $\lambda-2$.
\end{lemma}

\begin{proof}
This basically follows by the definition of kernel of type $\lambda$ and
Theorem \ref{Rothschild-Stein's approximation Theorem} in
\S \ \ref{subsection lifting}. When the $\widetilde{X}_{j}$ derivative acts on
the $\xi$ variable of a kernel $D_{i}^{\xi}\Gamma\left(  \xi_{0},\cdot\right)
,$ one also has to take into account Remark \ref{Remark parameter derivative}.

Here we just want to point out the following fact. The prototype of frozen
kernel of type $2$ is the function%
\[
a\left(  \xi\right)  \Gamma\left(  \xi_{0};\Theta(\eta,\xi)\right)  b\left(
\eta\right)  .
\]
Note that the computation
\begin{align*}
&  \widetilde{X}_{i}\left[  a\left(  \cdot\right)  \Gamma\left(  \xi
_{0};\Theta(\eta,\cdot)\right)  b\left(  \eta\right)  \right]  \left(
\xi\right) \\
&  =a\left(  \xi\right)  \left[  \left(  Y_{i}+R_{i}^{\eta}\right)
\Gamma\left(  \xi_{0};\cdot\right)  \right]  \left(  \Theta(\eta,\xi)\right)
b\left(  \eta\right)  +\left(  \widetilde{X}_{i}a\right)  \left(  \xi\right)
\Gamma\left(  \xi_{0};\Theta(\eta,\xi)\right)  b\left(  \eta\right)
\end{align*}
in particular generates the term%
\[
a\left(  \xi\right)  \left(  R_{i}^{\eta}\Gamma\right)  \left(  \xi_{0}%
;\cdot\right)  \left(  \Theta(\eta,\xi)\right)  b\left(  \eta\right)
\]
where the differential operator $R_{i}^{\eta}$ has coefficients depending on
$\eta$. In the proof of Theorem \ref{main theorem} we will see another basic
computation on frozen kernels which generates differential operators with
coefficients also depending on $\xi.$ This is the reason why Definition
\ref{iidef:kernels:type} allows for this kind of dependence.
\end{proof}

\begin{lemma}
\label{iilem:der:optype}If $T(\xi_{0})$ is a frozen operator of type
$\lambda\geq1$ over $\widetilde{B}\left(  \overline{\xi},R\right)  $, then
$\widetilde{X}_{i}T(\xi_{0})$ ($i=1,2,...,q$) is a frozen operator of type
$\lambda-1$. If $\lambda\geq2$, then $\widetilde{X}_{0}T\left(  \xi
_{0}\right)  $ is a frozen operator of type $\lambda-2.$
\end{lemma}

\begin{proof}
With reference to Definition \ref{iidef:kernels:type}, it is enough to
consider the part $k^{\prime}$ of the kernel of $T$, the proof for
$k^{\prime\prime}$ being completely analogous. So, let us consider the
operator $\widetilde{X}_{i}T(\xi_{0})$ ($i=1,2,...,q$), where $T\left(
\xi_{0}\right)  $ has kernel $k^{\prime}$.

If $\lambda>1,$ the result immediately follows by the previous lemma. If
$\lambda=1,$ then%
\[
T(\xi_{0})f\left(  \xi\right)  =\int_{\widetilde{B}\left(  \overline{\xi
},R\right)  }a(\xi)b(\eta)D_{1}\Gamma\left(  \xi_{0};\Theta(\eta,\xi)\right)
f\left(  \eta\right)  d\eta+T^{\prime}(\xi_{0})f\left(  \xi\right)
\]
where $T^{\prime}\left(  \xi_{0}\right)  $ is a frozen operator of type $2$,
and $D_{1}$ is a 1-homogeneous differential operator. We already know that
$\widetilde{X}_{i}T^{\prime}(\xi_{0})$ is a frozen operator of type $1$, so we
are left to show that
\[
\widetilde{X}_{i}\int_{\widetilde{B}\left(  \overline{\xi},R\right)  }%
a(\xi)b(\eta)D_{1}\Gamma(\xi_{0};\left(  \Theta(\eta,\xi)\right)  )f\left(
\eta\right)  d\eta
\]
is a frozen operator of type $0.$ To do this, we have to apply a
distributional argument, which will be used several times in the following:
let us compute, for any $\omega\in C_{0}^{\infty}\left(  \widetilde{B}\left(
\overline{\xi},R\right)  \right)  ,$%
\begin{align*}
&  \int_{\widetilde{B}\left(  \overline{\xi},R\right)  }\widetilde{X}_{i}%
^{T}\omega\left(  \xi\right)  \int_{\widetilde{B}\left(  \overline{\xi
},R\right)  }a(\xi)b(\eta)D_{1}^{\xi}\Gamma(\xi_{0};\left(  \Theta(\eta
,\xi)\right)  )f\left(  \eta\right)  d\eta d\xi\\
&  =\lim_{\varepsilon\rightarrow0}\int_{\widetilde{B}\left(  \overline{\xi
},R\right)  }\widetilde{X}_{i}^{T}\omega\left(  \xi\right)  \int
_{\widetilde{B}\left(  \overline{\xi},R\right)  }a(\xi)b(\eta)\varphi
_{\varepsilon}\left(  \Theta(\eta,\xi)\right)  D_{1}^{\xi}\Gamma(\xi
_{0};\left(  \Theta(\eta,\xi)\right)  )f\left(  \eta\right)  d\eta d\xi
\end{align*}
where $\varphi_{\varepsilon}\left(  u\right)  =\varphi\left(  D\left(
\varepsilon^{-1}\right)  u\right)  $ and $\varphi\in C_{0}^{\infty}\left(
\mathbb{R}^{N}\right)  ,\varphi\left(  u\right)  =0$ for $\left\Vert
u\right\Vert <1,\varphi\left(  u\right)  =1$ for $\left\Vert u\right\Vert >2.$
Here we have written $D_{1}^{\xi}$ to recall that the coefficients of the
differential operator $D_{1}$ also depend (smoothly) on $\xi$ as a parameter.
By Theorem \ref{Rothschild-Stein's approximation Theorem},%
\begin{align}
&  \int_{\widetilde{B}\left(  \overline{\xi},R\right)  }\widetilde{X}_{i}%
^{T}\omega\left(  \xi\right)  \int_{\widetilde{B}\left(  \overline{\xi
},R\right)  }a(\xi)b(\eta)\varphi_{\varepsilon}\left(  \Theta(\eta
,\xi)\right)  D_{1}^{\xi}\Gamma(\xi_{0};\left(  \Theta(\eta,\xi)\right)
)f\left(  \eta\right)  d\eta d\xi\nonumber\\
&  =\int_{\widetilde{B}\left(  \overline{\xi},R\right)  }b(\eta)f\left(
\eta\right)  \int_{\widetilde{B}\left(  \overline{\xi},R\right)  }\left(
\widetilde{X}_{i}^{T}\omega\right)  \left(  \xi\right)  a(\xi)\varphi
_{\varepsilon}\left(  \Theta(\eta,\xi)\right)  D_{1}^{\xi}\Gamma(\xi
_{0};\left(  \Theta(\eta,\xi)\right)  )d\xi d\eta\nonumber\\
&  =\int_{\widetilde{B}\left(  \overline{\xi},R\right)  }b(\eta)f\left(
\eta\right)  \int_{\widetilde{B}\left(  \overline{\xi},R\right)  }%
\omega\left(  \xi\right)  \left(  \widetilde{X}_{i}a\right)  (\xi
)\varphi_{\varepsilon}\left(  \Theta(\eta,\xi)\right)  D_{1}^{\xi}\Gamma
(\xi_{0};\left(  \Theta(\eta,\xi)\right)  )d\xi d\eta\nonumber\\
&  +\int_{\widetilde{B}\left(  \overline{\xi},R\right)  }b(\eta)f\left(
\eta\right)  \int_{\widetilde{B}\left(  \overline{\xi},R\right)  }%
\omega\left(  \xi\right)  a(\xi)\varphi_{\varepsilon}\left(  \Theta(\eta
,\xi)\right)  \left(  \widetilde{X}_{i}D_{1}^{\xi}\right)  \Gamma(\xi
_{0};\left(  \Theta(\eta,\xi)\right)  )d\xi d\eta\nonumber\\
&  +\int_{\widetilde{B}\left(  \overline{\xi},R\right)  }b(\eta)f\left(
\eta\right)  \int_{\widetilde{B}\left(  \overline{\xi},R\right)  }%
\omega\left(  \xi\right)  a(\xi)\left[  \left(  Y_{i}+R_{i}^{\eta}\right)
\left(  \varphi_{\varepsilon}D_{1}^{\xi}\Gamma(\xi_{0};\cdot)\right)  \right]
\left(  \Theta(\eta,\xi)\right)  d\xi d\eta\nonumber\\
&  \equiv A_{\varepsilon}+B_{\varepsilon}+C_{\varepsilon}. \label{A+B+C}%
\end{align}
Now,%
\begin{align}
A_{\varepsilon}  &  \rightarrow\int_{\widetilde{B}\left(  \overline{\xi
},R\right)  }b(\eta)f\left(  \eta\right)  \int_{\widetilde{B}\left(
\overline{\xi},R\right)  }\omega\left(  \xi\right)  \left(  \widetilde{X}%
_{i}a\right)  (\xi)D_{1}\Gamma(\xi_{0};\left(  \Theta(\eta,\xi)\right)  )d\xi
d\eta\nonumber\\
&  =\int_{\widetilde{B}\left(  \overline{\xi},R\right)  }f\left(  \eta\right)
S_{1}\left(  \xi_{0}\right)  \omega\left(  \eta\right)  d\eta\nonumber\\
&  =\int_{\widetilde{B}\left(  \overline{\xi},R\right)  }\omega\left(
\eta\right)  S_{1}\left(  \xi_{0}\right)  ^{T}f\left(  \eta\right)
d\eta\label{A}%
\end{align}
where $S_{1}\left(  \xi_{0}\right)  $ is a frozen operator of type 1, and
$S_{1}\left(  \xi_{0}\right)  ^{T},$ its transpose, is still a frozen operator
of type 1 (see Proposition \ref{Prop transposed operators}).%
\begin{align}
B_{\varepsilon}  &  \rightarrow\int_{\widetilde{B}\left(  \overline{\xi
},R\right)  }b(\eta)f\left(  \eta\right)  \int_{\widetilde{B}\left(
\overline{\xi},R\right)  }\omega\left(  \xi\right)  a(\xi)\left(
\widetilde{X}_{i}D_{1}^{\xi}\right)  \Gamma(\xi_{0};\left(  \Theta(\eta
,\xi)\right)  )d\xi d\eta\nonumber\\
&  =\int_{\widetilde{B}\left(  \overline{\xi},R\right)  }f\left(  \eta\right)
S_{1}^{\prime}\left(  \xi_{0}\right)  \omega\left(  \eta\right)
d\eta\nonumber\\
&  =\int_{\widetilde{B}\left(  \overline{\xi},R\right)  }\omega\left(
\eta\right)  S_{1}^{\prime}\left(  \xi_{0}\right)  ^{T}f\left(  \eta\right)
d\eta\label{B}%
\end{align}
where, by Remark \ref{Remark parameter derivative}, $S_{1}^{\prime}\left(
\xi_{0}\right)  $ is a frozen operator of type 1, and the same is true for
$S_{1}^{\prime}\left(  \xi_{0}\right)  ^{T}$ by Proposition
\ref{Prop transposed operators}.
\begin{align}
C_{\varepsilon}  &  =\int_{\widetilde{B}\left(  \overline{\xi},R\right)
}b(\eta)f\left(  \eta\right)  \int_{\widetilde{B}\left(  \overline{\xi
},R\right)  }\omega\left(  \xi\right)  a(\xi)\left[  \varphi_{\varepsilon
}Y_{i}D_{1}\Gamma(\xi_{0};\cdot)\right]  \left(  \Theta(\eta,\xi)\right)  d\xi
d\eta\nonumber\\
&  +\int_{\widetilde{B}\left(  \overline{\xi},R\right)  }b(\eta)f\left(
\eta\right)  \int_{\widetilde{B}\left(  \overline{\xi},R\right)  }%
\omega\left(  \xi\right)  a(\xi)\left[  \varphi_{\varepsilon}R_{i}^{\eta}%
D_{1}\Gamma(\xi_{0};\cdot)\right]  \left(  \Theta(\eta,\xi)\right)  d\xi
d\eta\nonumber\\
&  +\int_{\widetilde{B}\left(  \overline{\xi},R\right)  }b(\eta)f\left(
\eta\right)  \int_{\widetilde{B}\left(  \overline{\xi},R\right)  }%
\omega\left(  \xi\right)  a(\xi)\left[  \left(  Y_{i}+R_{i}^{\eta}\right)
\varphi_{\varepsilon}D_{1}\Gamma(\xi_{0};\cdot)\right]  \left(  \Theta
(\eta,\xi)\right)  d\xi d\eta\nonumber\\
&  \equiv C_{\varepsilon}^{1}+C_{\varepsilon}^{2}+C_{\varepsilon}^{3}.
\label{C1+C2+C3}%
\end{align}
Now:%
\begin{align}
C_{\varepsilon}^{1}  &  \rightarrow\int_{\widetilde{B}\left(  \overline{\xi
},R\right)  }\omega\left(  \xi\right)  \left\{  P.V.\int_{\widetilde{B}\left(
\overline{\xi},R\right)  }a(\xi)Y_{i}D_{1}\Gamma(\xi_{0};\Theta(\eta
,\xi))b(\eta)f\left(  \eta\right)  d\eta\right\}  d\xi\nonumber\\
&  =\int_{\widetilde{B}\left(  \overline{\xi},R\right)  }\omega\left(
\xi\right)  T\left(  \xi_{0}\right)  f\left(  \xi\right)  d\xi\label{C1}%
\end{align}
with $T\left(  \xi_{0}\right)  $ frozen operator of type 0. Note that the
principal value exists because the kernel $Y_{i}D_{1}\Gamma\left(  \xi
_{0};u\right)  $ has vanishing integral over spherical shells $\left\{
u\in\mathbb{G}:r_{1}<\left\Vert u\right\Vert <r_{2}\right\}  $ (see Theorem
\ref{ith:fundsol:lzero}).%
\begin{align}
C_{\varepsilon}^{2}  &  \rightarrow\int_{\widetilde{B}\left(  \overline{\xi
},R\right)  }\omega\left(  \xi\right)  \left\{  \int_{\left\Vert u\right\Vert
<R}a(\xi)R_{i}^{\eta}D_{1}\Gamma(\xi_{0};\Theta(\eta,\xi))b(\eta)f\left(
\eta\right)  d\eta\right\}  d\xi\nonumber\\
&  =\int_{\widetilde{B}\left(  \overline{\xi},R\right)  }\omega\left(
\xi\right)  S\left(  \xi_{0}\right)  f\left(  \xi\right)  d\xi\label{C2}%
\end{align}
with $S\left(  \xi_{0}\right)  $ frozen operator of type 1.

To handle $C_{\varepsilon}^{3},$ let us perform the change of variables
$u=\Theta(\eta,\xi)$ which, by Theorem \ref{metric} gives
\begin{align*}
C_{\varepsilon}^{3}  &  =\int_{\widetilde{B}\left(  \overline{\xi},R\right)
}\left(  bf\right)  \left(  \eta\right)  \int_{\left\Vert u\right\Vert
<R}\left(  \omega a\right)  \left(  \Theta(\eta,\cdot)^{-1}\left(  u\right)
\right)  \left[  \left(  Y_{i}+R_{i}^{\eta}\right)  \varphi_{\varepsilon}%
D_{1}\Gamma(\xi_{0};\cdot)\right]  \left(  u\right)  \cdot\\
&  \cdot c\left(  \eta\right)  \left(  1+O\left(  \left\Vert u\right\Vert
\right)  \right)  dud\eta
\end{align*}
On the other hand, $Y_{i}\varphi_{\varepsilon}\left(  u\right)  =\frac
{1}{\varepsilon}Y_{i}\varphi\left(  D\left(  \frac{1}{\varepsilon}\right)
u\right)  $, while $R_{i}^{\eta}\varphi_{\varepsilon}\left(  u\right)  $ is
uniformly bounded in $\varepsilon.$ Hence the change of variables $D\left(
\frac{1}{\varepsilon}\right)  u=v$ gives%
\begin{align}
C_{\varepsilon}^{3}  &  =\int_{\widetilde{B}\left(  \overline{\xi},R\right)
}\left(  bf\right)  \left(  \eta\right)  \int_{\left\Vert v\right\Vert
<\frac{R}{\varepsilon}}\left(  \omega a\right)  \left(  \Theta(\eta
,\cdot)^{-1}\left(  D\left(  \varepsilon\right)  v\right)  \right)  \left[
\frac{1}{\varepsilon}Y_{i}\varphi\left(  v\right)  +O\left(  1\right)
\right]  \cdot\nonumber\\
&  \cdot c\left(  \eta\right)  \varepsilon^{1-Q}D_{1}^{\eta}\Gamma\left(
\xi_{0};v\right)  \left(  1+O\left(  \varepsilon\left\Vert v\right\Vert
\right)  \right)  \varepsilon^{Q}dvd\eta\nonumber\\
&  \rightarrow\int_{\widetilde{B}\left(  \overline{\xi},R\right)  }\left(
bcf\right)  \left(  \eta\right)  \int_{\left\Vert v\right\Vert <\frac
{R}{\varepsilon}}\left(  \omega a\right)  \left(  \Theta(\eta,\cdot
)^{-1}\left(  0\right)  \right)  Y_{i}\varphi\left(  v\right)  D_{1}^{\eta
}\Gamma\left(  \xi_{0};v\right)  dvd\eta\nonumber\\
&  =\int_{\widetilde{B}\left(  \overline{\xi},R\right)  }\left(  \omega
abcf\right)  \left(  \eta\right)  \int_{\left\Vert v\right\Vert <\frac
{R}{\varepsilon}}Y_{i}\varphi\left(  v\right)  D_{1}^{\eta}\Gamma\left(
\xi_{0};v\right)  dvd\eta\nonumber\\
&  =\int_{\widetilde{B}\left(  \overline{\xi},R\right)  }\left(  \omega
abcf\right)  \left(  \eta\right)  \alpha\left(  \xi_{0},\eta\right)  d\eta,
\label{alfa}%
\end{align}
which is the integral of $\omega$ times the multiplicative part of a frozen
operator of type $0.$ It is worthwhile (although not logically necessary to
prove the theorem) to realize that the quantity $\alpha\left(  \xi_{0}%
,\eta\right)  $ appearing in (\ref{alfa}) actually does not depend on the
function $\varphi.$ Namely, recalling that $Y_{i}\varphi\left(  v\right)  $ is
supported in the spherical shell $1\leq\left\Vert v\right\Vert \leq2,$ with
$\varphi\left(  u\right)  =1$ for $\left\Vert u\right\Vert =2$ and
$\varphi\left(  u\right)  =0$ for $\left\Vert u\right\Vert =1,$ an integration
by parts gives%
\begin{align*}
&  \int_{1\leq\left\Vert v\right\Vert \leq2}Y_{i}\varphi\left(  v\right)
D_{1}^{\eta}\Gamma\left(  \xi_{0};v\right)  dv\\
&  =-\int_{1\leq\left\Vert v\right\Vert \leq2}\varphi\left(  v\right)
Y_{i}D_{1}^{\eta}\Gamma\left(  \xi_{0};v\right)  dv+\int_{\left\Vert
v\right\Vert =2}D_{1}^{\eta}\Gamma\left(  \xi_{0};v\right)  n_{i}%
d\sigma\left(  v\right)
\end{align*}
with $n_{i}=\sum_{j=1}^{N}b_{ij}\left(  u\right)  \nu_{j}$, where $Y_{i}=$
$\sum_{j=1}^{N}b_{ij}\left(  u\right)  \partial_{u_{j}}$ and $\nu$ is the
outer normal on $\left\Vert v\right\Vert =2.$ The vanishing property of the
kernel $Y_{i}D_{1}^{\xi}\Gamma\left(  \xi_{0};\cdot\right)  $ implies that if
$\varphi$ is a radial function the first integral vanishes. Therefore%
\[
\alpha\left(  \xi_{0},\eta\right)  =\int_{\left\Vert v\right\Vert =2}%
D_{1}^{\eta}\Gamma\left(  \xi_{0};v\right)  n_{i}d\sigma\left(  v\right)
\]
which also shows that $\alpha\left(  \xi_{0},\eta\right)  $ smoothly depends
on $\eta$ and is bounded in $\xi_{0}$ (by Theorem \ref{ith:stima:unif:ipo}).
By (\ref{A+B+C}), (\ref{A}), (\ref{B}), (\ref{C1}), (\ref{C2}), (\ref{alfa})
we have therefore proved that%
\[
\widetilde{X}_{i}T(\xi_{0})f\left(  \xi\right)  =S_{1}\left(  \xi_{0}\right)
^{T}f\left(  \xi\right)  +S_{1}^{\prime}\left(  \xi_{0}\right)  ^{T}f\left(
\xi\right)  +T\left(  \xi_{0}\right)  f\left(  \xi\right)  +\alpha\left(
\xi_{0},\xi\right)  \left(  abcf\right)  \left(  \xi\right)
\]
which is a frozen operator of type 0.

This completes the proof of the first statement of the Lemma. The proof of the
fact that if $\lambda\geq2$ then $\widetilde{X}_{0}T\left(  \xi_{0}\right)  $
is a frozen operator of type $\lambda-2$ is completely analogous.
\end{proof}

\bigskip

The above two lemmas imply the assertion on $\widetilde{X}_{k}T\left(  \xi
_{0}\right)  $ and $\widetilde{X}_{0}T\left(  \xi_{0}\right)  $ in Theorem
\ref{main theorem}. To prove the assertions about $T\left(  \xi_{0}\right)
\widetilde{X}_{k},T\left(  \xi_{0}\right)  \widetilde{X}_{0}$ we need a way to
express $\xi$-derivatives of the integral kernel in terms of $\eta
$-derivatives of the kernel, in order to integrate by parts. This will involve
the use of \emph{right invariant vector fields} on the group $\mathbb{G}$:
throughout the following, we will denote by%
\[
Y_{i,k}^{R}%
\]
the right invariant vector field on $\mathbb{G}$ satisfying $Y_{i,k}%
^{R}f(0)=Y_{i,k}f(0)$. We have the following:

\begin{lemma}
\label{change differential orders}For any $f\in C_{0}^{\infty}(\mathbb{G})$
and $\eta,\xi$ in a neighborhood of $\xi_{0}$, we can write, for any
$i=1,2,...,s,$ $k=1,2,...,k_{i}$ (recall $s$ is the step of the Lie algebra)%
\begin{equation}
\widetilde{X}_{i,k}\left[  f\left(  \Theta\left(  \cdot,\xi\right)  \right)
\right]  \left(  \eta\right)  =-\left(  Y_{i,k}^{R}f\right)  \left(
\Theta\left(  \eta,\xi\right)  \right)  +\left(  \left(  R_{i,k}^{\xi}\right)
^{\prime}f\right)  \left(  \Theta\left(  \eta,\xi\right)  \right)  , \label{5}%
\end{equation}
where $\left(  R_{i,k}^{\xi}\right)  ^{\prime}$ is a vector field of local
degree $\leq i-1$ smoothly depending on $\xi$.
\end{lemma}

\begin{proof}
We start with the following

\textbf{Claim.} For any function $f$ defined on $\mathbb{G}$, let%
\[
f^{\prime}\left(  u\right)  =f\left(  -u\right)
\]
(recall that $-u=u^{-1}$); then the following identities hold:%
\begin{equation}
Y_{i,k}\left(  f^{\prime}\right)  =-\left(  Y_{i,k}^{R}f\right)  ^{\prime}.
\label{right invariant}%
\end{equation}

To prove this, let us define the vector fields $\widehat{Y}_{i,k}$ by%
\begin{equation}
Y_{i,k}\left(  f^{\prime}\right)  =-\left(  \widehat{Y}_{i,k}f\right)
^{\prime}, \label{def Y hat}%
\end{equation}
then for any $a\in\mathbb{G}$, denoting by $L_{a},R_{a}$ the corresponding
operators of left and right translation, respectively (acting on functions),
we have%
\begin{align*}
(\widehat{Y}_{i,k}R_{a}f)^{\prime}  &  =-Y_{i,k}((R_{a}f)^{\prime}%
)=-Y_{i,k}(L_{-a}f^{\prime})=\\
&  =-L_{-a}Y_{i,k}f^{\prime}=L_{-a}(-Y_{i,k}f^{\prime})=\\
&  =L_{-a}(\widehat{Y}_{i,k}f)^{\prime}=(R_{a}\widehat{Y}_{i,k}f)^{\prime},
\end{align*}
hence $\widehat{Y}_{i,k}$ are right invariant vector fields. Also, note that
for any vector field $Y=\sum a_{j}\left(  u\right)  \partial_{u_{j}}$ we have%
\[
Y\left(  f^{\prime}\right)  \left(  0\right)  =-\left(  Yf\right)  \left(
0\right)
\]
because%
\begin{align*}
Y\left(  f^{\prime}\right)  \left(  u\right)   &  =\sum a_{j}\left(  u\right)
\partial_{u_{j}}\left[  f\left(  -u\right)  \right]  =-\sum a_{j}\left(
u\right)  \left(  \partial_{u_{j}}f\right)  \left(  -u\right)  \text{
implies}\\
Y\left(  f^{\prime}\right)  \left(  0\right)   &  =-\sum a_{j}\left(
0\right)  \left(  \partial_{u_{j}}f\right)  \left(  0\right)  =-\left(
Yf\right)  \left(  0\right)
\end{align*}
hence by (\ref{def Y hat}) we know that $\widehat{Y}_{k}f(0)=Y_{k}f(0)$.
Therefore $\widehat{Y}_{k}$ is the right invariant vector field which
coincides with $Y_{k}$ at the origin, that is $\widehat{Y}_{k}=Y_{k}^{R},$ and
the Claim is proved.

By (\ref{approximation format}) and (\ref{right invariant}),%
\begin{align}
\widetilde{X}_{i,k}\left[  f\left(  \Theta\left(  \cdot,\xi\right)  \right)
\right]  \left(  \eta\right)   &  =\widetilde{X}_{i,k}\left[  f^{\prime
}\left(  \Theta\left(  \xi\,,\cdot\right)  \right)  \right]  \left(
\eta\right)  =\label{14.3"}\\
&  =\left(  Y_{i,k}f^{\prime}+R_{i,k}^{\xi}f^{\prime}\right)  \left(
\Theta\left(  \xi\,,\eta\right)  \right)  =\nonumber\\
&  =-\left(  Y_{i,k}^{R}f\right)  ^{\prime}\left(  \Theta\left(  \xi
\,,\eta\right)  \right)  +R_{i,k}^{\xi}f^{\prime}\left(  \Theta\left(
\xi\,,\eta\right)  \right)  =\nonumber\\
&  =-\left(  Y_{i,k}^{R}f\right)  \left(  \Theta\left(  \eta,\xi\right)
\right)  +\left(  \left(  R_{i,k}^{\xi}\right)  ^{\prime}f\right)  \left(
\Theta\left(  \eta,\xi\right)  \right)  ,\nonumber
\end{align}
where%
\[
\left(  \left(  R_{i,k}^{\xi}\right)  ^{\prime}f\right)  \left(  u\right)
=\left(  R_{i,k}^{\xi}f^{\prime}\right)  \left(  -u\right)
\]
is a differential operator of degree $\leq i-1$. This proves (\ref{5}).
\end{proof}

\bigskip

\begin{proof}
[Proof of Theorem \ref{main theorem}]As we noted after Lemma
\ref{iilem:der:optype}, we are left to prove the assertion about $T\left(
\xi_{0}\right)  \widetilde{X}_{i}$ and $T\left(  \xi_{0}\right)  \widetilde
{X}_{0}$. We only give the proof for the case $\lambda\geq1$, $i=1,\cdots,q$,
the proof for $\lambda\geq2$, $i=0$ being very similar. Like in the proof of
Lemma \ref{iilem:der:optype}, it is enough to consider the part $k^{\prime}$
of the kernel of $T$, the proof for $k^{\prime\prime}$ being completely
analogous (see Definition \ref{iidef:kernels:type}). Let us expand%
\[
k^{\prime}(\xi_{0};\xi,\eta)=\left\{  \sum_{j=1}^{H_{m}}a_{j}(\xi)b_{j}%
(\eta)D_{j}\Gamma(\xi_{0};\cdot)+a_{0}(\xi)b_{0}(\eta)D_{0}\Gamma(\xi
_{0};\cdot)\right\}  \left(  \Theta(\eta,\xi)\right)
\]
where $D_{0}\Gamma(\xi_{0};\cdot)$ has bounded $Y_{i}$-derivatives
($i=1,2,...,q$). We can consider each of the terms%
\[
T_{j}\left(  \xi_{0}\right)  \widetilde{X}_{i}f\left(  \xi\right)  \equiv\int
a_{j}\left(  \xi\right)  b_{j}\left(  \eta\right)  D_{j}^{\eta}\Gamma(\xi
_{0};\Theta(\eta,\xi))\widetilde{X}_{i}f\left(  \eta\right)  d\eta
\]
(this time it is important to recall the $\eta$-dependence of the coefficients
of $D_{j}$) and distinguish 2 cases:

(i) $D_{j}\Gamma$ is homogeneous of degree $\geq2-Q$ or it is regular (i.e.
$D_{j}\Gamma$ has bounded $Y_{i}$-derivatives);

(ii) $T_{j}\left(  \xi_{0}\right)  $ is a frozen operator of type $1$ and
$D_{j}\Gamma$ is homogeneous of degree $1-Q.$

Case (i). We can integrate by parts, recalling that the transpose of
$\widetilde{X}_{i}$ is
\[
\left(  \widetilde{X}_{i}\right)  ^{T}g\left(  \eta\right)  =-\widetilde
{X}_{i}g\left(  \eta\right)  +c_{i}\left(  \eta\right)  g\left(  \eta\right)
\]
with $c_{i}$ smooth functions:%
\begin{align*}
T_{j}\left(  \xi_{0}\right)  \widetilde{X}_{i}f\left(  \xi\right)   &  =\int
c_{i}\left(  \eta\right)  a_{j}\left(  \xi\right)  b_{j}\left(  \eta\right)
D_{j}^{\eta}\Gamma(\xi_{0};\Theta(\eta,\xi))f\left(  \eta\right)  d\eta\\
&  -\int a_{j}\left(  \xi\right)  \left(  \widetilde{X}_{i}b_{j}\right)
\left(  \eta\right)  D_{j}^{\eta}\Gamma(\xi_{0};\Theta(\eta,\xi))f\left(
\eta\right)  d\eta\\
&  -\int a_{j}\left(  \xi\right)  b_{j}\left(  \eta\right)  \widetilde{X}%
_{i}\left[  D_{j}^{\eta}\Gamma(\xi_{0};\Theta(\cdot,\xi))\right]  \left(
\eta\right)  f\left(  \eta\right)  d\eta\\
&  -\int a_{j}\left(  \xi\right)  b_{j}\left(  \eta\right)  \left(
\widetilde{X}_{i}^{\eta}D_{j}^{\eta}\right)  \Gamma(\xi_{0};\Theta(\eta
,\xi))f\left(  \eta\right)  d\eta\\
&  =A\left(  \xi\right)  +B\left(  \xi\right)  +C\left(  \xi\right)  +D\left(
\xi\right)  .
\end{align*}
Now, $A\left(  \xi\right)  +B\left(  \xi\right)  \ $is still an operator of
type $\lambda$, applied to $f$; in particular, it can be seen as operator of
type $\lambda-1$; the same is true for $D\left(  \xi\right)  $, by Remark
\ref{Remark parameter derivative}. To study $C\left(  \xi\right)  $, we apply
Lemma \ref{change differential orders},
\[
\widetilde{X}_{i}\left[  D_{j}^{\eta}\Gamma(\xi_{0};\Theta(\cdot,\xi))\right]
\left(  \eta\right)  =-\left(  Y_{i}^{R}D_{j}^{\eta}\Gamma\right)  \left(
\xi_{0},\Theta\left(  \eta,\xi\right)  \right)  +\left(  \left(  R_{i}^{\xi
}\right)  ^{\prime}D_{j}^{\eta}\Gamma\right)  \left(  \xi_{0},\Theta\left(
\eta,\xi\right)  \right)  .
\]
Since $Y_{i}^{R}$ is homogeneous of degree $1$, $a_{j}\left(  \xi\right)
b_{j}\left(  \eta\right)  Y_{i}^{R}D_{j}^{\eta}\Gamma\left(  \xi_{0}%
,\Theta\left(  \eta,\xi\right)  \right)  $ is a kernel of type $\lambda-1$.
Since $\left(  R_{i}^{\xi}\right)  ^{\prime}$ is a differential operator of
degree $\leq0$, the kernel $a_{j}\left(  \xi\right)  b_{j}\left(  \eta\right)
\left(  \left(  R_{i}^{\xi}\right)  ^{\prime}D_{j}^{\eta}\Gamma\right)
\left(  \xi_{0},\Theta\left(  \eta,\xi\right)  \right)  $ is of type $\lambda$.

Note that, even when the coefficients of the differential operator $D_{j}$ (in
the expression $D_{j}\Gamma(\xi_{0};\Theta(\eta,\xi))$) do not depend on $\xi$
and $\eta,$ this procedure introduces, with the operator $\left(  R_{i}^{\xi
}\right)  ^{\prime}$, a new $\xi$-dependence of the coefficients. Compare with
what we have remarked in the proof of Lemma \ref{iilem:der:optype2}.

Case (ii). In this case the kernel $\left(  Y_{i}^{R}D_{j}\Gamma\right)  $ is
singular, so that the computation must be handled with more care. We can write%
\begin{align*}
&  T_{j}\left(  \xi_{0}\right)  \widetilde{X}_{i}f\left(  \xi\right)  =\\
&  =\lim_{\varepsilon\rightarrow0}\int a_{j}\left(  \xi\right)  b_{j}\left(
\eta\right)  \varphi_{\varepsilon}\left(  \Theta(\xi,\eta)\right)  D_{j}%
\Gamma(\xi_{0};\Theta(\eta,\xi))\widetilde{X}_{i}f\left(  \eta\right)
d\eta\equiv\lim_{\varepsilon\rightarrow0}T_{\varepsilon}\left(  \xi\right)
\end{align*}
with $\varphi_{\varepsilon}$ as in the proof of Lemma \ref{iilem:der:optype}.
Note that, choosing a radial $\varphi$, we have $\varphi_{\varepsilon}\left(
\Theta(\xi,\eta)\right)  =\varphi_{\varepsilon}\left(  \Theta(\eta
,\xi)\right)  .$ Then%
\begin{align*}
T_{\varepsilon}\left(  \xi\right)   &  =\int c_{i}\left(  \eta\right)
a_{j}\left(  \xi\right)  b_{j}\left(  \eta\right)  \varphi_{\varepsilon
}\left(  \Theta(\xi,\eta)\right)  D_{j}\Gamma(\xi_{0};\Theta(\eta
,\xi))f\left(  \eta\right)  d\eta\\
&  -\int a_{j}\left(  \xi\right)  \left(  \widetilde{X}_{i}b_{j}\right)
\left(  \eta\right)  \varphi_{\varepsilon}\left(  \Theta(\xi,\eta)\right)
D_{j}\Gamma(\xi_{0};\Theta(\eta,\xi))f\left(  \eta\right)  d\eta\\
&  -\int a_{j}\left(  \xi\right)  b_{j}\left(  \eta\right)  \widetilde{X}%
_{i}\left[  \varphi_{\varepsilon}\left(  \Theta(\cdot,\xi)\right)  D_{j}%
\Gamma(\xi_{0};\Theta(\cdot,\xi))\right]  \left(  \eta\right)  f\left(
\eta\right)  d\eta\\
&  -\int a_{j}\left(  \xi\right)  b_{j}\left(  \eta\right)  \varphi
_{\varepsilon}\left(  \Theta(\xi,\eta)\right)  \left(  \widetilde{X}_{i}%
^{\eta}D_{j}^{\eta}\right)  \Gamma(\xi_{0};\Theta(\eta,\xi))f\left(
\eta\right)  d\eta\\
&  =A_{\varepsilon}\left(  \xi\right)  +B_{\varepsilon}\left(  \xi\right)
+C_{\varepsilon}\left(  \xi\right)  +D_{\varepsilon}\left(  \xi\right)  .
\end{align*}
Now $A_{\varepsilon}\left(  \xi\right)  +B_{\varepsilon}\left(  \xi\right)
+D_{\varepsilon}\left(  \xi\right)  $ converge to an operator of type
$\lambda,$ as $A\left(  \xi\right)  ,B\left(  \xi\right)  ,D\left(
\xi\right)  $ are in Case (i), while by Theorem
\ref{Rothschild-Stein's approximation Theorem} and Lemma
\ref{change differential orders}
\begin{align*}
C_{\varepsilon}\left(  \xi\right)   &  =-\int a_{j}\left(  \xi\right)
b_{j}\left(  \eta\right)  f\left(  \eta\right)  \left(  Y_{i}\varphi
_{\varepsilon}\right)  \left(  \Theta(\eta,\xi)\right)  D_{j}\Gamma(\xi
_{0};\Theta(\eta,\xi))d\eta\\
&  -\int a_{j}\left(  \xi\right)  b_{j}\left(  \eta\right)  f\left(
\eta\right)  \left(  R_{i}^{\xi}\varphi_{\varepsilon}\right)  \left(
\Theta(\eta,\xi)\right)  D_{j}\Gamma(\xi_{0};\Theta(\eta,\xi))d\eta\\
&  +\int a_{j}\left(  \xi\right)  b_{j}\left(  \eta\right)  f\left(
\eta\right)  \varphi_{\varepsilon}\left(  \Theta(\eta,\xi)\right)  \left(
Y_{i}^{R}D_{j}\Gamma\right)  \left(  \xi_{0},\Theta\left(  \eta,\xi\right)
\right)  d\eta\\
&  -\int a_{j}\left(  \xi\right)  b_{j}\left(  \eta\right)  f\left(
\eta\right)  \varphi_{\varepsilon}\left(  \Theta\left(  \eta,\xi\right)
\right)  \left(  \left(  R_{i}^{\xi}\right)  ^{\prime}D_{j}\Gamma\right)
\left(  \xi_{0},\Theta\left(  \eta,\xi\right)  \right)  d\eta\\
&  \equiv E_{\varepsilon}\left(  \xi\right)  +F_{\varepsilon}\left(
\xi\right)  +G_{\varepsilon}\left(  \xi\right)  +H_{\varepsilon}\left(
\xi\right)  .
\end{align*}
Now: $H_{\varepsilon}\left(  \xi\right)  $ tends to an operator of type $1;$
$G_{\varepsilon}\left(  \xi\right)  $ tends to%
\[
P.V.\int a_{j}\left(  \xi\right)  b_{j}\left(  \eta\right)  f\left(
\eta\right)  \left(  Y_{i}^{R}D_{j}\Gamma\right)  \left(  \xi_{0}%
,\Theta\left(  \eta,\xi\right)  \right)  d\eta,
\]
which is an operator of type $0.$ As to $E_{\varepsilon}\left(  \xi\right)  ,$
the same computation performed in the proof of Lemma \ref{iilem:der:optype}
gives%
\[
E_{\varepsilon}\left(  \xi\right)  \rightarrow\alpha\left(  \xi_{0}%
,\xi\right)  \left(  abcf\right)  \left(  \xi\right)
\]
with
\[
\alpha\left(  \xi_{0},\xi\right)  =\int Y_{i}\varphi\left(  v\right)
D_{1}^{\xi}\Gamma\left(  \xi_{0};v\right)  dv
\]
which is the multiplicative part of an operator of type $0.$ A similar
computation shows that $F_{\varepsilon}\left(  \xi\right)  \rightarrow0,$ so
we are done.
\end{proof}

\bigskip

Let us come to the second important result of this section. In \cite[Corollary
p. 296]{rs}, the following fact is proved in the case of a family of
H\"{o}rmander's vector fields of weight one (that is, without the drift
$\widetilde{X}_{0}$): for any frozen operator $T\left(  \xi_{0}\right)  $ of
type $1,$ $i=1,2,...,q,$ there exist operators $T_{ij}\left(  \xi_{0}\right)
,T_{i}\left(  \xi_{0}\right)  $ of type $1$ such that%
\[
\widetilde{X}_{i}T\left(  \xi_{0}\right)  =\sum_{j=1}^{q}T_{ij}\left(  \xi
_{0}\right)  \widetilde{X}_{j}+T_{i}\left(  \xi_{0}\right)  .
\]
This possibility of exchanging the order of integral and differential
operators will be crucial in the proof of representation formulas. However,
such an identity cannot be proved in this form when the drift $\widetilde
{X}_{0}$ is present. Instead, we are going to prove the following, which will
be enough for our purposes:

\begin{theorem}
\label{scambioderivoper}If $T\left(  \xi_{0}\right)  $ is a frozen operator of
type $\lambda\geq1,$ $k_{0}=1,2\ldots q$, then%
\begin{equation}
\widetilde{X}_{k_{0}}T\left(  \xi_{0}\right)  =\sum_{k=1}^{q}T_{k}^{k_{0}%
}\left(  \xi_{0}\right)  \widetilde{X}_{k}+\sum_{h,j=1}^{q}\widetilde{a}%
_{hj}\left(  \xi_{0}\right)  T^{hk_{0}}\left(  \xi_{0}\right)  \widetilde
{X}_{j}+T_{0}^{k_{0}}\left(  \xi_{0}\right)  +T^{k_{0}}\left(  \xi_{0}\right)
\widetilde{\mathcal{L}}_{0}, \label{commute of operator}%
\end{equation}
where $T_{k}^{k_{0}}\left(  \xi_{0}\right)  $ $\left(  k=0,1,...,q\right)  $
and $T^{hk_{0}}\left(  \xi_{0}\right)  $ are frozen operators of type
$\lambda$, $T^{k_{0}}\left(  \xi_{0}\right)  $ are frozen operators of type
$\lambda+1$, and $\widetilde{a}_{hj}\left(  \xi_{0}\right)  $ are the frozen
coefficients of $\widetilde{\mathcal{L}}_{0}$.

If $T\left(  \xi_{0}\right)  $ is a frozen operator of type $\lambda\geq2$,
then
\begin{equation}
\widetilde{X}_{0}T\left(  \xi_{0}\right)  =\sum_{k=1}^{q}T_{k}\left(  \xi
_{0}\right)  \widetilde{X}_{k}+\sum_{h,j=1}^{q}\widetilde{a}_{hj}\left(
\xi_{0}\right)  T^{h}\left(  \xi_{0}\right)  \widetilde{X}_{j}+T_{0}\left(
\xi_{0}\right)  +T\left(  \xi_{0}\right)  \widetilde{\mathcal{L}}_{0},
\label{commutate of operator 2}%
\end{equation}
where $T_{k}\left(  \xi_{0}\right)  $ $\left(  k=0,1,...,q\right)  $ and
$T^{h}\left(  \xi_{0}\right)  $ are frozen operators of type $\lambda-1,$
$T\left(  \xi_{0}\right)  $ is a frozen operator of type $\lambda.$
\end{theorem}

We start with the following lemma, which is similar to that proved in \cite[p.
296]{rs}. Again, we prefer to present a detailed proof since the one given in
\cite{rs} is very condensed.

\begin{lemma}
\label{lemma change differential order on domain}For any vector field
$\widetilde{X}_{j_{0},k_{0}}$ ($j_{0}=1,2,...,s$, $k_{0}=1,2,...,k_{j_{0}}$)
there exist smooth functions $\left\{  a_{jk}^{j_{0}k_{0}\eta}\right\}
_{\substack{j=1,2,...,s\\k=1,2,...h_{j}}}$ having local degree $\geq
\max\left\{  j-j_{0},0\right\}  $ and smoothly depending on $\eta,$ such that
for any $f\in C_{0}^{\infty}\left(  \mathbb{G}\right)  $, one can write
\begin{align}
&  \widetilde{X}_{j_{0},k_{0}}\left[  f\left(  \Theta\left(  \eta
,\cdot\right)  \right)  \right]  \left(  \xi\right)
=\label{change differential order on domain}\\
&  =\sum_{\substack{j=1,2,...,s\\k=1,2,...,k_{j}}}a_{jk}^{j_{0}k_{0}\eta
}\left(  \Theta\left(  \eta,\xi\right)  \right)  \widetilde{X}_{j,k}\left[
f\left(  \Theta\left(  \cdot,\xi\right)  \right)  \right]  \left(
\eta\right)  +\left(  R_{j_{0}}^{\xi,\eta}f\right)  \left(  \Theta\left(
\eta,\xi\right)  \right) \nonumber
\end{align}
where $R_{j_{0}}^{\xi,\eta}$ is a vector field of local degree $\leq j_{0}-1$,
smoothly depending on $\xi,\eta.$
\end{lemma}

\begin{proof}
By Theorem \ref{Rothschild-Stein's approximation Theorem} we know that%
\begin{equation}
\widetilde{X}_{j_{0},k_{0}}\left[  f\left(  \Theta\left(  \eta,\cdot\right)
\right)  \right]  \left(  \xi\right)  =\left(  Y_{j_{0},k_{0}}f+R_{j_{0}%
,k_{0}}^{\eta}f\right)  \left(  \Theta\left(  \eta,\xi\right)  \right)
\equiv\left(  Z_{j_{0}}^{\eta}f\right)  \left(  \Theta\left(  \eta,\xi\right)
\right)  , \label{RS}%
\end{equation}
where $Z_{j_{0}}^{\eta}$ is a vector field of local degree $\leq j_{0}$,
smoothly depending on $\eta.$ To rewrite $\left(  Z_{j_{0}}^{\eta}f\right)  $
in the suitable form, we start from the following identities:%
\begin{equation}
Y_{i,k}=\frac{\partial}{\partial u_{ik}}+\sum\limits_{r}\sum_{i<l\leq s}%
g_{lr}^{ik}\left(  u\right)  \frac{\partial}{\partial u_{lr}} \label{system}%
\end{equation}
for any $i=1,2,...,s$ and $k=1,2,...,k_{i};$%
\begin{equation}
Y_{i,k}=\sum g_{lr}^{ik}\left(  u\right)  Y_{l,r}^{R},
\label{left and right relation}%
\end{equation}
where $g_{lr}^{ik}\left(  u\right)  $ are homogeneous of degree $l-i$ (see
\cite[p. 295]{rs}). Hence we can write%
\[
Z_{j_{0}}^{\eta}=\sum a_{jk}^{\eta}\left(  u\right)  \frac{\partial}{\partial
u_{jk}},
\]
where $a_{jk}$ has local degree $\geq j-j_{0}$ and smoothly depends on $\eta.$
By inverting (for any $i,k$) the triangular system (\ref{system}), we obtain%
\[
\frac{\partial}{\partial u_{jk}}=Y_{j,k}+\sum_{j<l\leq s}f_{lr}^{jk}\left(
u\right)  Y_{l,r},
\]
\bigskip where each $f_{lr}^{jk}\left(  u\right)  $ is homogeneous of degree
$l-j$. Using also (\ref{left and right relation}), we have%
\begin{align}
\left(  Z_{j_{0}}^{\eta}f\right)  \left(  u\right)   &  =\sum a_{jk}^{\eta
}\left(  u\right)  \left[  \left(  Y_{j,k}f\right)  \left(  u\right)
+\sum_{j<l\leq s}f_{lr}^{jk}\left(  u\right)  \left(  Y_{l,r}f\right)  \left(
u\right)  \right] \nonumber\\
&  =\sum b_{lr}^{\eta}\left(  u\right)  \left(  Y_{l,r}^{R}f\right)  \left(
u\right)  , \label{Rf}%
\end{align}
where
\begin{equation}
b_{lr}^{\eta}\text{ has local degree }\geq\max\left\{  l-j_{0},0\right\}
\label{b_lr}%
\end{equation}
and smoothly depends on $\eta.$ By Lemma \ref{change differential orders},
then%
\begin{align}
\left(  Z_{j_{0}}^{\eta}f\right)  \left(  \Theta\left(  \eta,\xi\right)
\right)   &  =\sum_{l,r}-b_{lr}^{\eta}\left(  \Theta\left(  \eta,\xi\right)
\right)  \widetilde{X}_{l,r}\left[  f\left(  \Theta\left(  \cdot,\xi\right)
\right)  \right]  \left(  \eta\right) \nonumber\\
&  +\sum_{l,r}\left(  b_{lr}^{\eta}\left(  R_{l,r}^{\xi}\right)  ^{\prime
}f\right)  \left(  \Theta\left(  \eta,\xi\right)  \right)  , \label{Rf 2}%
\end{align}
where $\left(  R_{l,r}^{\xi}\right)  ^{\prime}$ is a differential operator of
local degree $\leq l-1$, hence the differential operator on $\mathbb{G}$%
\begin{equation}
R_{j_{0}}^{\xi,\eta}\equiv\sum_{l,r}b_{lr}^{\eta}\left(  R_{l,r}^{\xi}\right)
^{\prime}\text{ has local degree }\leq j_{0}-1, \label{R_j0}%
\end{equation}
and depends smoothly on $\xi,\eta.$ Collecting (\ref{RS}), (\ref{b_lr}),
(\ref{Rf 2}), (\ref{R_j0}), the Lemma is proved, with $a_{jk}^{j_{0}k_{0}\eta
}=-b_{jk}^{\eta}$.
\end{proof}

With this lemma in hand, we can prove the following, similar to \cite[Thm
9]{rs}:

\begin{theorem}
(i)\ Suppose $T\left(  \xi_{0}\right)  $ is a frozen operator of type
$\lambda\geq1$. Given a vector field $\widetilde{X}_{i}$ for $i=1,2,...,q,$
there exist $T^{i}\left(  \xi_{0}\right)  ,$ frozen operator of type
$\lambda,$ and $T_{jk}^{i}\left(  \xi_{0}\right)  $, frozen operators of type
$\lambda+j-1,$ such that:%
\begin{equation}
\widetilde{X}_{i}T\left(  \xi_{0}\right)  =\sum_{j,k}T_{jk}^{i}\left(  \xi
_{0}\right)  \widetilde{X}_{j,k}+T^{i}\left(  \xi_{0}\right)  ; \label{1}%
\end{equation}

(ii) Suppose $T\left(  \xi_{0}\right)  $ is a frozen operator of type
$\lambda\geq2$. There exist $T^{0}\left(  \xi_{0}\right)  ,T_{jk}^{i}\left(
\xi_{0}\right)  $ frozen operators of type $\lambda-1,$ $\lambda+\max\left\{
j-2,0\right\}  ,$ respectively, such that:%
\begin{equation}
\widetilde{X}_{0}T\left(  \xi_{0}\right)  =\sum_{j,k}T_{jk}^{0}\left(  \xi
_{0}\right)  \widetilde{X}_{j,k}+T^{0}\left(  \xi_{0}\right)  ; \label{2}%
\end{equation}

\end{theorem}

\begin{proof}
First of all, it is enough to consider the part $k^{\prime}$ of the kernel of
$T\left(  \xi_{0}\right)  $, the proof for $k^{\prime\prime}$ being completely
analogous (see Definition \ref{iidef:kernels:type}).

(i) If $T\left(  \xi_{0}\right)  $ is a frozen operator of type $\lambda\geq1$
with kernel $k^{\prime}$ we can write it as%
\[
T\left(  \xi_{0}\right)  f\left(  \xi\right)  =\int a\left(  \xi\right)
D\Gamma\left(  \xi_{0};\Theta\left(  \eta,\xi\right)  \right)  b\left(
\eta\right)  f\left(  \eta\right)  d\eta+T^{\prime}\left(  \xi_{0}\right)
f\left(  \xi\right)
\]
where $D\Gamma\left(  \xi_{0},\cdot\right)  $ is homogeneous of degree
$\lambda-Q$ and $T^{\prime}\left(  \xi_{0}\right)  $ is a frozen operator of
degree $\lambda+1.$ Since $\widetilde{X}_{i}T^{\prime}\left(  \xi_{0}\right)
$ is a frozen operator of type $\lambda,$ it has already the form
$T^{i}\left(  \xi_{0}\right)  $ required by the theorem, hence it is enough to
prove that%
\[
\widetilde{X}_{i}\int a\left(  \xi\right)  D\Gamma\left(  \xi_{0}%
;\Theta\left(  \eta,\xi\right)  \right)  b\left(  \eta\right)  f\left(
\eta\right)  d\eta
\]
can be rewritten in the form
\[
\sum_{j,k}T_{jk}^{i}\left(  \xi_{0}\right)  \widetilde{X}_{j,k}f\left(
\xi\right)  +T^{i}\left(  \xi_{0}\right)  f\left(  \xi\right)
\]
with $T_{jk}^{i}\left(  \xi_{0}\right)  ,T^{i}\left(  \xi_{0}\right)  $ frozen
operators of type $\lambda+j-1$ and $\lambda,$ respectively. Next, we have to
distinguish two cases.

Case 1: $\lambda\geq2.$ In this case the $\widetilde{X}_{i}$ derivative can be
taken under the integral sign, writing:
\begin{align*}
&  \widetilde{X}_{i}\int a\left(  \xi\right)  D\Gamma\left(  \xi_{0}%
;\Theta\left(  \eta,\xi\right)  \right)  b\left(  \eta\right)  f\left(
\eta\right)  d\eta\\
&  =\int\left(  \widetilde{X}_{i}a\right)  \left(  \xi\right)  D\Gamma\left(
\xi_{0};\Theta\left(  \eta,\xi\right)  \right)  b\left(  \eta\right)  f\left(
\eta\right)  d\eta\\
&  +\int a\left(  \xi\right)  \widetilde{X}_{i}\left[  D\Gamma\left(
\Theta\left(  \eta,\cdot\right)  \right)  \right]  \left(  \xi\right)
b\left(  \eta\right)  f\left(  \eta\right)  d\eta\\
&  =A\left(  \xi\right)  +B\left(  \xi\right)  .
\end{align*}
Now $A\left(  \xi\right)  $ is frozen operator of type $\lambda,$ while
applying Lemma \ref{lemma change differential order on domain} with $j_{0}=1$
we get%
\begin{align*}
B\left(  \xi\right)   &  =\int a\left(  \xi\right)  \sum_{l,r}a_{lr}%
^{i}\left(  \Theta\left(  \eta,\xi\right)  \right)  \widetilde{X}_{l,r}\left[
D\Gamma\left(  \xi_{0};\Theta\left(  \cdot,\xi\right)  \right)  \right]
\left(  \eta\right)  b\left(  \eta\right)  f\left(  \eta\right)  d\eta\\
&  +\int a\left(  \xi\right)  \sum_{l,r}a_{lr}^{i}\left(  \Theta\left(
\eta,\xi\right)  \right)  \left(  R_{rs}^{\xi}D\Gamma\right)  \left(  \xi
_{0};\Theta\left(  \eta,\xi\right)  \right)  b\left(  \eta\right)  f\left(
\eta\right)  d\eta\\
&  \equiv C\left(  \xi\right)  +D\left(  \xi\right)
\end{align*}
where $R_{rs}^{\xi}$ are differential operators of local degree $\leq0,$ and
the $a_{lr}^{i}$'s have local degree $\geq l-1.$ Hence $D$ is a frozen
operator of type $\lambda,$ while, since the transposed vector field of
$\widetilde{X}_{l,r}$ is%
\[
\widetilde{X}_{l,r}^{T}=-\widetilde{X}_{l,r}+c_{l,r}%
\]
with $c_{l,r}$ smooth functions,%
\begin{align*}
C\left(  \xi\right)   &  =-a\left(  \xi\right)  \sum_{l,r}\int\widetilde
{X}_{l,r}\left[  a_{lr}^{i}\left(  \Theta\left(  \cdot,\xi\right)  \right)
b\left(  \cdot\right)  \right]  \left(  \eta\right)  D\Gamma\left(  \xi
_{0};\Theta\left(  \eta,\xi\right)  \right)  f\left(  \eta\right)  d\eta\\
&  +a\left(  \xi\right)  \sum_{l,r}\int a_{lr}^{i}\left(  \Theta\left(
\eta,\xi\right)  \right)  D\Gamma\left(  \xi_{0};\Theta\left(  \eta
,\xi\right)  \right)  c_{l,r}\left(  \eta\right)  b\left(  \eta\right)
f\left(  \eta\right)  d\eta\\
&  -a\left(  \xi\right)  \sum_{l,r}\int a_{lr}^{i}\left(  \Theta\left(
\eta,\xi\right)  \right)  D\Gamma\left(  \xi_{0};\Theta\left(  \eta
,\xi\right)  \right)  b\left(  \eta\right)  \widetilde{X}_{l,r}f\left(
\eta\right)  d\eta.
\end{align*}
The first two terms in the last expression are still frozen operators of type
$\lambda$ applied to $f,$ while the third is a sum of operators of type
$\lambda+l-1$ applied to $\widetilde{X}_{l,r}f,$ as required by the theorem.

Case 2. $\lambda=1.$ In this case we have to compute the derivative of the
integral in distributional sense, as already done in the proof of Lemma
\ref{iilem:der:optype}: with the same meaning of $\varphi_{\varepsilon}$, let
us compute%
\[
\lim_{\varepsilon\rightarrow0}\widetilde{X}_{i}\int a\left(  \xi\right)
\varphi_{\varepsilon}\left(  \Theta\left(  \eta,\xi\right)  \right)
D\Gamma\left(  \xi_{0};\Theta\left(  \eta,\xi\right)  \right)  b\left(
\eta\right)  f\left(  \eta\right)  d\eta.
\]
Actually, this gives exactly the same result as in case 1:%
\begin{align*}
&  \widetilde{X}_{i}\int a\left(  \xi\right)  \varphi_{\varepsilon}\left(
\Theta\left(  \eta,\xi\right)  \right)  D\Gamma\left(  \xi_{0};\Theta\left(
\eta,\xi\right)  \right)  b\left(  \eta\right)  f\left(  \eta\right)  d\eta\\
&  =\int\left(  \widetilde{X}_{i}a\right)  \left(  \xi\right)  \varphi
_{\varepsilon}\left(  \Theta\left(  \eta,\xi\right)  \right)  D\Gamma\left(
\xi_{0};\Theta\left(  \eta,\xi\right)  \right)  b\left(  \eta\right)  f\left(
\eta\right)  d\eta\\
&  +\int a\left(  \xi\right)  \widetilde{X}_{i}\left[  \left(  \varphi
_{\varepsilon}D\Gamma\right)  \left(  \Theta\left(  \eta,\cdot\right)
\right)  \right]  \left(  \xi\right)  b\left(  \eta\right)  f\left(
\eta\right)  d\eta\\
&  =A_{\varepsilon}\left(  \xi\right)  +B_{\varepsilon}\left(  \xi\right)
\end{align*}
where
\[
A_{\varepsilon}\left(  \xi\right)  \rightarrow\int\left(  \widetilde{X}%
_{i}a\right)  \left(  \xi\right)  D\Gamma\left(  \xi_{0};\Theta\left(
\eta,\xi\right)  \right)  b\left(  \eta\right)  f\left(  \eta\right)  d\eta
\]
and%
\begin{align*}
B_{\varepsilon}\left(  \xi\right)   &  =\int a\left(  \xi\right)  \sum
_{l,r}a_{lr}^{i}\left(  \Theta\left(  \eta,\xi\right)  \right)  \widetilde
{X}_{l,r}\left[  \varphi_{\varepsilon}\left(  \Theta\left(  \cdot,\xi\right)
\right)  D\Gamma\left(  \xi_{0};\Theta\left(  \cdot,\xi\right)  \right)
\right]  \left(  \eta\right)  b\left(  \eta\right)  f\left(  \eta\right)
d\eta\\
&  +\int a\left(  \xi\right)  \sum_{l,r}a_{lr}^{i}\left(  \Theta\left(
\eta,\xi\right)  \right)  \left(  R_{rs}^{\xi}\left(  \varphi_{\varepsilon
}D\Gamma\right)  \right)  \left(  \xi_{0};\Theta\left(  \eta,\xi\right)
\right)  b\left(  \eta\right)  f\left(  \eta\right)  d\eta\\
&  \equiv C_{\varepsilon}\left(  \xi\right)  +D_{\varepsilon}\left(
\xi\right)
\end{align*}
where $C_{\varepsilon}\left(  \xi\right)  $ converges to the expression called
$C\left(  \xi\right)  $ in the computation of case 1; as to $D_{\varepsilon
}\left(  \xi\right)  $,
\[
R_{rs}^{\xi}\left(  \varphi_{\varepsilon}D\Gamma\right)  =\left(  R_{rs}^{\xi
}\varphi_{\varepsilon}\right)  D\Gamma+\varphi_{\varepsilon}R_{rs}^{\xi
}D\Gamma.
\]
Now, $\varphi_{\varepsilon}R_{rs}^{\xi}D\Gamma\rightarrow R_{rs}^{\xi}D\Gamma$
while $\left(  R_{rs}^{\xi}\varphi_{\varepsilon}\right)  D\Gamma\rightarrow0,$
being $R_{rs}^{\xi}$ a vector field of local degree $\leq0.$ Hence also
$D_{\varepsilon}\left(  \xi\right)  $ converges to the expression called
$D\left(  \xi\right)  $ in the computation of case 1, and we are done.

(ii) Let now $T\left(  \xi_{0}\right)  $ be a frozen operator of type
$\lambda\geq2$ with kernel $k^{\prime}$. As in the case (i), it is enough to
prove that
\[
\widetilde{X}_{0}\int a\left(  \xi\right)  D\Gamma\left(  \xi_{0}%
;\Theta\left(  \eta,\xi\right)  \right)  b\left(  \eta\right)  f\left(
\eta\right)  d\eta,
\]
where $D\Gamma$ is homogeneous of degree $\lambda-Q,$ can be rewritten in the
form
\[
\sum_{j,k}T_{jk}^{0}\left(  \xi_{0}\right)  \widetilde{X}_{j,k}f\left(
\xi\right)  +T^{0}\left(  \xi_{0}\right)  f\left(  \xi\right)
\]
with $T_{jk}^{0}\left(  \xi_{0}\right)  ,T^{0}\left(  \xi_{0}\right)  $ frozen
operators of type $\lambda+j-2$ and $\lambda-1,$ respectively. Let us consider
only the case $\lambda\geq3,$ the case $\lambda=2$ being handled with the
modification seen in (i), Case 2.%
\begin{align*}
&  \widetilde{X}_{0}\int a\left(  \xi\right)  D\Gamma\left(  \xi_{0}%
;\Theta\left(  \eta,\xi\right)  \right)  b\left(  \eta\right)  f\left(
\eta\right)  d\eta\\
&  =\int\left(  \widetilde{X}_{0}a\right)  \left(  \xi\right)  D\Gamma\left(
\xi_{0};\Theta\left(  \eta,\xi\right)  \right)  b\left(  \eta\right)  f\left(
\eta\right)  d\eta+\\
&  +\int a\left(  \xi\right)  \widetilde{X}_{0}\left[  D\Gamma\left(  \xi
_{0};\Theta\left(  \eta,\cdot\right)  \right)  \right]  \left(  \xi\right)
b\left(  \eta\right)  f\left(  \eta\right)  d\eta\\
&  =\int\left(  \widetilde{X}_{0}a\right)  \left(  \xi\right)  D\Gamma\left(
\xi_{0};\Theta\left(  \eta,\xi\right)  \right)  b\left(  \eta\right)  f\left(
\eta\right)  d\eta+\\
&  +\int a\left(  \xi\right)  \sum_{l,r}a_{lr}^{0}\left(  \Theta\left(
\eta,\xi\right)  \right)  \widetilde{X}_{l,r}\left[  D\Gamma\left(  \xi
_{0};\Theta\left(  \cdot,\xi\right)  \right)  \right]  \left(  \eta\right)
b\left(  \eta\right)  f\left(  \eta\right)  d\eta\\
&  +\int a\left(  \xi\right)  \sum_{l,r}a_{lr}^{0}\left(  \Theta\left(
\eta,\xi\right)  \right)  \left(  R_{rs}^{\xi}D\Gamma\right)  \left(  \xi
_{0};\Theta\left(  \eta,\xi\right)  \right)  b\left(  \eta\right)  f\left(
\eta\right)  d\eta\\
&  \equiv A\left(  \xi\right)  +C\left(  \xi\right)  +D\left(  \xi\right)  .
\end{align*}
where $R_{rs}^{\xi}$ are now differential operators of local degree $\leq1,$
and the $a_{lr}^{0}$'s have local degree $\geq\max\left\{  j-2,0\right\}  .$
Then $A\left(  \xi\right)  $ is a frozen operator of type $\lambda$, applied
to $f;$ $D\left(  \xi\right)  $ is a frozen operator of type $\lambda-1,$
applied to $f$. Moreover,%
\begin{align*}
C\left(  \xi\right)   &  =-a\left(  \xi\right)  \sum_{l,r}\int\widetilde
{X}_{l,r}\left[  a_{lr}^{0}\left(  \Theta\left(  \cdot,\xi\right)  \right)
b\left(  \cdot\right)  \right]  \left(  \eta\right)  D\Gamma\left(  \xi
_{0};\Theta\left(  \eta,\xi\right)  \right)  f\left(  \eta\right)  d\eta\\
&  +a\left(  \xi\right)  \sum_{l,r}\int a_{lr}^{0}\left(  \Theta\left(
\eta,\xi\right)  \right)  D\Gamma\left(  \xi_{0};\Theta\left(  \eta
,\xi\right)  \right)  c_{l,r}\left(  \eta\right)  b\left(  \eta\right)
f\left(  \eta\right)  d\eta\\
&  -a\left(  \xi\right)  \sum_{l,r}\int a_{lr}^{0}\left(  \Theta\left(
\eta,\xi\right)  \right)  D\Gamma\left(  \xi_{0};\Theta\left(  \eta
,\xi\right)  \right)  b\left(  \eta\right)  \widetilde{X}_{l,r}f\left(
\eta\right)  d\eta
\end{align*}
where the first two terms are still frozen operators of type $\lambda,$
applied to $f$, while the third is the sum of frozen operators of type
$\lambda+\max\left\{  j-2,0\right\}  $ applied to $\widetilde{X}_{l,r}f.$
\end{proof}

We can now proceed to the:

\begin{proof}
[Proof of Theorem \ref{scambioderivoper}]It suffices to prove the formula
(\ref{commute of operator}), for the second is similar. Let us consider one of
the terms $T_{jk}^{i}\left(  \xi_{0}\right)  \widetilde{X}_{j,k}$ appearing in
(\ref{1}).

If $j=1$, the term is already in the form required by the Theorem.

If $j=2$, then $\widetilde{X}_{2,k}$ can be written as a combination of
commutators of the vector fields $\widetilde{X}_{1},\widetilde{X}%
_{2},...,\widetilde{X}_{q},$ plus (possibly) the field $\widetilde{X}_{0}$.
Then $T_{2k}^{i}\left(  \xi_{0}\right)  \widetilde{X}_{2,k}$ contains terms
$T_{2k}^{i}\left(  \xi_{0}\right)  \widetilde{X}_{h}\widetilde{X}_{j}$ and
possibly a term $T_{2k}^{i}\left(  \xi_{0}\right)  \widetilde{X}_{0}$. By the
above theorem we know $T_{2k}^{i}$ is a frozen operator of type $\lambda+1$.
Now:%
\[
T_{2k}^{i}\left(  \xi_{0}\right)  \widetilde{X}_{h}\widetilde{X}_{j}=\left(
T_{2k}^{i}\left(  \xi_{0}\right)  \widetilde{X}_{h}\right)  \widetilde{X}%
_{j}=T_{k}^{i}\left(  \xi_{0}\right)  \widetilde{X}_{j},
\]
where by Theorem \ref{main theorem}, $T_{k}^{i}\left(  \xi_{0}\right)  $ is a
frozen operator of type $\lambda$; on the other hand, by
(\ref{frozen operator}),%
\begin{align*}
T_{2k}^{i}\left(  \xi_{0}\right)  \widetilde{X}_{0}  &  =T_{2k}^{i}\left(
\xi_{0}\right)  \left(  \widetilde{\mathcal{L}}_{0}-\sum_{h,j=1}^{q}%
\widetilde{a}_{hj}\left(  \xi_{0}\right)  \widetilde{X}_{h}\widetilde{X}%
_{j}\right) \\
&  =T_{2k}^{i}\left(  \xi_{0}\right)  \widetilde{\mathcal{L}}_{0}-\sum
_{h,j=1}^{q}\widetilde{a}_{hj}\left(  \xi_{0}\right)  \left(  T_{2k}%
^{i}\left(  \xi_{0}\right)  \widetilde{X}_{h}\right)  \widetilde{X}_{j}\\
&  =T_{2k}^{i}\left(  \xi_{0}\right)  \widetilde{\mathcal{L}}_{0}-\sum
_{h,j=1}^{q}\widetilde{a}_{hj}\left(  \xi_{0}\right)  T_{h,k}^{i}\left(
\xi_{0}\right)  \widetilde{X}_{j},
\end{align*}
with $T_{2k}^{i}\left(  \xi_{0}\right)  ,$ $T_{h,k}^{i}\left(  \xi_{0}\right)
,$ frozen operators of type $\lambda+1,\lambda,$ which is in the form allowed
by the thesis of the Theorem.

Finally, if $j>2$, it is enough to look at the final part of the differential
operator $\widetilde{X}_{j,k}$: it is always possible to rewrite
$\widetilde{X}_{j,k}$ either as $\widetilde{X}_{j-1,k}\widetilde{X}_{1,k}$ or
as $\widetilde{X}_{j-2,k}\widetilde{X}_{2,k}$. In the first case, we have%
\[
T_{jk}^{i}\left(  \xi_{0}\right)  \widetilde{X}_{j,k}=\left(  T_{jk}%
^{i}\left(  \xi_{0}\right)  \widetilde{X}_{j-1,k}\right)  \widetilde{X}%
_{1,k}=T_{jk}^{\prime i}\left(  \xi_{0}\right)  \widetilde{X}_{1,k},
\]
with $T_{jk}^{\prime i}\left(  \xi_{0}\right)  $ frozen operator of type
$\lambda,$ which is already in the proper form; in the second case, we have%
\[
T_{jk}^{i}\left(  \xi_{0}\right)  \widetilde{X}_{j,k}=\left(  T_{jk}%
^{i}\left(  \xi_{0}\right)  \widetilde{X}_{j-2,k}\right)  \widetilde{X}%
_{2,k}=T_{j}^{^{\prime}i}\left(  \xi_{0}\right)  \widetilde{X}_{2,k}%
\]
with $T_{jk}^{\prime i}\left(  \xi_{0}\right)  $ frozen operator of type
$\lambda+1,$ and then we can proceed as in the case $j=2$. So the Theorem is proved.
\end{proof}

\subsection{Parametrix and representation formulas\label{subsec parametrix}}

Throughout this subsection we will make extensive use of computations on
frozen operators of type $\lambda.$ To make more readable our formulas, we
will use the symbols%
\[
T\left(  \xi_{0}\right)  ,S\left(  \xi_{0}\right)  ,P\left(  \xi_{0}\right)
\]
(with possibly other indexes) to denote frozen operators of type $0,1,2$, respectively.

In order to prove representation formulas for second order derivatives, we
start with the following parametrix identities, analogous to \cite[Thm.
10]{rs}, \cite[Thm. 3.1]{bb2}.

\begin{theorem}
\label{parametrix}Given $a\in C_{0}^{\infty}\left(  \widetilde{B}\left(
\overline{\xi},R\right)  \right)  $, there exist $S_{ij}(\xi_{0})$,
$S_{0}\left(  \xi_{0}\right)  $, $S_{ij}^{\ast}(\xi_{0})$, $S_{0}^{\ast
}\left(  \xi_{0}\right)  ,$ frozen operators of type $1$ and $P(\xi
_{0}),P^{\ast}(\xi_{0}),$ frozen operators of type $2$ (over the ball
$\widetilde{B}\left(  \overline{\xi},R\right)  $) such that:%
\begin{align}
aI  &  =\widetilde{\mathcal{L}}_{0}^{T}P^{\ast}\left(  \xi_{0}\right)
+\sum_{i,j=1}^{q}\widetilde{a}_{ij}\left(  \xi_{0}\right)  S_{ij}^{\ast
}\left(  \xi_{0}\right)  +S_{0}^{\ast}\left(  \xi_{0}\right)  ;\nonumber\\
aI  &  =P(\xi_{0})\widetilde{\mathcal{L}}_{0}+\sum_{i,j=1}^{q}\widetilde
{a}_{ij}\left(  \xi_{0}\right)  S_{ij}\left(  \xi_{0}\right)  +S_{0}\left(
\xi_{0}\right)  \label{representation2}%
\end{align}
where $I$ denotes the identity. Moreover, $S_{ij}^{\ast}\left(  \xi
_{0}\right)  ,S_{0}^{\ast}\left(  \xi_{0}\right)  ,P^{\ast}\left(  \xi
_{0}\right)  $ are modeled on $\Gamma^{T},$ while $S_{ij}\left(  \xi
_{0}\right)  ,S_{0}\left(  \xi_{0}\right)  ,P\left(  \xi_{0}\right)  $ are
modeled on $\Gamma.$ Explicitly,%
\begin{align*}
P^{\ast}\left(  \xi_{0}\right)  f\left(  \xi\right)   &  =-\frac{a(\xi)}%
{c(\xi)}\int_{\widetilde{B}}\Gamma^{T}(\xi_{0};\Theta(\eta,\xi))\,b(\eta
)\,f(\eta)\,d\eta\\
P\left(  \xi_{0}\right)  f\left(  \xi\right)   &  =-b\left(  \xi\right)
\int_{\widetilde{B}}\frac{a(\eta)}{c(\eta)}\Gamma(\xi_{0};\Theta(\eta
,\xi))\,f(\eta)\,d\eta
\end{align*}
where $c$ is the function appearing in Theorem \ref{metric} (c).
\end{theorem}

The proof of this result is similar to that of \cite{rs}, \cite[Thm. 3.1]%
{bb2}. However, we will write a detailed version.

\begin{proof}
Let us define%
\[
P^{\ast}(\xi_{0})f(\xi)=\frac{a(\xi)}{c(\xi)}\int_{\widetilde{B}}\Gamma
^{T}(\xi_{0};\Theta(\eta,\xi))\,b(\eta)\,f(\eta)\,d\eta
\]
where $a,b\in C_{0}^{\infty}\left(  \widetilde{B}\left(  \overline{\xi
},R\right)  \right)  $ such that $ab=a,$ and $c\left(  \xi\right)  $ is the
function appearing in the formula of change of variables
(\ref{change of variables}), and let us compute $\widetilde{\mathcal{L}}%
_{0}^{T}P^{\ast}(\xi_{0})f,$ for $f\in C_{0}^{\infty}\left(  \widetilde
{B}\left(  \overline{\xi},R\right)  \right)  $. We have to apply a
distributional argument like in the proof of Lemma \ref{iilem:der:optype}: for
$\omega\in C_{0}^{\infty}\left(  \widetilde{B}\left(  \overline{\xi},R\right)
\right)  ,$ let us evaluate%
\[
\int_{\widetilde{B}}\widetilde{\mathcal{L}}_{0}\omega\left(  \xi\right)
P^{\ast}(\xi_{0})f\left(  \xi\right)  d\xi=\lim_{\varepsilon\rightarrow0}%
\int_{\widetilde{B}}\widetilde{\mathcal{L}}_{0}\omega\left(  \xi\right)
P_{\varepsilon}^{\ast}(\xi_{0})f\left(  \xi\right)  d\xi
\]
where%
\[
P_{\varepsilon}^{\ast}(\xi_{0})f(\xi)=\frac{a(\xi)}{c(\xi)}\int_{\widetilde
{B}}\varphi_{\varepsilon}\left(  \Theta(\eta,\xi)\right)  \Gamma^{T}(\xi
_{0};\Theta(\eta,\xi))\,b(\eta)\,f(\eta)\,d\eta
\]
with $\varphi_{\varepsilon}$ as in the proof of the quoted Lemma.%
\begin{align*}
&  \int_{\widetilde{B}}\widetilde{\mathcal{L}}_{0}\omega\left(  \xi\right)
P_{\varepsilon}^{\ast}(\xi_{0})f\left(  \xi\right)  d\xi\\
&  =\int_{\widetilde{B}}\,bf(\eta)\left(  \int_{\widetilde{B}}\frac{a(\xi
)}{c(\xi)}\varphi_{\varepsilon}\left(  \Theta(\eta,\xi)\right)  \Gamma^{T}%
(\xi_{0};\Theta(\eta,\xi))\widetilde{\mathcal{L}}_{0}\omega\left(  \xi\right)
d\xi\right)  d\eta\\
&  =\int_{\widetilde{B}}\,bf(\eta)\left(  \int_{\widetilde{B}}\widetilde
{\mathcal{L}}_{0}^{T}\left(  \frac{a(\xi)}{c(\xi)}\right)  \varphi
_{\varepsilon}\left(  \Theta(\eta,\xi)\right)  \Gamma^{T}(\xi_{0};\Theta
(\eta,\xi))\omega\left(  \xi\right)  d\xi\right)  d\eta\\
&  +\int_{\widetilde{B}}\,bf(\eta)\left(  \int_{\widetilde{B}}\frac{a(\xi
)}{c(\xi)}\widetilde{\mathcal{L}}_{0}^{T}\left[  \varphi_{\varepsilon}\left(
\Theta(\eta,\cdot)\right)  \Gamma^{T}(\xi_{0};\Theta(\eta,\cdot))\right]
\left(  \xi\right)  \omega\left(  \xi\right)  d\xi\right)  d\eta\\
&  +\int_{\widetilde{B}}\,bf(\eta)\left(  \int_{\widetilde{B}}2\sum
_{h,k=1}^{q}\widetilde{a}_{hk}\left(  \xi_{0}\right)  \widetilde{X}_{h}%
^{T}\left(  \frac{a}{c}\right)  (\xi)\widetilde{X}_{k}^{T}\left[
\varphi_{\varepsilon}\left(  \Theta(\eta,\cdot)\right)  \Gamma^{T}(\xi
_{0};\Theta(\eta,\cdot))\right]  \left(  \xi\right)  \omega\left(  \xi\right)
d\xi\right)  d\eta\\
&  \equiv A_{\varepsilon}+B_{\varepsilon}+C_{\varepsilon}.
\end{align*}%
\begin{align*}
A_{\varepsilon}  &  \rightarrow\int_{\widetilde{B}}\,bf(\eta)\left(
\int_{\widetilde{B}}\widetilde{\mathcal{L}}_{0}^{T}\left(  \frac{a(\xi)}%
{c(\xi)}\right)  \Gamma^{T}(\xi_{0};\Theta(\eta,\xi))\omega\left(  \xi\right)
d\xi\right)  d\eta\\
&  =\sum_{h,k=1}^{q}\widetilde{a}_{hk}\left(  \xi_{0}\right)  \int
_{\widetilde{B}}\,f(\eta)P_{hk}^{T}\left(  \xi_{0}\right)  \omega\left(
\eta\right)  d\eta+\int_{\widetilde{B}}\,f(\eta)P_{0}^{T}\left(  \xi
_{0}\right)  \omega\left(  \eta\right)  d\eta\\
&  =\sum_{h,k=1}^{q}\widetilde{a}_{hk}\left(  \xi_{0}\right)  \int
_{\widetilde{B}}\,\omega\left(  \xi\right)  P_{hk}\left(  \xi_{0}\right)
f(\xi)d\xi+\int_{\widetilde{B}}\,\omega\left(  \xi\right)  P_{0}\left(
\xi_{0}\right)  f(\xi)d\xi
\end{align*}
where%
\begin{align*}
P_{hk}\left(  \xi_{0}\right)  f(\xi)  &  =\widetilde{X}_{h}^{T}\widetilde
{X}_{k}^{T}\left(  \frac{a}{c}\right)  (\xi)\left(  \int_{\widetilde{B}}%
\Gamma^{T}(\xi_{0};\Theta(\eta,\xi))b\left(  \eta\right)  f(\eta)d\eta\right)
d\xi\\
P_{0}\left(  \xi_{0}\right)  f(\xi)  &  =\widetilde{X}_{0}^{T}\left(  \frac
{a}{c}\right)  (\xi)\left(  \int_{\widetilde{B}}\Gamma^{T}(\xi_{0};\Theta
(\eta,\xi))b\left(  \eta\right)  f(\eta)d\eta\right)  d\xi
\end{align*}
are frozen operators of type 2, modeled on $\Gamma^{T}$.%
\begin{align*}
B_{\varepsilon}  &  =\int_{\widetilde{B}}\,bf(\eta)\left\{  \int
_{\widetilde{B}}\frac{a(\xi)}{c(\xi)}\left[  \mathcal{L}_{0}^{\ast T}\left(
\varphi_{\varepsilon}\Gamma^{T}(\xi_{0};\cdot)\right)  \right]  \left(
\Theta(\eta,\xi)\right)  \omega\left(  \xi\right)  d\xi\right. \\
&  +\left.  \left(  \int_{\widetilde{B}}\frac{a(\xi)}{c(\xi)}\left(
\sum_{i,j}\widetilde{a}_{ij}\left(  \xi_{0}\right)  \left[  Y_{i}R_{j}^{\eta
}+R_{i}^{\eta}Y_{j}+R_{i}^{\eta}R_{j}^{\eta}\right]  +R_{0}^{\eta}\right)
\left(  \varphi_{\varepsilon}\Gamma^{T}(\xi_{0};\cdot)\right)  \left(
\Theta(\eta,\xi)\right)  \omega\left(  \xi\right)  d\xi\right)  \right\}
d\eta\\
&  =\int_{\widetilde{B}}\,bf(\eta)\left(  \int_{\left\Vert u\right\Vert
<R}\mathcal{L}_{0}^{\ast T}\left(  \varphi_{\varepsilon}\left(  u\right)
\Gamma^{T}(\xi_{0};u)\right)  \left(  a\omega\right)  \left(  \Theta
(\eta,\cdot)^{-1}\left(  u\right)  \right)  \left(  1+O\left(  \left\Vert
u\right\Vert \right)  \right)  du\right)  d\eta\\
&  +\int_{\widetilde{B}}\,bf(\eta)\int_{\left\Vert u\right\Vert <R}\left(
\sum_{i,j}\widetilde{a}_{ij}\left(  \xi_{0}\right)  \left[  Y_{i}R_{j}^{\eta
}+R_{i}^{\eta}Y_{j}+R_{i}^{\eta}R_{j}^{\eta}\right]  +R_{0}^{\eta}\right)
\left(  \varphi_{\varepsilon}\left(  u\right)  \Gamma^{T}(\xi_{0};u)\right)
\cdot\\
&  \cdot\left(  a\omega\right)  \left(  \Theta(\eta,\cdot)^{-1}\left(
u\right)  \right)  \left(  1+O\left(  \left\Vert u\right\Vert \right)
\right)  dud\eta\\
&  =D_{\varepsilon}+E_{\varepsilon}.
\end{align*}
To study $D_{\varepsilon},$ let $\left(  a\omega\right)  _{\eta}\left(
u\right)  \equiv\left(  a\omega\right)  \left(  \Theta(\eta,\cdot)^{-1}\left(
u\right)  \right)  .$ Then%
\begin{align*}
D_{\varepsilon}  &  =\int_{\widetilde{B}}\,bf(\eta)\int_{\left\Vert
u\right\Vert <R}\varphi_{\varepsilon}\left(  u\right)  \Gamma^{T}(\xi
_{0};u)\mathcal{L}_{0}^{\ast}\left(  a\omega\right)  _{\eta}\left(  u\right)
dud\eta+\\
&  +\int_{\widetilde{B}}\,bf(\eta)\int_{\left\Vert u\right\Vert <R}%
\mathcal{L}_{0}^{\ast T}\left[  \varphi_{\varepsilon}\Gamma^{T}(\xi_{0}%
;\cdot)\right]  \left(  u\right)  \left(  a\omega\right)  _{\eta}\left(
u\right)  O\left(  \left\Vert u\right\Vert \right)  dud\eta\\
&  \equiv D_{\varepsilon}^{1}+D_{\varepsilon}^{2}.
\end{align*}%
\begin{align*}
D_{\varepsilon}^{1}  &  \rightarrow\int_{\widetilde{B}}\,bf(\eta
)\int_{\left\Vert u\right\Vert <R}\Gamma^{T}(\xi_{0};u)\mathcal{L}_{0}^{\ast
}\left(  a\omega\right)  _{\eta}\left(  u\right)  dud\eta=-\int_{\widetilde
{B}}\,bf(\eta)\left(  a\omega\right)  _{\eta}\left(  0\right)  d\eta\\
&  =-\int_{\widetilde{B}}\,bf(\eta)\left(  a\omega\right)  \left(
\eta\right)  d\eta=-\int_{\widetilde{B}}\,\left(  af\omega\right)  (\eta
)d\eta.
\end{align*}%
\begin{align*}
D_{\varepsilon}^{2}  &  =\int_{\widetilde{B}}\,bf(\eta)\int_{\left\Vert
u\right\Vert <R}\left(  \mathcal{L}_{0}^{\ast T}\varphi_{\varepsilon}\right)
\left(  u\right)  \Gamma^{T}(\xi_{0};u)\left(  a\omega\right)  _{\eta}\left(
u\right)  O\left(  \left\Vert u\right\Vert \right)  dud\eta\\
&  +\int_{\widetilde{B}}\,bf(\eta)\int_{\left\Vert u\right\Vert <R}2\sum
_{i,j}\widetilde{a}_{ij}\left(  \xi_{0}\right)  \left(  Y_{i}\varphi
_{\varepsilon}\right)  \left(  u\right)  Y_{j}\Gamma^{T}(\xi_{0};u)\left(
a\omega\right)  _{\eta}\left(  u\right)  O\left(  \left\Vert u\right\Vert
\right)  dud\eta.
\end{align*}
A dilation argument as in the proof of Lemma \ref{iilem:der:optype} then gives%
\[
D_{\varepsilon}^{2}\rightarrow0.
\]
Moreover,%
\begin{align*}
E_{\varepsilon}  &  \rightarrow\int_{\widetilde{B}}\,bf(\eta)\int_{\left\Vert
u\right\Vert <R}\left(  \sum_{i,j}\widetilde{a}_{ij}\left(  \xi_{0}\right)
\left[  Y_{i}R_{j}^{\eta}+R_{i}^{\eta}Y_{j}+R_{i}^{\eta}R_{j}^{\eta}\right]
+R_{0}^{\eta}\right)  \Gamma^{T}(\xi_{0};u)\cdot\\
&  \cdot\left(  a\omega\right)  \left(  \Theta(\eta,\cdot)^{-1}\left(
u\right)  \right)  \left(  1+O\left(  \left\Vert u\right\Vert \right)
\right)  dud\eta
\end{align*}
coming back to the original variables $\xi$ in the inner integral%
\begin{align*}
&  =\int_{\widetilde{B}}\,bf(\eta)\int_{\widetilde{B}}\left(  \sum
_{i,j}\widetilde{a}_{ij}\left(  \xi_{0}\right)  \left[  Y_{i}R_{j}^{\eta
}+R_{i}^{\eta}Y_{j}+R_{i}^{\eta}R_{j}^{\eta}\right]  \Gamma^{T}(\xi_{0}%
;\cdot)+\right. \\
&  \left.  +R_{0}^{\eta}\Gamma^{T}(\xi_{0};\cdot)\left(  \Theta(\eta
,\xi)\right)  \frac{\left(  a\omega\right)  \left(  \xi\right)  }{c\left(
\xi\right)  }\right)  d\xi d\eta\\
&  =\int_{\widetilde{B}}\,f(\eta)\left[  \sum_{i,j}\widetilde{a}_{ij}\left(
\xi_{0}\right)  S_{ij}^{\prime}\left(  \xi_{0}\right)  +S_{0}^{\prime}\left(
\xi_{0}\right)  \right]  \omega\left(  \eta\right)  d\eta\\
&  =\int_{\widetilde{B}}\,\omega\left(  \eta\right)  \left[  \sum
_{i,j}\widetilde{a}_{ij}\left(  \xi_{0}\right)  S_{ij}^{\prime T}\left(
\xi_{0}\right)  +S_{0}^{\prime T}\left(  \xi_{0}\right)  \right]  f(\eta)d\eta
\end{align*}
where $S_{ij}^{\prime}\left(  \xi_{0}\right)  ,S_{0}^{\prime}\left(  \xi
_{0}\right)  $ are transposed of frozen operators of type one modeled on
$\Gamma^{T},$ hence are frozen operators of type one, modeled on $\Gamma$ (see
Proposition \ref{Prop transposed operators}); therefore $S_{ij}^{\prime
T}\left(  \xi_{0}\right)  ,S_{0}^{\prime T}\left(  \xi_{0}\right)  $ are
frozen operators of type one, modeled on $\Gamma^{T}$. Analogously, one can
check that%
\[
C_{\varepsilon}\rightarrow\sum_{h,k=1}^{q}\widetilde{a}_{hk}\left(  \xi
_{0}\right)  \int_{\widetilde{B}}\omega\left(  \xi\right)  S_{hk}%
^{^{\prime\prime}}\left(  \xi_{0}\right)  \,f(\xi)d\xi
\]
where $S_{hk}^{^{\prime\prime}}\left(  \xi_{0}\right)  $ are frozen operators
of type one, modeled on $\Gamma^{T}$.

Hence we have proved that%
\begin{align*}
\widetilde{\mathcal{L}}_{0}^{T}P^{\ast}(\xi_{0})f  &  =P_{1}\left(  \xi
_{0}\right)  f-af+\left[  \sum_{i,j}\widetilde{a}_{ij}\left(  \xi_{0}\right)
S_{ij}^{\prime T}\left(  \xi_{0}\right)  +S_{0}^{\prime T}\left(  \xi
_{0}\right)  \right]  f\\
&  =-af+\left[  \sum_{i,j}\widetilde{a}_{ij}\left(  \xi_{0}\right)
S_{ij}^{\ast}\left(  \xi_{0}\right)  +S_{0}^{\ast}\left(  \xi_{0}\right)
\right]  f
\end{align*}
since $S_{0}^{\prime T}\left(  \xi_{0}\right)  +P_{1}\left(  \xi_{0}\right)  $
is a frozen operator of type 1, and simplifying our notation with $S_{ij}$ in
place of $S_{ij}^{\prime T}.$ Note that $S_{ij}^{\ast}\left(  \xi_{0}\right)
,S_{0}^{\ast}\left(  \xi_{0}\right)  $ are frozen operators of type 1, modeled
on $\Gamma^{T}.$ This proves the first identity in the statement of the
theorem, apart from an immaterial change of sign in the definition of
$P^{\ast}(\xi_{0}).$

Next, let us transpose this identity, getting%
\[
P^{\ast T}(\xi_{0})\widetilde{\mathcal{L}}_{0}f\left(  \xi\right)  =\left(
\sum_{ij}\widetilde{a}_{ij}\left(  \xi_{0}\right)  S_{ij}^{\ast T}\left(
\xi_{0}\right)  +S_{0}^{\ast T}\left(  \xi_{0}\right)  \right)  f\left(
\xi\right)  -\left(  af\right)  \left(  \xi\right)  .
\]
Note that%
\begin{align*}
P^{\ast T}(\xi_{0})f\left(  \xi\right)   &  =b(\xi)\int_{\widetilde{B}}%
\frac{a(\eta)}{c(\eta)}\Gamma^{T}(\xi_{0};\Theta(\xi,\eta))\,f(\eta)\,d\eta\\
&  =b\left(  \xi\right)  \int_{\widetilde{B}}\frac{a(\eta)}{c(\eta)}\Gamma
(\xi_{0};\Theta(\eta,\xi))\,f(\eta)\,d\eta,
\end{align*}
which is a frozen operator of type two, modeled on $\Gamma$. On the other
hand, $S_{ij}^{\ast T}\left(  \xi_{0}\right)  ,S_{0}^{\ast T}\left(  \xi
_{0}\right)  $ are transposed of frozen operators of type 1 modeled on
$\Gamma^{T},$ therefore are frozen operators of type 1, modeled on $\Gamma.$
This concludes the proof.
\end{proof}

\begin{theorem}
[Representation of $\widetilde{X}_{m}\widetilde{X}_{l}u$ by frozen
operators]Given $a\in C_{0}^{\infty}\left(  \widetilde{B}\left(  \overline
{\xi},R\right)  \right)  ,\xi_{0}\in\widetilde{B}\left(  \overline{\xi
},R\right)  ,$ for any $m,l=1,2,...,q$, there exist frozen operators over the
ball $\widetilde{B}\left(  \overline{\xi},R\right)  ,$ such that for any $u\in
C_{0}^{\infty}\left(  \widetilde{B}\left(  \overline{\xi},R\right)  \right)  $%
\begin{align}
\widetilde{X}_{m}\widetilde{X}_{l}\left(  au\right)   &  =T_{lm}\left(
\xi_{0}\right)  \widetilde{\mathcal{L}}_{0}u+\sum_{k=1}^{q}T_{lm,k}\left(
\xi_{0}\right)  \widetilde{X}_{k}u+T_{lm}^{0}\left(  \xi_{0}\right)
u\nonumber\\
&  +\sum_{i,j=1}^{q}\widetilde{a}_{ij}\left(  \xi_{0}\right)  \left\{
\sum_{k=1}^{q}T_{lm,k}^{ij}\left(  \xi_{0}\right)  \widetilde{X}_{k}%
u+\sum_{h,k=1}^{q}\widetilde{a}_{hk}\left(  \xi_{0}\right)  T_{lm,h}^{\prime
ij}\left(  \xi_{0}\right)  \widetilde{X}_{k}u+\right.
\label{Representation for second derivatives 1}\\
&  +\left.  \frac{{}}{{}}S_{lm}^{ij}\left(  \xi_{0}\right)  \widetilde
{\mathcal{L}}_{0}u+T_{lm}^{ij}\left(  \xi_{0}\right)  u\right\} \nonumber
\end{align}
(All the $T_{...}\left(  \xi_{0}\right)  $ are frozen operators of type 0,
$S_{lm}^{ij}\left(  \xi_{0}\right)  $ are of type 1). Also,
\begin{align}
&  \widetilde{X}_{m}\widetilde{X}_{l}\left(  au\right)  =T_{lm}\left(  \xi
_{0}\right)  \widetilde{\mathcal{L}}u+T_{lm}\left(  \xi_{0}\right)  \left(
\sum_{i,j=1}^{q}\left[  \widetilde{a}_{ij}(\xi_{0})-\widetilde{a}_{ij}%
(\cdot)\right]  \,\widetilde{X}_{i}\widetilde{X}_{j}u\right)  +\nonumber\\
&  +\sum_{k=1}^{q}T_{lm,k}\left(  \xi_{0}\right)  \widetilde{X}_{k}%
u+T_{lm}^{0}\left(  \xi_{0}\right)  u+\nonumber\\
&  +\sum_{i,j=1}^{q}\widetilde{a}_{ij}\left(  \xi_{0}\right)  \left\{
\sum_{k=1}^{q}T_{lm,k}^{ij}\left(  \xi_{0}\right)  \widetilde{X}_{k}%
u+\sum_{h,k=1}^{q}\widetilde{a}_{hk}\left(  \xi_{0}\right)  T_{lm,h}^{\prime
ij}\left(  \xi_{0}\right)  \widetilde{X}_{k}u+S_{lm}^{ij}\left(  \xi
_{0}\right)  \widetilde{\mathcal{L}}u\right. \nonumber\\
&  \left.  +S_{lm}^{ij}\left(  \xi_{0}\right)  \left(  \sum_{i,j=1}^{q}\left[
\widetilde{a}_{ij}(\xi_{0})-\widetilde{a}_{ij}(\cdot)\right]  \,\widetilde
{X}_{i}\widetilde{X}_{j}u\right)  +T_{lm}^{ij}\left(  \xi_{0}\right)
u\right\}  . \label{Rep formula Schauder}%
\end{align}

\end{theorem}

\begin{remark}
The representation formulas of the above theorem have a cumbersome aspect, due
to the presence of the coefficients $\widetilde{a}_{ij}(\xi_{0})$ which appear
several times as multiplicative factors. Anyway, if we agree to leave
implicitly understood in the symbol of frozen operators the possible
multiplication by the coefficients $\widetilde{a}_{ij}$, our formulas assume
the following more compact form:%
\[
\widetilde{X}_{m}\widetilde{X}_{l}\left(  au\right)  =T_{lm}\left(  \xi
_{0}\right)  \widetilde{\mathcal{L}}_{0}u+\sum_{k=1}^{q}T_{k}^{lm}\left(
\xi_{0}\right)  \widetilde{X}_{k}u+T_{lm}^{0}\left(  \xi_{0}\right)  u
\]
and%
\begin{align*}
\widetilde{X}_{m}\widetilde{X}_{l}\left(  au\right)   &  =T_{lm}\left(
\xi_{0}\right)  \widetilde{\mathcal{L}}u+T_{lm}\left(  \xi_{0}\right)  \left(
\sum_{i,j=1}^{q}\left[  \widetilde{a}_{ij}(\xi_{0})-\widetilde{a}_{ij}%
(\cdot)\right]  \,\widetilde{X}_{i}\widetilde{X}_{j}u\right)  +\\
&  +\sum_{k=1}^{q}T_{k}^{lm}\left(  \xi_{0}\right)  \widetilde{X}_{k}%
u+T_{lm}^{0}\left(  \xi_{0}\right)  u.
\end{align*}
In the proof of \textit{a priori} estimates, when we will take $C^{\alpha}$ or
$L^{p}$ norms of both sides of these identities, the multiplicative factors
$\widetilde{a}_{hj}$ will be simply bounded by taking, respectively, the
$C^{\alpha}$ or the $L^{\infty}$ norms of the $\widetilde{a}_{hj}$; hence
leaving these factors implicitly understood is harmless.
\end{remark}

\begin{proof}
For $u\in C_{0}^{\infty}\left(  \widetilde{B}\left(  \overline{\xi},R\right)
\right)  $, let us start with the identity%
\[
au=P\left(  \xi_{0}\right)  \widetilde{\mathcal{L}}_{0}u+\sum_{i,j=1}%
^{q}\widetilde{a}_{ij}\left(  \xi_{0}\right)  S_{ij}\left(  \xi_{0}\right)
u+S_{0}\left(  \xi_{0}\right)  u
\]
\linebreak(see Theorem \ref{parametrix}); taking one derivative $\widetilde
{X}_{l}$ $\left(  l=1,2,...,q\right)  $ and applying Theorem
\ref{main theorem} and Theorem \ref{scambioderivoper}, we get%
\begin{align*}
\widetilde{X}_{l}\left(  au\right)   &  =\widetilde{X}_{l}P\left(  \xi
_{0}\right)  \widetilde{\mathcal{L}}_{0}u+\sum_{i,j=1}^{q}\widetilde{a}%
_{ij}\left(  \xi_{0}\right)  \widetilde{X}_{l}S_{ij}\left(  \xi_{0}\right)
u+\widetilde{X}_{l}S_{0}\left(  \xi_{0}\right)  u\\
&  =S_{l}\left(  \xi_{0}\right)  \widetilde{\mathcal{L}}_{0}u+\sum_{i,j=1}%
^{q}\widetilde{a}_{ij}\left(  \xi_{0}\right)  \left\{  \sum_{k=1}^{q}%
S_{l,k}^{ij}\left(  \xi_{0}\right)  \widetilde{X}_{k}u+\right. \\
&  \left.  +\sum_{h,k=1}^{q}\widetilde{a}_{hk}\left(  \xi_{0}\right)
S_{l,h}^{\prime ij}\left(  \xi_{0}\right)  \widetilde{X}_{k}u+P_{l}%
^{ij}\left(  \xi_{0}\right)  \widetilde{\mathcal{L}}_{0}u+S_{l}^{ij}\left(
\xi_{0}\right)  u\right\}  +\\
&  +\sum_{k=1}^{q}S_{k}^{l}\left(  \xi_{0}\right)  \widetilde{X}_{k}%
u+\sum_{h,k=1}^{q}\widetilde{a}_{hk}\left(  \xi_{0}\right)  S^{hl}\left(
\xi_{0}\right)  \widetilde{X}_{k}u+S_{0}^{l}\left(  \xi_{0}\right)
u+P^{l}\left(  \xi_{0}\right)  \widetilde{\mathcal{L}}_{0}u
\end{align*}%
\begin{align*}
&  =S_{l}^{\prime}\left(  \xi_{0}\right)  \widetilde{\mathcal{L}}_{0}%
u+\sum_{k=1}^{q}S_{l,k}\left(  \xi_{0}\right)  \widetilde{X}_{k}u+S_{l}%
^{0}\left(  \xi_{0}\right)  u\\
&  +\sum_{i,j=1}^{q}\widetilde{a}_{ij}\left(  \xi_{0}\right)  \left\{
\sum_{k=1}^{q}S_{l,k}^{ij}\left(  \xi_{0}\right)  \widetilde{X}_{k}%
u+\sum_{h,k=1}^{q}\widetilde{a}_{hk}\left(  \xi_{0}\right)  S_{l,h}^{\prime
ij}\left(  \xi_{0}\right)  \widetilde{X}_{k}u\right.  +\\
&  \left.  +P_{l}^{ij}\left(  \xi_{0}\right)  \widetilde{\mathcal{L}}%
_{0}u+S_{l}^{ij}\left(  \xi_{0}\right)  \underset{}{u}\right\}
\end{align*}
where\ all the $S_{...}\left(  \xi_{0}\right)  \ $are frozen operators of type
$1$ and $P_{l}^{ij}\left(  \xi_{0}\right)  $ is of type $2$.

Next, we perform another derivative $\widetilde{X}_{m}$ of both sides, getting%
\begin{align*}
\widetilde{X}_{m}\widetilde{X}_{l}\left(  au\right)   &  =T_{lm}\left(
\xi_{0}\right)  \widetilde{\mathcal{L}}_{0}u+\sum_{k=1}^{q}T_{lm,k}\left(
\xi_{0}\right)  \widetilde{X}_{k}u+T_{lm}^{0}\left(  \xi_{0}\right)  u\\
&  +\sum_{i,j=1}^{q}\widetilde{a}_{ij}\left(  \xi_{0}\right)  \left\{
\sum_{k=1}^{q}T_{lm,k}^{ij}\left(  \xi_{0}\right)  \widetilde{X}_{k}%
u+\sum_{h,k=1}^{q}\widetilde{a}_{hk}\left(  \xi_{0}\right)  T_{lm,h}^{\prime
ij}\left(  \xi_{0}\right)  \widetilde{X}_{k}u\right.  +\\
&  \left.  +S_{lm}^{ij}\left(  \xi_{0}\right)  \widetilde{\mathcal{L}}%
_{0}u+T_{lm}^{ij}\left(  \xi_{0}\right)  \underset{}{u}\right\}
\end{align*}
where\ all the $T_{...}\left(  \xi_{0}\right)  \ $are frozen operators of type
$0,$ and $S_{l}^{ij}\left(  \xi_{0}\right)  $ is of type 1. This is exactly
(\ref{Representation for second derivatives 1}). Finally, inserting in this
equation the identity
\begin{align}
\widetilde{\mathcal{L}}_{0}u(\xi)  &  =\widetilde{\mathcal{L}}u(\xi
)+(\widetilde{\mathcal{L}}_{0}-\widetilde{\mathcal{L}})u(\xi)\label{10.1}\\
&  =\widetilde{\mathcal{L}}u(\xi)+\sum_{i,j=1}^{q}\left[  \widetilde{a}%
_{ij}(\xi_{0})-\widetilde{a}_{ij}(\xi)\right]  \,\widetilde{X}_{i}%
\widetilde{X}_{j}u(\xi)\nonumber
\end{align}
we find (\ref{Rep formula Schauder}), and the theorem is proved.
\end{proof}

The above theorem is suited to the proof of $C_{\widetilde{X}}^{\alpha}$
estimates for $\widetilde{X}_{i}\widetilde{X}_{j}u$. In order to prove $L^{p}$
estimate for $\widetilde{X}_{i}\widetilde{X}_{j}u$ we need the following variation:

\begin{theorem}
[Representation of $\widetilde{X}_{m}\widetilde{X}_{l}u$ by variable
operators]\label{lemma representation for derivatives} Given $a\in
C_{0}^{\infty}\left(  \widetilde{B}\left(  \overline{\xi},R\right)  \right)
$, for any $m,l=1,2,...,q$, there exist variable operators over the ball
$\widetilde{B}\left(  \overline{\xi},R\right)  ,$ such that for any $u\in
C_{0}^{\infty}\left(  \widetilde{B}\left(  \overline{\xi},R\right)  \right)
$
\begin{align}
\widetilde{X}_{m}\widetilde{X}_{l}\left(  au\right)   &  =T_{lm}%
\widetilde{\mathcal{L}}u+\sum_{i,j=1}^{q}\left[  \widetilde{a}_{ij}%
,T_{lm}\right]  \,\widetilde{X}_{i}\widetilde{X}_{j}u+\sum_{k=1}^{q}%
T_{lm,k}\widetilde{X}_{k}u+T_{lm}^{0}u+\nonumber\\
&  +\sum_{i,j=1}^{q}\widetilde{a}_{ij}\left\{  \sum_{k=1}^{q}T_{lm,k}%
^{ij}\widetilde{X}_{k}u+\sum_{h,k=1}^{q}\widetilde{a}_{hk}T_{lm,h}^{\prime
ij}\widetilde{X}_{k}u+S_{lm}^{ij}\widetilde{\mathcal{L}}u+\right. \nonumber\\
&  \left.  +\sum_{i,j=1}^{q}\left[  \widetilde{a}_{ij},S_{lm}^{ij}\right]
\,\widetilde{X}_{i}\widetilde{X}_{j}u+T_{lm}^{ij}u\right\}  .
\label{representation of derive}%
\end{align}
Here all the $T_{...}$ are variable operators of type 0, $S_{lm}^{ij}$ is of
type 1, $\left[  a,T\right]  $ denotes the commutator of the multiplication
for $a$ with the operator $T$, and $\widetilde{a}_{ij}$ are the coefficients
of the operator $\widetilde{\mathcal{L}}$ (which are no longer frozen at
$\xi_{0}$).
\end{theorem}

\begin{remark}
\label{Remark easy representation}The above representation formula can be
written in a shorter way as%
\[
\widetilde{X}_{m}\widetilde{X}_{l}\left(  au\right)  =T_{lm}\widetilde
{\mathcal{L}}u+\sum_{i,j=1}^{q}\left[  \widetilde{a}_{ij},T_{lm}\right]
\,\widetilde{X}_{i}\widetilde{X}_{j}u+\sum_{k=1}^{q}T_{lm,k}\widetilde{X}%
_{k}u+T_{lm}^{0}u
\]
if we leave understood in the symbol of variable operator the possible
multiplication by the coefficients $\widetilde{a}_{ij}.$ See the previous remark.
\end{remark}

\begin{proof}
Let us rewrite (\ref{Rep formula Schauder}) as%
\begin{align*}
\widetilde{X}_{m}\widetilde{X}_{l}\left(  au\right)  \left(  \xi\right)   &
=\left(  T_{lm}\left(  \xi_{0}\right)  \widetilde{\mathcal{L}}u\right)
\left(  \xi\right)  +T_{lm}\left(  \xi_{0}\right)  \left(  \sum_{i,j=1}%
^{q}\left[  \widetilde{a}_{ij}(\xi_{0})-\widetilde{a}_{ij}(\cdot)\right]
\,\widetilde{X}_{i}\widetilde{X}_{j}u\right)  \left(  \xi\right)  +\\
&  +\sum_{k=1}^{q}\left(  T_{lm,k}\left(  \xi_{0}\right)  \widetilde{X}%
_{k}u\right)  \left(  \xi\right)  +\left(  T_{lm}^{0}\left(  \xi_{0}\right)
u\right)  \left(  \xi\right)  +\\
&  +\sum_{i,j=1}^{q}\widetilde{a}_{ij}\left(  \xi_{0}\right)  \left\{
\sum_{k=1}^{q}\left(  T_{lm,k}^{ij}\left(  \xi_{0}\right)  \widetilde{X}%
_{k}u\right)  \left(  \xi\right)  \right. \\
&  +\sum_{h,k=1}^{q}\widetilde{a}_{hk}\left(  \xi_{0}\right)  \left(
T_{lm,h}^{\prime ij}\left(  \xi_{0}\right)  \widetilde{X}_{k}u\right)  \left(
\xi\right)  +\left(  S_{lm}^{ij}\left(  \xi_{0}\right)  \widetilde
{\mathcal{L}}u\right)  \left(  \xi\right)  +\\
&  \left.  +S_{lm}^{ij}\left(  \xi_{0}\right)  \left(  \sum_{i,j=1}^{q}\left[
\widetilde{a}_{ij}(\xi_{0})-\widetilde{a}_{ij}(\cdot)\right]  \,\widetilde
{X}_{i}\widetilde{X}_{j}u\right)  \left(  \xi\right)  +\left(  T_{lm}%
^{ij}\left(  \xi_{0}\right)  u\right)  \left(  \xi\right)  \right\}  .
\end{align*}
for any $\xi\in\widetilde{B}\left(  \overline{\xi},R\right)  .$ Letting now
$\xi_{0}=\xi$ we get%
\begin{align*}
\widetilde{X}_{m}\widetilde{X}_{l}\left(  au\right)  \left(  \xi\right)   &
=\left(  T_{lm}\widetilde{\mathcal{L}}u\right)  \left(  \xi\right)
+T_{lm}\left(  \sum_{i,j=1}^{q}\left[  \widetilde{a}_{ij}(\xi)-\widetilde
{a}_{ij}(\cdot)\right]  \,\widetilde{X}_{i}\widetilde{X}_{j}u\right)  \left(
\xi\right)  +\\
&  +\sum_{k=1}^{q}\left(  T_{lm,k}\widetilde{X}_{k}u\right)  \left(
\xi\right)  +\left(  T_{lm}^{0}u\right)  \left(  \xi\right)  +\sum_{i,j=1}%
^{q}\widetilde{a}_{ij}\left(  \xi\right)  \left\{  \sum_{k=1}^{q}\left(
T_{lm,k}^{ij}\widetilde{X}_{k}u\right)  \left(  \xi\right)  \right. \\
&  +\sum_{h,k=1}^{q}\widetilde{a}_{hk}\left(  \xi\right)  \left(
T_{lm,h}^{\prime ij}\widetilde{X}_{k}u\right)  \left(  \xi\right)  +\left(
S_{lm}^{ij}\widetilde{\mathcal{L}}u\right)  \left(  \xi\right)  +\\
&  \left.  +S_{lm}^{ij}\left(  \sum_{i,j=1}^{q}\left[  \widetilde{a}_{ij}%
(\xi)-\widetilde{a}_{ij}(\cdot)\right]  \,\widetilde{X}_{i}\widetilde{X}%
_{j}u\right)  \left(  \xi\right)  +\left(  T_{lm}^{ij}u\right)  \left(
\xi\right)  \right\}  .
\end{align*}
where all the $T_{...}$ are \emph{variable }operators of type zero over
$\widetilde{B}\left(  \overline{\xi},R\right)  ,$ and $S_{lm}^{ij}$ are
variable operators of type 1. Note that%
\[
T\left(  \left[  \widetilde{a}_{ij}(\xi)-\widetilde{a}_{ij}(\cdot)\right]
\,\widetilde{X}_{i}\widetilde{X}_{j}u\right)  \left(  \xi\right)
\]
is exactly the commutator $\left[  \widetilde{a}_{ij},T\right]  $ applied to
$\widetilde{X}_{i}\widetilde{X}_{j}u.$ Hence the theorem is proved.
\end{proof}

\section{Singular integral estimates for operators of type
zero\label{Section singular kernels}}

The proof of \textit{a priori} estimates on the derivatives $\widetilde{X}%
_{i}\widetilde{X}_{j}u$ will follow, as will be explained in detail in
\S \ \ref{subsec shauder small support} and
\S \ \ref{subsec L^p small support}, combining the representation formulas
proved in \S \ \ref{subsec parametrix} with suitable $C^{\alpha}$ or $L^{p}$
estimates for \textquotedblleft operators of type zero\textquotedblright. To
be more precise, the results we need are the $C_{\widetilde{X}}^{\alpha
}\left(  \widetilde{B}\left(  \overline{\xi},R\right)  \right)  $ continuity
of a \emph{frozen operator of type zero} and the $L^{p}\left(  \widetilde
{B}\left(  \overline{\xi},R\right)  \right)  $ continuity of a \emph{variable
operator of type zero}, together with the $L^{p}\left(  \widetilde{B}\left(
\overline{\xi},r\right)  \right)  $ estimate for the commutator of a variable
operator of type zero with the multiplication with a $VMO$ function, implying
that the $L^{p}\left(  \widetilde{B}\left(  \overline{\xi},r\right)  \right)
$ norm of the commutator vanishes as $r\rightarrow0$. All these results will
be derived in the present section, as an application of abstract results
proved in \cite{bz} in the context of locally homogeneous spaces, which have
been recalled in \S \ \ref{subsec locally hom space}. To apply them, we need
to check that our kernels of type zero satisfy suitable properties. Moreover,
to study \emph{variable }operators of type zero, we also have to resort to the
classical technique of expansion in series of spherical harmonics, dating back
to Calder\'{o}n-Zygmund \cite{cz}, and already applied in the framework of
vector fields in \cite{bb1}, \cite{bb2}. This study will be split in two
subsection, the first devoted to frozen operators on $C^{\alpha},$ the second
to variable operators on $L^{p}$.

\subsection{$C^{\alpha}$ continuity of frozen operators of type $0$%
\label{subsec frozen C^alfa}}

The goal of this section is the proof of the following:

\begin{theorem}
\label{continuity for frozen operator on holder space}Let $\widetilde
{B}\left(  \overline{\xi},R\right)  $ be as before, $\xi_{0}\in\widetilde
{B}\left(  \overline{\xi},R\right)  ,$ and let $T(\xi_{0})$ be a frozen
operator of type $\lambda\geq0$ over $\widetilde{B}\left(  \overline{\xi
},R\right)  $. Then there exists $c>0$ such that for any $r<R$ and $u\in
C_{\widetilde{X},0}^{\alpha}(\widetilde{B}\left(  \overline{\xi},r\right)
)$,
\begin{equation}
\left\Vert T(\xi_{0})u\right\Vert _{C_{\widetilde{X}}^{\alpha}\left(
\widetilde{B}\left(  \overline{\xi},r\right)  \right)  }\leq c\left\Vert
u\right\Vert _{C_{\widetilde{X}}^{\alpha}\left(  \widetilde{B}\left(
\overline{\xi},r\right)  \right)  } \label{continuity for frozen operator}%
\end{equation}
where $c$ depends on $R,\left\{  \widetilde{X}_{i}\right\}  ,\alpha$ and
$\mu.$
\end{theorem}

Recall that $\mu$ is the \textquotedblleft ellipticity
constant\textquotedblright\ appearing in Assumption (H) (see
\S \ \ref{main result}).

To prove this, we will apply Theorems \ref{Theorem L^p C^eta} and
\ref{Thm frac C^eta} about the $C^{\alpha}$ continuity of singular and
fractional integrals in spaces of locally homogeneous type, taking
\begin{equation}
\Omega_{k}=\widetilde{B}\left(  \overline{\xi},\frac{kR}{k+1}\right)  \text{
for }k=1,2,3... \label{Om_k}%
\end{equation}

By Definition \ref{iidef:kernels:type}, our frozen kernel of type zero can be
written as:%

\begin{align*}
k(\xi_{0};\xi,\eta)  &  =k^{\prime}(\xi_{0};\xi,\eta)+k^{\prime\prime}(\xi
_{0};\xi,\eta)\\
&  =\left\{  \sum_{i=1}^{H}a_{i}(\xi)b_{i}(\eta)D_{i}\Gamma(\xi_{0}%
;\cdot)+a_{0}(\xi)b_{0}(\eta)D_{0}\Gamma(\xi_{0};\cdot)\right\}  \left(
\Theta(\eta,\xi)\right) \\
&  +\left\{  \sum_{i=1}^{H}a_{i}^{\prime}(\xi)b_{i}^{\prime}(\eta
)D_{i}^{\prime}\Gamma^{T}(\xi_{0};\cdot)+a_{0}^{\prime}(\xi)b_{0}^{\prime
}(\eta)D_{0}^{\prime}\Gamma^{T}(\xi_{0};\cdot)\right\}  \left(  \Theta
(\eta,\xi)\right)
\end{align*}
for some positive integer $H,$ where $a_{i},b_{i},a_{i}^{\prime},b_{i}%
^{\prime}\in C_{0}^{\infty}\left(  \widetilde{B}\left(  \overline{\xi
},R\right)  \right)  $ $(i=0,1,\ldots H)$, $D_{i}$ and $D_{i}^{\prime}$ are
differential operators such that: for $i=1,\ldots,H\,$, $D_{i}$ and
$D_{i}^{\prime}$ are homogeneous of degree $\leq2$ (so that $D_{i}\Gamma
(\xi_{0};\cdot)$ and $D_{i}^{\prime}\Gamma^{T}(\xi_{0};\cdot)$ are homogeneous
function of degree $\geq-Q$), $D_{0}$ and $D_{0}^{\prime}$ are differential
operators such that $D_{0}\Gamma(\xi_{0};\cdot)$ and\ $D_{0}^{\prime}%
\Gamma^{T}(\xi_{0};\cdot)$ have bounded first order (Euclidean) derivatives
(we will briefly say that $D_{0}\Gamma(\xi_{0};\cdot)$ and\ $D_{0}^{\prime
}\Gamma^{T}(\xi_{0};\cdot)$ are regular).

We will prove our Theorem for the operator with kernel $k^{\prime},$ the proof
for $k^{\prime\prime}$ being completely analogous. Let us split $k^{\prime}$
as%
\begin{align*}
k^{\prime}(\xi_{0};\xi,\eta)  &  =a_{1}(\xi)b_{1}(\eta)D_{1}\Gamma(\xi
_{0};\Theta(\eta,\xi))\\
&  +\left\{  \sum_{i=2}^{H_{m}}a_{i}(\xi)b_{i}(\eta)D_{i}\Gamma(\xi_{0}%
;\cdot)+a_{0}(\xi)b_{0}(\eta)D_{0}\Gamma(\xi_{0};\cdot)\right\}  \left(
\Theta(\eta,\xi)\right) \\
&  \equiv k_{S}\left(  \xi,\eta\right)  +k_{F}\left(  \xi,\eta\right)
\end{align*}
where $D_{1}\Gamma(\xi_{0};u)$ is homogeneous of degree $-Q$ while all the
kernels $D_{i}\Gamma(\xi_{0};u)$ are homogeneous of some degree $\geq1-Q$ and
$D_{0}\Gamma(\xi_{0};u)$ is regular. Recall that all these kernels may have
also an explicit (smooth) dependence on $\xi,\eta$; we will write $D_{i}%
^{\xi,\eta}\Gamma(\xi_{0};\Theta(\eta,\xi))$ to point out this fact, when it
will be important.

We want to apply Theorem \ref{Theorem L^p C^eta} (about singular integrals) to
the kernel $k_{S}$ and Theorem \ref{Thm frac C^eta} (about fractional
integrals) to each term of the kernel $k_{F}$.

We start with the following result, very similar to \cite[Proposition
2.17]{bb2}:

\begin{proposition}
\label{Prop homog kernels}Let $W^{\xi,\eta}\left(  \cdot\right)  $ be a
function defined on the homogeneous group $\mathbb{G}$, smooth outside the
origin and homogeneous of degree $\ell-Q$ for some nonnegative integer $\ell$,
smoothly depending on the parameters $\xi,\eta\in\widetilde{B}\left(
\overline{\xi},R\right)  ,$ and let
\[
K(\xi,\eta)=W^{\xi,\eta}\left(  \Theta\left(  \eta,\xi\right)  \right)
\]
be defined for $\xi,\eta\in\widetilde{B}\left(  \overline{\xi},R\right)  $.
Then $K$ \textit{satisfies}

\textit{(i)\ the growth condition: there exists a constant }$c$ such that%
\[
\left\vert K(\xi,\eta)\right\vert \leq c\cdot\sup_{\left\Vert u\right\Vert
=1}\left\vert W^{\xi,\eta}(u)\right\vert \cdot d_{\widetilde{X}}(\xi
,\eta)^{\ell-Q};
\]

(ii) the mean value inequality: there exists a constant $c>0,$ such that for
every $\xi_{0},\xi,\eta$ with $d_{\widetilde{X}}(\xi_{0},\eta)\geq
2d_{\widetilde{X}}(\xi_{0},\xi)$,%
\begin{equation}
\left\vert K(\xi_{0},\eta)-K(\xi,\eta)\right\vert +\left\vert K(\eta,\xi
_{0})-K(\eta,\xi)\right\vert \leq C\frac{d_{\widetilde{X}}(\xi_{0},\xi
)}{d_{\widetilde{X}}(\xi_{0},\eta)^{Q+1-\ell}} \label{ii20}%
\end{equation}
where the constant $C$ has the form%
\[
c\cdot\sup_{\left\Vert u\right\Vert =1,\xi,\eta\in\widetilde{B}\left(
\overline{\xi},R\right)  }\left\{  \frac{{}}{{}}\left\vert \nabla_{u}%
W^{\xi,\eta}(u)\right\vert +\left\vert \nabla_{\xi}W^{\xi,\eta}(u)\right\vert
+\left\vert \nabla_{\eta}W^{\xi,\eta}(u)\right\vert \frac{{}}{{}}\right\}
\]

(iii) the cancellation property: if $\ell=0$ and $W$ satisfies the vanishing
property%
\begin{equation}
\int_{r<\left\Vert u\right\Vert <R}W^{\xi,\eta}(u)du=0\text{ for every
}R>r>0,\text{ any }\xi,\eta\in\widetilde{B}\left(  \overline{\xi},R\right)
\label{vanishing}%
\end{equation}
then for any positive integer $k,$ for every $\varepsilon_{2}>\varepsilon
_{1}>0$ and $\xi\in\Omega_{k}$ (see (\ref{Om_k})) such that $\widetilde
{B}\left(  \xi,\varepsilon_{2}\right)  \subset\Omega_{k+1}$%
\begin{equation}
\left\vert \int_{\Omega_{k+1},\varepsilon_{1}<\rho\left(  \xi,\eta\right)
<\varepsilon_{2}}\text{ }K(\xi,\eta)\,d\eta\right\vert +\left\vert
\int_{\Omega_{k+1},\varepsilon_{1}<\rho\left(  \xi,\eta\right)  <\varepsilon
_{2}}\text{ }K(\eta,\xi)\,d\eta\right\vert \leq C\cdot(\varepsilon
_{2}-\varepsilon_{1}) \label{2.19}%
\end{equation}
where the constant $C$ has the form%
\[
c\cdot\sup_{\left\Vert u\right\Vert =1,\xi,\eta\in\widetilde{B}\left(
\overline{\xi},R\right)  }\left\{  \frac{{}}{{}}\left\vert W^{\xi,\eta
}(u)\right\vert +\left\vert \nabla_{\xi}W^{\xi,\eta}(u)\right\vert +\left\vert
\nabla_{\eta}W^{\xi,\eta}(u)\right\vert \frac{{}}{{}}\right\}  .
\]

\end{proposition}

\begin{proof}
Point (i) is trivial, by the homogeneity of $W,$ and the equivalence between
$d_{\widetilde{X}}$ and $\rho$ (see Lemma \ref{equivalence of distance}).

In order to prove (ii), fix $\xi_{0}$, $\eta$, and let $r=\frac{1}{2}\rho
(\eta,\xi_{0})$. Condition $\rho(\eta,\xi_{0})>2\rho(\xi,\xi_{0})$\ means that
$\xi$ is a point ranging in $\widetilde{B}_{r}(\xi_{0})$. Applying (\ref{4.6})
to the function%
\[
f\left(  \xi\right)  =K\left(  \xi,\eta\right)
\]
we can write
\begin{align*}
&  \left\vert f\left(  \xi\right)  -f\left(  \xi_{0}\right)  \right\vert \leq
cd_{\widetilde{X}}\left(  \xi,\xi_{0}\right)  \cdot\\
\cdot &  \left(  \overset{q}{\underset{i=1}{\sum}}\underset{\zeta\in
\widetilde{B}\left(  \xi_{0},\frac{3}{4}d_{\widetilde{X}}\left(  \xi_{0}%
,\eta\right)  \right)  }{\sup}\left\vert \widetilde{X}_{i}f\left(
\zeta\right)  \right\vert +d_{\widetilde{X}}\left(  \xi,\xi_{0}\right)
\underset{\zeta\in\widetilde{B}\left(  \xi_{0},\frac{3}{4}d_{\widetilde{X}%
}\left(  \xi_{0},\eta\right)  \right)  }{\sup}\left\vert \widetilde{X}%
_{0}f\left(  \zeta\right)  \right\vert \right)  .
\end{align*}
Noting that, for $\zeta\in\widetilde{B}\left(  \xi_{0},\frac{3}{4}%
d_{\widetilde{X}}\left(  \xi_{0},\eta\right)  \right)  ,$%
\begin{align*}
\left\vert \widetilde{X}_{i}K\left(  \cdot,\eta\right)  \left(  \zeta\right)
\right\vert  &  =\left\vert \widetilde{X}_{i}\left(  W^{\zeta,\eta}\left(
\Theta\left(  \cdot,\eta\right)  \right)  \right)  \left(  \zeta\right)
+\left(  \widetilde{X}_{i}W^{\cdot,\eta}\left(  \Theta\left(  \zeta
,\eta\right)  \right)  \right)  \left(  \zeta\right)  \right\vert \\
&  \leq\left\vert \left(  Y_{i}W+R_{i}^{\eta}W\right)  \left(  \Theta\left(
\eta,\zeta\right)  \right)  \right\vert +\left\vert \left(  \widetilde{X}%
_{i}W^{\cdot,\eta}\left(  \Theta\left(  \zeta,\eta\right)  \right)  \right)
\left(  \zeta\right)  \right\vert
\end{align*}
and recalling that, by Remark \ref{Remark parameter derivative},
$\nabla_{\zeta}W^{\zeta,\eta}(u)$ has the same $u$ homogeneity as
$W^{\zeta,\eta}(u)$, we get%
\begin{align*}
\left\vert \widetilde{X}_{i}K\left(  \cdot,\eta\right)  \left(  \zeta\right)
\right\vert  &  \leq\sup_{\left\Vert u\right\Vert =1,\zeta,\eta\in
\widetilde{B}\left(  \overline{\xi},R\right)  }\left\vert \nabla_{u}%
W^{\xi,\eta}(u)\right\vert \frac{c}{\rho\left(  \zeta,\eta\right)  ^{Q-\ell
+1}}+\\
&  +\sup_{\left\Vert u\right\Vert =1,\zeta,\eta\in\widetilde{B}\left(
\overline{\xi},R\right)  }\left\vert \nabla_{\zeta}W^{\zeta,\eta
}(u)\right\vert \frac{c}{\rho\left(  \zeta,\eta\right)  ^{Q-\ell}}\\
&  \leq\sup_{\left\Vert u\right\Vert =1,\zeta,\eta\in\widetilde{B}\left(
\overline{\xi},R\right)  }\left\{  \left\vert \nabla_{u}W^{\zeta,\eta
}(u)\right\vert +\left\vert \nabla_{\zeta}W^{\zeta,\eta}(u)\right\vert
\right\}  \frac{c}{d_{\widetilde{X}}\left(  \xi_{0},\eta\right)  ^{Q-\ell+1}}.
\end{align*}
Analogously%
\[
\left\vert \widetilde{X}_{0}K\left(  \cdot,\eta\right)  \left(  \zeta\right)
\right\vert \leq\sup_{\left\Vert u\right\Vert =1,\zeta,\eta\in\widetilde
{B}\left(  \overline{\xi},R\right)  }\left\{  \left\vert \nabla_{u}%
W^{\zeta,\eta}(u)\right\vert +\left\vert \nabla_{\zeta}W^{\zeta,\eta
}(u)\right\vert \right\}  \frac{c}{d_{\widetilde{X}}\left(  \xi_{0}%
,\eta\right)  ^{Q-\ell+2}},
\]
hence%
\begin{align*}
\left\vert K\left(  \xi,\eta\right)  -K\left(  \xi_{0},\eta\right)
\right\vert  &  \leq Cd_{\widetilde{X}}\left(  \xi,\xi_{0}\right)  \left(
\frac{1}{d_{\widetilde{X}}\left(  \xi_{0},\eta\right)  ^{Q-\ell+1}}%
+\frac{d_{\widetilde{X}}\left(  \xi,\xi_{0}\right)  }{d_{\widetilde{X}}\left(
\xi_{0},\eta\right)  ^{Q-\ell+2}}\right) \\
&  \leq C\frac{d_{\widetilde{X}}\left(  \xi,\xi_{0}\right)  }{d_{\widetilde
{X}}\left(  \xi_{0},\eta\right)  ^{Q-\ell+1}}%
\end{align*}
with%
\[
C=c\sup_{\left\Vert u\right\Vert =1,\zeta,\eta\in\widetilde{B}\left(
\overline{\xi},R\right)  }\left\{  \left\vert \nabla_{u}W^{\zeta,\eta
}(u)\right\vert +\left\vert \nabla_{\zeta}W^{\zeta,\eta}(u)\right\vert
\right\}  .
\]
To get the analogous bound for $\left\vert K(\eta,\xi_{0})-K(\eta
,\xi)\right\vert ,$ it is enough to apply the previous estimate to the
function%
\[
\widetilde{K}(\xi,\eta)=\widetilde{W}^{\xi,\eta}\left(  \Theta\left(  \eta
,\xi\right)  \right)  \text{ with }\widetilde{W}^{\xi,\eta}(u)=W^{\eta,\xi
}(u^{-1})\text{.}%
\]
This completes the proof of (ii).

To prove (iii), we first ignore the dependence on the parameters $\xi,\eta,$
and then we will show how to modify our argument to keep it into account. Let
us write:
\[
\int_{\Omega_{k+1},\varepsilon_{1}<\rho\left(  \xi,\eta\right)  <\varepsilon
_{2}}\text{ }W(\Theta(\eta,\xi))\,d\eta=
\]
by the change of variables $u=\Theta(\eta,\xi)$, Theorem \ref{metric} (b)
gives%
\[
=c(\xi)\int_{\varepsilon_{1}<\left\Vert u\right\Vert <\varepsilon_{2}}\text{
}W(u)\,\left(  1+\omega\left(  \xi,u\right)  )\right)  \,du=
\]
by the vanishing property of $W$,
\[
=\,c(\xi)\int_{\varepsilon_{1}<\left\Vert u\right\Vert <\varepsilon_{2}}\text{
}W(u)\omega\left(  \xi,u\right)  du.
\]
Then
\begin{align*}
\left\vert \int_{\Omega_{k+1},\varepsilon_{1}<\rho\left(  \xi,\eta\right)
<\varepsilon_{2}}\text{ }W(\Theta(\eta,\xi))\,d\eta\right\vert  &  \leq
c\cdot\int_{\varepsilon_{1}<\left\Vert u\right\Vert <\varepsilon_{2}}\text{
}\left\vert W(u)\right\vert \left\Vert u\right\Vert \,du\\
&  \leq c\cdot\sup_{\left\Vert u\right\Vert =1}\left\vert W\right\vert
\cdot\int_{\varepsilon_{1}<\left\Vert u\right\Vert <\varepsilon_{2}}\text{
}\left\Vert u\right\Vert ^{1-Q}du\\
&  \leq c\cdot\sup_{\left\Vert u\right\Vert =1}\left\vert W\right\vert
\cdot(\varepsilon_{2}-\varepsilon_{1}).
\end{align*}
Analogously one proves the bound on $W(\Theta(\xi,\eta)).$ Now, to keep track
of the possible dependence of $W$ on the parameters $\xi,\eta,$ let us write:%
\begin{align*}
&  \int_{\Omega_{k+1},\varepsilon_{1}<\rho\left(  \xi,\eta\right)
<\varepsilon_{2}}\text{ }W^{\xi,\eta}(\Theta(\eta,\xi))\,d\eta=\int
_{\Omega_{k+1},\varepsilon_{1}<\rho\left(  \xi,\eta\right)  <\varepsilon_{2}%
}\text{ }W^{\xi,\xi}(\Theta(\eta,\xi))\,d\eta+\\
&  +\int_{\Omega_{k+1},\varepsilon_{1}<\rho\left(  \xi,\eta\right)
<\varepsilon_{2}}\text{ }\left[  W^{\xi,\eta}(\Theta(\eta,\xi))-W^{\xi,\xi
}(\Theta(\eta,\xi))\,\right]  d\eta\\
&  \equiv I+II.
\end{align*}
The term $I$ can be bounded as above, while%
\[
\left\vert W^{\xi,\eta}(u)-W^{\xi,\xi}(u)\right\vert \leq\left\vert \xi
-\eta\right\vert \left\vert \nabla_{\eta}W^{\xi,\eta^{\prime}}(u)\right\vert
\]
for some point $\eta^{\prime}\ $near $\xi$ and $\eta.$ Recalling again that
the function $\nabla_{\eta}W^{\xi,\eta^{\prime}}\left(  \cdot\right)  $ has
the same homogeneity as $W^{\xi,\eta^{\prime}}\left(  \cdot\right)  ,$ while
\[
\left\vert \xi-\eta\right\vert \leq cd_{\widetilde{X}}\left(  \xi,\eta\right)
\leq c\rho\left(  \xi,\eta\right)  ,
\]
we have%
\[
\left\vert II\right\vert \leq c\sup_{\left\Vert u\right\Vert =1,\xi,\eta
\in\widetilde{B}\left(  \overline{\xi},R\right)  }\left\vert \nabla_{\eta
}W^{\xi,\eta}(u)\right\vert \int_{\Omega_{k+1},\varepsilon_{1}<\left\Vert
u\right\Vert <\varepsilon_{2}}\left\Vert u\right\Vert ^{1-Q}du
\]
and the same reasoning as above applies. This proves the bound on $\left\vert
\int\text{ }K(\xi,\eta)\,d\eta\right\vert $ in (\ref{2.19}). The proof of the
bound on $\left\vert \int K(\eta,\xi)\,d\eta\right\vert $ is analogous, since
the vanishing property (\ref{vanishing}) also implies the same for
$\int_{r<\left\Vert u\right\Vert <R}W^{\xi,\eta}(u^{-1})du.$
\end{proof}

\bigskip

The above Proposition implies that $D_{1}\Gamma(\xi_{0};\Theta(\eta,\xi))$
satisfies the standard estimates, cancellation property and convergence
condition stated in \S \ \ref{subsec locally hom space}.\ Note that
(\ref{2.19}) implies both the cancellation property and the convergence condition.

In order to apply to the kernel $k_{S}\left(  \xi,\eta\right)  $ Thm.
\ref{Theorem L^p C^eta} we still need to prove that the singular integral $T$
with kernel $k_{S}\left(  \xi,\eta\right)  $ satisfies the condition $T\left(
1\right)  \in C_{\widetilde{X}}^{\gamma}.$ (see condition
(\ref{h tilde C^gamma}) in Theorem \ref{Theorem L^p C^eta}).

This result is more delicate than the previous condition, and is contained in
the following:

\begin{proposition}
\label{Prop T(1)}Let
\[
\widetilde{h}\left(  \xi\right)  \equiv\lim_{\varepsilon\rightarrow0}%
\int_{\rho(\xi,\eta)>\varepsilon}\widetilde{K}(\xi,\eta)d\eta
\]
with
\[
\widetilde{K}(\xi,\eta)=a_{1}(\xi)b_{1}(\eta)D_{1}^{\xi,\eta}\Gamma(\xi
_{0};\Theta(\eta,\xi)),
\]
$D_{1}^{\xi,\eta}\Gamma(\xi_{0};\cdot)$ homogeneous of degree $-Q$ and
satisfying the vanishing property%
\[
\int_{r<\left\Vert u\right\Vert <R}D_{1}^{\xi,\eta}\Gamma(\xi_{0};u)du=0\text{
for every }R>r>0,\text{ any }\xi,\eta\in\widetilde{B}\left(  \overline{\xi
},R\right)  .
\]
Then $\widetilde{h}\in C_{\widetilde{X}}^{\gamma}\left(  \widetilde{B}\left(
\overline{\xi},R\right)  \right)  $ for any $\gamma\in\left(  0,1\right)  .$
\end{proposition}

\begin{proof}
Since $a_{1},b_{1}$ are compactly supported in $\widetilde{B}\left(
\overline{\xi},R\right)  ,$ we can choose a radial cutoff function%
\[
\phi\left(  \xi,\eta\right)  =f\left(  \rho\left(  \xi,\eta\right)  \right)
,
\]
with
\[
f\left(  \left\Vert u\right\Vert \right)  =1\text{ for }\left\Vert
u\right\Vert \leq R,\text{ }f\left(  \left\Vert u\right\Vert \right)  =0\text{
for }\left\Vert u\right\Vert \geq2R,
\]
so that $\widetilde{K}(\xi,\eta)=\widetilde{K}(\xi,\eta)\phi\left(  \xi
,\eta\right)  .$ To begin with, let us prove the assertion without taking into
consideration the dependence of $D_{1}^{\xi,\eta}\Gamma(\xi_{0};u)$ on
$\xi,\eta$. Then%
\begin{align*}
\widetilde{h}\left(  \xi\right)   &  =a_{1}(\xi)b_{1}\left(  \xi\right)
\lim_{\varepsilon\rightarrow0}\int_{\rho(\xi,\eta)>\varepsilon}\phi\left(
\xi,\eta\right)  D_{1}\Gamma(\xi_{0};\Theta(\eta,\xi))d\eta+\\
&  +a_{1}(\xi)\int\phi\left(  \xi,\eta\right)  D_{1}\Gamma(\xi_{0};\Theta
(\eta,\xi))\left[  b_{1}(\eta)-b_{1}\left(  \xi\right)  \right]  d\eta\\
&  \equiv I\left(  \xi\right)  +II\left(  \xi\right)  .
\end{align*}
Now,%
\begin{align*}
I\left(  \xi\right)   &  =a_{1}(\xi)b_{1}\left(  \xi\right)  c\left(
\xi\right)  \lim_{\varepsilon\rightarrow0}\int_{\left\Vert u\right\Vert
>\varepsilon}f\left(  \left\Vert u\right\Vert \right)  D_{1}\Gamma(\xi
_{0};u)\left(  1+\omega\left(  \xi,u\right)  \right)  du\\
&  =a_{1}(\xi)b_{1}\left(  \xi\right)  c\left(  \xi\right)  \int f\left(
\left\Vert u\right\Vert \right)  D_{1}\Gamma(\xi_{0};u)\omega\left(
\xi,u\right)  du,
\end{align*}
by the vanishing property, with $\omega$ smoothly depending on $\xi$ and
uniformly bounded by $c\left\Vert u\right\Vert $. Hence $I\left(  \xi\right)
$ is Lipschitz continuous, in particular H\"{o}lder \ continuous of any
exponent $\gamma\in\left(  0,1\right)  $. Moreover,
\begin{align*}
II\left(  \xi\right)   &  =a_{1}(\xi)\int_{\widetilde{B}\left(  \overline{\xi
},R\right)  }\kappa\left(  \xi,\eta\right)  d\eta\text{ with}\\
\kappa\left(  \xi,\eta\right)   &  =\phi\left(  \xi,\eta\right)  D_{1}%
\Gamma(\xi_{0};\Theta(\eta,\xi))\left[  b_{1}(\eta)-b_{1}\left(  \xi\right)
\right]  .
\end{align*}
It is not difficult to check that the kernel $\kappa\left(  \xi,\eta\right)  $
satisfies the standard estimates of \emph{fractional }integrals
(\ref{standard 1}) and (\ref{standard 2}) in
\S \ \ref{subsec locally hom space} for any $\nu\in\left(  0,1\right)  $
(actually, for $\nu=1$). Hence the operator with kernel $\kappa$ is continuous
on $C^{\gamma}\left(  \widetilde{B}\left(  \overline{\xi},R\right)  \right)
;$ in particular, it maps the function $1$ into $C^{\gamma}\left(
\widetilde{B}\left(  \overline{\xi},R\right)  \right)  ,$ which proves that
$II\left(  \xi\right)  $ is H\"{o}lder continuous.

To conclude the proof, we have to show how to take into account the possible
dependence of $D_{1}^{\xi,\eta}\Gamma\left(  \xi_{0};u\right)  $ on $\xi,\eta
$. Let us start with the $\eta$ dependence.%
\begin{align*}
\widetilde{h}\left(  \xi\right)   &  =a_{1}(\xi)b_{1}\left(  \xi\right)
\lim_{\varepsilon\rightarrow0}\int_{\rho(\xi,\eta)>\varepsilon}\phi\left(
\xi,\eta\right)  D_{1}^{\eta}\Gamma(\xi_{0};\Theta(\eta,\xi))d\eta+\\
&  +a_{1}(\xi)\int\phi\left(  \xi,\eta\right)  D_{1}^{\eta}\Gamma(\xi
_{0};\Theta(\eta,\xi))\left[  b_{1}(\eta)-b_{1}\left(  \xi\right)  \right]
d\eta\\
&  \equiv I^{\prime}\left(  \xi\right)  +II^{\prime}\left(  \xi\right)  .
\end{align*}
The term $II^{\prime}\left(  \xi\right)  $ can be handled as the term
$II\left(  \xi\right)  $ above. As to $I^{\prime}\left(  \xi\right)  $,%
\begin{align*}
I^{\prime}\left(  \xi\right)   &  =a_{1}(\xi)b_{1}\left(  \xi\right)  c\left(
\xi\right)  \lim_{\varepsilon\rightarrow0}\int_{\left\Vert u\right\Vert
>\varepsilon}f\left(  \left\Vert u\right\Vert \right)  D_{1}^{\Theta(\cdot
,\xi)^{-1}\left(  u\right)  }\Gamma(\xi_{0};u)\left(  1+\omega\left(
\xi,u\right)  \right)  du\\
&  =a_{1}(\xi)b_{1}\left(  \xi\right)  c\left(  \xi\right)  \lim
_{\varepsilon\rightarrow0}\int_{\left\Vert u\right\Vert >\varepsilon}f\left(
\left\Vert u\right\Vert \right)  D_{1}^{\Theta(\cdot,\xi)^{-1}\left(
u\right)  }\Gamma(\xi_{0};u)du+\\
&  +a_{1}(\xi)b_{1}\left(  \xi\right)  c\left(  \xi\right)  \int f\left(
\left\Vert u\right\Vert \right)  D_{1}^{\Theta(\cdot,\xi)^{-1}\left(
u\right)  }\Gamma(\xi_{0};u)\omega\left(  \xi,u\right)  du.
\end{align*}
The second term can be handled as above, while the first requires some care.
By the vanishing property of the kernel $D_{1}^{\zeta}\Gamma(\xi_{0};u)$ for
any fixed $\zeta$ we can write%
\begin{align*}
&  \lim_{\varepsilon\rightarrow0}\int_{\left\Vert u\right\Vert >\varepsilon
}f\left(  \left\Vert u\right\Vert \right)  D_{1}^{\Theta(\cdot,\xi
)^{-1}\left(  u\right)  }\Gamma(\xi_{0};u)du\\
&  =\lim_{\varepsilon\rightarrow0}\int_{\left\Vert u\right\Vert >\varepsilon
}f\left(  \left\Vert u\right\Vert \right)  \left[  D_{1}^{\Theta(\cdot
,\xi)^{-1}\left(  u\right)  }\Gamma(\xi_{0};u)-D_{1}^{\xi}\Gamma(\xi
_{0};u)\right]  du.
\end{align*}
On the other hand,%
\[
D_{1}^{\Theta(\cdot,\xi)^{-1}\left(  u\right)  }\Gamma(\xi_{0};u)=D_{1}^{\xi
}\Gamma(\xi_{0};u)+D_{0}^{\xi}\Gamma(\xi_{0};u)
\]
where $D_{0}^{\xi}$ is a vector field of local weight $\leq0,$ smoothly
depending on $\xi.$ Hence%
\[
\lim_{\varepsilon\rightarrow0}\int_{\left\Vert u\right\Vert >\varepsilon
}f\left(  \left\Vert u\right\Vert \right)  D_{1}^{\Theta(\cdot,\xi
)^{-1}\left(  u\right)  }\Gamma(\xi_{0};u)du=\int f\left(  \left\Vert
u\right\Vert \right)  D_{0}^{\xi}\Gamma(\xi_{0};u)du,
\]
which can be handled as the term $I\left(  \xi\right)  $ above.

Dependence on the variable $\xi$ can be taken into account as follows. If%
\begin{align*}
\widetilde{h}\left(  \xi\right)   &  =a_{1}(\xi)b_{1}\left(  \xi\right)
\lim_{\varepsilon\rightarrow0}\int_{\rho(\xi,\eta)>\varepsilon}\phi\left(
\xi,\eta\right)  D_{1}^{\xi,\eta}\Gamma(\xi_{0};\Theta(\eta,\xi))d\eta\\
&  \equiv\lim_{\varepsilon\rightarrow0}\int F_{\varepsilon}\left(  \xi
,\xi,\eta\right)  \text{ with}\\
F_{\varepsilon}\left(  \zeta,\xi,\eta\right)   &  =a_{1}(\xi)b_{1}\left(
\xi\right)  \chi_{\rho(\xi,\eta)>\varepsilon}\left(  \eta\right)  \phi\left(
\xi,\eta\right)  D_{1}^{\zeta,\eta}\Gamma(\xi_{0};\Theta(\eta,\xi))d\eta
\end{align*}
then%
\begin{align*}
\widetilde{h}\left(  \xi_{1}\right)  -\widetilde{h}\left(  \xi_{2}\right)   &
=\lim_{\varepsilon\rightarrow0}\int\left[  F_{\varepsilon}\left(  \xi_{1}%
,\xi_{1},\eta\right)  -F_{\varepsilon}\left(  \xi_{2},\xi_{1},\eta\right)
\right]  d\eta+\\
&  +\lim_{\varepsilon\rightarrow0}\int\left[  F_{\varepsilon}\left(  \xi
_{2},\xi_{1},\eta\right)  -F_{\varepsilon}\left(  \xi_{2},\xi_{2},\eta\right)
\right]  d\eta\\
&  \equiv A\left(  \xi_{1},\xi_{2}\right)  +B\left(  \xi_{1},\xi_{2}\right)  .
\end{align*}
Now,
\[
\left\vert A\left(  \xi_{1},\xi_{2}\right)  \right\vert \leq c\rho\left(
\xi_{1},\xi_{2}\right)
\]
by the smoothness of $\xi\mapsto D_{1}^{\xi,\eta}\Gamma(\xi_{0};u)$. As to
$B\left(  \xi_{1},\xi_{2}\right)  ,$ it is enough to apply the previous
reasoning to the kernel $D_{1}^{\zeta,\eta}\Gamma(\xi_{0};\Theta(\eta,\xi)),$
for any fixed $\zeta,$ to conclude that%
\[
\left\vert \lim_{\varepsilon\rightarrow0}\int\left[  F\left(  \zeta,\xi
_{1},\eta\right)  -F\left(  \zeta,\xi_{2},\eta\right)  \right]  d\eta
\right\vert \leq c\rho\left(  \xi_{1},\xi_{2}\right)  ^{\gamma}%
\]
for some constant uniformly bounded in $\zeta,$ and then apply this inequality
taking $\zeta=\xi_{2}$. This completes the proof.
\end{proof}

\bigskip

We are now ready for the

\begin{proof}
[Conclusion of the proof of Theorem
\ref{continuity for frozen operator on holder space}]Recall that a frozen
operator of type zero is written as
\[
T(\xi_{0})f(\xi)=P.V.\int_{\widetilde{B}}k(\xi_{0};\xi,\eta)\,f(\eta
)\,d\eta+\alpha\left(  \xi_{0},\xi\right)  f\left(  \xi\right)  ,
\]
where $\alpha$ is a bounded measurable function, smooth in $\xi.$ The
multiplicative part%
\[
f\left(  \xi\right)  \longmapsto\alpha\left(  \xi_{0},\xi\right)  f\left(
\xi\right)
\]
clearly maps $C^{\alpha}$ in $C^{\alpha},$ since $\alpha\left(  \xi_{0}%
,\cdot\right)  $ is smooth, with operator norm bounded by some constant
depending on the vector fields and the ellipticity constant $\mu,$ by Theorem
\ref{ith:stima:unif:ipo}.

Let us now consider the integral part. With the notation introduced at the
beginning of this section, let us consider first
\[
k_{S}\left(  \xi,\eta\right)  =a_{1}(\xi)b_{1}(\eta)D_{1}^{\xi,\eta}\Gamma
(\xi_{0};\Theta(\eta,\xi)),
\]
with $D_{1}^{\xi,\eta}\Gamma\left(  \xi_{0};u\right)  $ homogeneous of degree
$-Q$ and satisfying the vanishing property (\ref{vanishing}). By Proposition
\ref{Prop homog kernels}, $k_{S}\left(  \xi,\eta\right)  $ satisfies
conditions (i), (ii), (iii) in \S \ \ref{subsec locally hom space}, with
constants bounded by%
\begin{equation}
c\sup_{\left\Vert u\right\Vert =1}\left\{  \left\vert D^{2}\Gamma\left(
\xi_{0},u\right)  \right\vert +\left\vert D^{3}\Gamma\left(  \xi_{0},u\right)
\right\vert \right\}  \label{kernel bound schauder}%
\end{equation}
where the symbols $D^{2},D^{3}$ denote standard derivatives of orders $2,3,$
respectively, with respect to $u,$ and the constant $c$ depends on the vector
fields but not on the point $\xi_{0}$. By Proposition \ref{Prop T(1)},
condition (\ref{h tilde C^gamma}) is also satisfied by $k_{S}\left(  \xi
,\eta\right)  ,$ with $C^{\gamma}$ norm bounded by a quantity of the kind
(\ref{kernel bound schauder}). Hence by Theorem \ref{Theorem L^p C^eta} and
Remark \ref{remark local holder}, the operator with kernel $k_{S}\left(
\xi,\eta\right)  $ satisfies the assertion of the theorem we are proving, with
a constant bounded by a quantity like (\ref{kernel bound schauder}). In turn,
by Theorem \ref{ith:stima:unif:ipo} this quantity can be bounded by a constant
depending on the vector fields and the ellipticity constant $\mu$ of the
matrix $a_{ij}\left(  x\right)  .$

Let us now come to the kernel%
\[
k_{F}\left(  \xi,\eta\right)  =\left\{  \sum_{i=2}^{H}a_{i}(\xi)b_{i}%
(\eta)D_{i}^{\xi,\eta}\Gamma\left(  \xi_{0};\cdot\right)  +a_{0}(\xi
)b_{0}(\eta)D_{0}^{\xi,\eta}\Gamma\left(  \xi_{0};\cdot\right)  \right\}
\left(  \Theta(\eta,\xi)\right)
\]
where each function $D_{i}^{\xi,\eta}\Gamma\left(  \xi_{0};u\right)  $
($i=2,3,...,H$) is homogeneous of some degree $\geq1-Q,$ while $D_{0}%
^{\xi,\eta}\Gamma\left(  \xi_{0};u\right)  $ is bounded and smooth.

By Proposition \ref{Prop homog kernels}, each kernel
\[
a_{i}(\xi)b_{i}(\eta)D_{i}^{\xi,\eta}\Gamma\left(  \xi_{0};\Theta(\eta
,\xi)\right)
\]
satisfies the standard estimates (i) in \S \ \ref{subsec locally hom space}
for some $\nu>0,$ hence we can apply Theorem \ref{Thm frac C^eta} to the
integral operators defined by these kernels, and conclude as above. Finally,
the integral operator with regular kernel clearly is $C^{\gamma}$ continuous.
So we are done.
\end{proof}

\subsection{$L^{p}$ continuity of variable operators of type $0$ and their
commutators\label{subsec variable L^p}}

In this subsection we are going to prove the following:

\begin{theorem}
\label{iith:comm:optypezer}Let $T$ be a variable operator of type $0$ (see
\S \ \ref{Operators of type}) over the ball $\widetilde{B}\left(
\overline{\xi},R\right)  ,$ and $p\in(1,\infty).$Then:

(i) there exists $c>0,$ depending on $p$, $R$, $\left\{  \widetilde{X}%
_{i}\right\}  _{i=0}^{q}$, $\mu$, such that%
\[
\left\Vert Tu\right\Vert _{L^{p}\left(  \widetilde{B}\left(  \overline{\xi
},r\right)  \right)  }\leq c\left\Vert u\right\Vert _{L^{p}\left(
\widetilde{B}\left(  \overline{\xi},r\right)  \right)  }%
\]
for every $u\in L^{p}\left(  \widetilde{B}\left(  \overline{\xi},r\right)
\right)  $ and $r\leq R;$

(ii) for every $a\in VMO_{X,loc}\left(  \Omega\right)  ,$ any $\varepsilon>0$,
there exists $r\leq R$ such that for every $u\in L^{p}\left(  \widetilde
{B}\left(  \overline{\xi},r\right)  \right)  ,$
\begin{equation}
\left\Vert T(\widetilde{a}u)-\widetilde{a}\cdot Tu\right\Vert _{L^{p}\left(
\widetilde{B}\left(  \overline{\xi},r\right)  \right)  }\leq\varepsilon
\left\Vert u\right\Vert _{L^{p}\left(  \widetilde{B}\left(  \overline{\xi
},r\right)  \right)  } \label{comm}%
\end{equation}
where $\widetilde{a}\left(  x,h\right)  =a\left(  x\right)  .$ The number $r$
depends on $p$, $R$, $\left\{  \widetilde{X}_{i}\right\}  _{i=0}^{q}$, $\mu$,
$\eta_{a,\Omega^{\prime},\Omega}^{\ast}$, and $\varepsilon$ (see
\ref{subsec VMO} for the definition of $VMO_{X,loc}\left(  \Omega\right)  $
and $\eta_{a,\Omega^{\prime},\Omega}^{\ast}$).{}
\end{theorem}

A basic difference with the context of the previous section is that here we
are considering \emph{variable }kernels and operators of type zero. To reduce
the study of these operators to that of constant operators of type zero we
will make use of the classical technique of expansion in series of spherical
harmonics, as already done in \cite{bb2}.

\begin{proof}
This proof is similar to that of \cite[Thm. 2.11]{bb2}. Recall that a variable
operator of type zero is written as
\[
Tf(\xi)=P.V.\int_{\widetilde{B}}k(\xi;\xi,\eta)\,f(\eta)\,d\eta+\alpha\left(
\xi,\xi\right)  f\left(  \xi\right)  ,
\]
where $\alpha\left(  \xi_{0},\xi\right)  $ is a bounded measurable function in
$\xi_{0}$, smooth in $\xi.$ The multiplicative part%
\[
f\left(  \xi\right)  \longmapsto\alpha\left(  \xi,\xi\right)  f\left(
\xi\right)
\]
clearly maps $L^{p}$ into $L^{p}$, with operator norm bounded by some constant
depending on the vector fields and the ellipticity constant $\mu,$ by Theorem
\ref{ith:stima:unif:ipo}. Moreover, this part does not affect the commutator
of $T$.

As to the integral part of $T$, let us split the variable kernel as%
\[
k(\xi;\xi,\eta)=k^{\prime}(\xi;\xi,\eta)+k^{\prime\prime}(\xi;\xi,\eta).
\]
Like in the previous section, it is enough to prove our result for the kernel
$k^{\prime}.$ Let us expand it as%
\begin{align*}
k^{\prime}(\xi;\xi,\eta)  &  =\sum_{i=1}^{H}a_{i}(\xi)b_{i}(\eta)D_{i}%
^{\xi,\eta}\Gamma\left(  \xi;\Theta\left(  \eta,\xi\right)  \right)
+a_{0}(\xi)b_{0}(\eta)D_{0}^{\xi,\eta}\Gamma(\xi;\Theta\left(  \eta
,\xi\right)  )\\
&  \equiv k_{S}\left(  \xi;\xi,\eta\right)  +k_{B}\left(  \xi;\xi,\eta\right)
\end{align*}
where the kernels $D_{i}^{\xi,\eta}\Gamma(\xi;u)$ (for $i=1,2,3,...,H$) are
homogeneous of some degree $\geq-Q,$ $D_{i}^{\xi,\eta}\Gamma\left(
\xi;u\right)  $ satisfies the cancellation property, and $D_{0}^{\xi,\eta
}\Gamma(\xi;u)$ is bounded in $u$ and smooth in $\xi,\eta$. The kernels
$k_{S},k_{B}$ are \textquotedblleft singular\textquotedblright\ and
\textquotedblleft bounded\textquotedblright, respectively.

The operator with kernel $k_{B}$ is obviously $L^{p}$ continuous. Moreover, it
satisfies the commutator estimate (\ref{comm}) by Theorem \ref{Thm comm pos},
since
\[
\left\vert k_{B}\left(  \xi;\xi,\eta\right)  \right\vert \leq ca_{0}(\xi
)b_{0}(\eta)
\]
and the constant function $1$ obviously satisfies the standard estimates
(\ref{standard 1}), (\ref{standard 2}) with $\nu=1$.

To handle the kernel $k_{S}$ we expand each of its terms in series of
spherical harmonics, exactly like in \cite[Section 2.4]{bb2}:%
\[
D_{i}^{\xi,\eta}\Gamma(\xi;u)=\sum_{m=0}^{\infty}\sum_{k=1}^{g_{m}}%
c_{i,km}^{\xi,\eta}\left(  \xi\right)  K_{i,km}\left(  u\right)
\]
where $K_{i,km}\left(  u\right)  $ are homogeneous kernels which on the sphere
$\left\Vert u\right\Vert =1$ coincide with the spherical harmonics, and
$c_{i,km}^{\xi,\eta}\left(  \cdot\right)  $ the corresponding Fourier coefficients.

Let us first prove the assertion without taking into account the dependence of
the coefficients $c_{i,km}^{\xi,\eta}\left(  \xi\right)  $ on $\eta.$ Then the
operator with kernel $k_{S}$ can be written as:%
\begin{equation}
Sf\left(  \xi\right)  =\sum_{m=0}^{\infty}\sum_{k=1}^{g_{m}}c_{i,km}^{\xi
}\left(  \xi\right)  S_{i,km}f\left(  \xi\right)  \label{expansion}%
\end{equation}
with%
\[
S_{i,km}f\left(  \xi\right)  =a_{i}(\xi)\int_{\widetilde{B}}b_{i}%
(\eta)K_{i,km}\left(  \Theta\left(  \eta,\xi\right)  \right)  f\left(
\eta\right)  d\eta.
\]
The number $g_{m}$ in (\ref{expansion}) is the dimension of the space of
spherical harmonics of degree $m$ in $\mathbb{R}^{N}$; it is known that
\begin{equation}
g_{m}\leq c(N)\cdot m^{N-2}\text{ \ for every }m=1,2,\ldots\label{ii12}%
\end{equation}
For every $p\in(1,\infty)$ we can write:%
\[
\left\Vert Sf\right\Vert _{L^{p}\left(  \widetilde{B}\left(  \overline{\xi
},r\right)  \right)  }\leq\sum_{m=0}^{\infty}\sum_{k=1}^{g_{m}}\left\Vert
c_{i,km}^{\cdot}\left(  \cdot\right)  \right\Vert _{L^{\infty}\left(
\widetilde{B}\left(  \overline{\xi},r\right)  \right)  }\left\Vert
S_{i,km}f\right\Vert _{L^{p}\left(  \widetilde{B}\left(  \overline{\xi
},r\right)  \right)  }\text{ }%
\]
and%
\begin{align*}
&  \left\Vert S(\widetilde{a}f)-\widetilde{a}\cdot Sf\right\Vert
_{L^{p}\left(  \widetilde{B}\left(  \overline{\xi},r\right)  \right)  }\leq\\
&  \leq\sum_{m=0}^{\infty}\sum_{k=1}^{g_{m}}\left\Vert c_{i,km}^{\cdot}\left(
\cdot\right)  \right\Vert _{L^{\infty}\left(  \widetilde{B}\left(
\overline{\xi},r\right)  \right)  }\left\Vert S_{i,km}(\widetilde
{a}f)-\widetilde{a}\cdot S_{i,km}f\right\Vert _{L^{p}\left(  \widetilde
{B}\left(  \overline{\xi},r\right)  \right)  }.
\end{align*}
Now each $S_{i,km}$ is a frozen operator of type $\lambda\geq0$, and the same
arguments of the previous section show that the kernel of $S_{i,km}$ satisfies
the assumptions (i),(ii),(iii) in \S \ \ref{subsec locally hom space} with
constants bounded by%
\[
c\cdot\sup_{\left\Vert u\right\Vert =1}\left\vert \nabla_{u}K_{km}%
(u)\right\vert ,
\]
(with $c$ depending on the vector fields);\ in turn, by known properties of
spherical harmonics we have%
\[
\sup_{\left\Vert u\right\Vert =1}\left\vert \nabla_{u}K_{km}(u)\right\vert
\leq c\left(  N\right)  m^{N/2},
\]
so that, by Theorem \ref{Theorem L^p C^eta} and Theorem \ref{frac lp-lq} we
conclude as in \cite[p. 807]{bb2},%
\[
\left\Vert S_{i,km}f\right\Vert _{L^{p}\left(  \widetilde{B}\left(
\overline{\xi},r\right)  \right)  }\leq c\cdot m^{N/2}\left\Vert f\right\Vert
_{L^{p}\left(  \widetilde{B}\left(  \overline{\xi},r\right)  \right)  }\text{
for }i=1,2,...,H.
\]
where we have also taken into account Remark \ref{remark linear constants}.

Analogously, applying Theorem \ref{Thm commutator} and Theorem
\ref{Thm comm frac} we have the commutator estimate:%
\[
\left\Vert S_{i,km}(\widetilde{a}f)-\widetilde{a}\cdot S_{i,km}f\right\Vert
_{L^{p}\left(  \widetilde{B}\left(  \overline{\xi},r\right)  \right)  }%
\leq\varepsilon\cdot m^{N/2}\left\Vert f\right\Vert _{L^{p}\left(
\widetilde{B}\left(  \overline{\xi},r\right)  \right)  }\text{ for
}i=1,2,...,H,
\]
for any $\varepsilon>0,$ provided $r$ is small enough, depending on
$\varepsilon$ and $\eta_{\widetilde{a},\Omega_{k+2},\Omega_{k+3}}^{\ast}$ (see
(\ref{Om_k}) and Definition \ref{definition local VMO} for the meaning of
symbols). By Proposition \ref{Prop lifted VMO}, then, the constant $r$ depends
on the function $a$ only through the local VMO modulus $\eta_{a,\Omega
^{\prime},\Omega}^{\ast}$.

Next, again by known properties of spherical harmonics, we can say that for
any positive integer $h$ there exists $c_{h}$ such that%
\[
\left\vert c_{i,km}^{\zeta}\left(  \xi\right)  \right\vert \leq c_{h}\cdot
m^{-2h}\sup_{\left\Vert u\right\Vert =1,\left\vert \beta\right\vert
=2h}\left\vert \left(  \frac{\partial}{\partial u}\right)  ^{\beta}%
D_{i}^{\zeta}\Gamma(\xi;u)\right\vert .
\]
By the uniform estimates contained in Theorem \ref{ith:stima:unif:ipo}, the
last expression is bounded by $Cm^{-2h},$ for some constant $C$ depending on
$h,$ the vector fields, and the ellipticity constant $\mu.$ Taking into
account also (\ref{ii12}) and choosing $h$ large enough we conclude%
\[
\left\Vert Sf\right\Vert _{L^{p}\left(  \widetilde{B}\left(  \overline{\xi
},r\right)  \right)  }\leq\sum_{m=0}^{\infty}Cg_{m}m^{-2h}m^{N/2}\left\Vert
f\right\Vert _{L^{p}\left(  \widetilde{B}\left(  \overline{\xi},r\right)
\right)  }=c\left\Vert f\right\Vert _{L^{p}\left(  \widetilde{B}\left(
\overline{\xi},r\right)  \right)  }%
\]
and%
\[
\left\Vert S(\widetilde{a}f)-\widetilde{a}\cdot Sf\right\Vert _{L^{p}\left(
\widetilde{B}\left(  \overline{\xi},r\right)  \right)  }\leq c\varepsilon
\left\Vert f\right\Vert _{L^{p}\left(  \widetilde{B}\left(  \overline{\xi
},r\right)  \right)  }%
\]
for any $\varepsilon>0,$ provided $r$ is small enough.

We are left to show how the previous argument needs to be modified to take
into account the possible dependence of $D_{i}^{\xi,\eta}\Gamma(\xi;u)$ (and
then of $c_{i,km}^{\xi,\eta}\left(  \xi\right)  $) on $\eta$. Let us expand:%
\[
D_{i}^{\zeta,\Theta\left(  \cdot,\zeta\right)  ^{-1}\left(  u\right)  }%
\Gamma(\xi;u)=\sum_{m=0}^{\infty}\sum_{k=1}^{g_{m}}c_{i,km}^{\zeta}\left(
\xi\right)  K_{i,km}\left(  u\right)
\]
so that%
\[
D_{i}^{\zeta,\eta}\Gamma(\xi;\Theta\left(  \eta,\zeta\right)  )=\sum
_{m=0}^{\infty}\sum_{k=1}^{g_{m}}c_{i,km}^{\zeta}\left(  \xi\right)
K_{i,km}\left(  \Theta\left(  \eta,\zeta\right)  \right)  .
\]
The kernels $K_{i,km}$ are the same as above, hence the estimates on the
operators $S_{i,km}$ and their commutators remain unchanged. As to the
coefficients $c_{i,km}^{\zeta}\left(  \xi\right)  ,$ we now have to write, for
any positive integer $h$ and some constant $c_{h}$,%
\[
\left\vert c_{i,km}^{\zeta}\left(  \xi\right)  \right\vert \leq c_{h}\cdot
m^{-2h}\sup_{\left\Vert u\right\Vert =1,\left\vert \beta\right\vert
=2h}\left\vert \left(  \frac{\partial}{\partial u}\right)  ^{\beta}\left(
D_{i}^{\zeta,\Theta\left(  \cdot,\zeta\right)  ^{-1}\left(  u\right)  }%
\Gamma(\xi;u)\right)  \right\vert .
\]
Now, from the identity%
\begin{align*}
\frac{\partial}{\partial u_{j}}\left(  D_{i}^{\zeta,\Theta\left(  \cdot
,\zeta\right)  ^{-1}\left(  u\right)  }\Gamma(\xi;u)\right)   &
=\frac{\partial}{\partial u_{j}}\left(  D_{i}^{\zeta,\eta}\Gamma
(\xi;u)\right)  _{/\eta=\Theta\left(  \cdot,\zeta\right)  ^{-1}\left(
u\right)  }+\\
&  +\sum_{m}\frac{\partial}{\partial\eta_{m}}\left(  D_{i}^{\zeta,\eta}%
\Gamma(\xi;u)\right)  \frac{\partial}{\partial u_{j}}\left(  \Theta\left(
\cdot,\zeta\right)  ^{-1}\left(  u\right)  \right)  _{m}%
\end{align*}
it is easy to see that we can still get a uniform bound of the kind%
\[
\left\vert c_{i,km}^{\zeta}\left(  \xi\right)  \right\vert \leq C\cdot m^{-2h}%
\]
with $C$ depending on $h,$ the vector fields and the ellipticity constant
$\mu$. So we are done.
\end{proof}

\section{Schauder estimates}

We are now in position to apply all the machinery presented in the previous
sections to prove our main results, that is $C^{\alpha}$ and $L^{p}$ estimates
on $X_{i}X_{j}u$ in terms of $u$ and $\mathcal{L}u$. We will prove $C^{\alpha
}$ estimates (that is Theorem \ref{schauder estimate}) in this section, and
$L^{p}$ estimates (that is Theorem \ref{lp estimate}) in
\S \ \ref{section L^p estimates}.

We keep assuming $a_{0}\left(  x\right)  \equiv1,$ which is not restrictive in
view of Remark \ref{a0=1}$.$

Let us recall the setting described at the end of
\S \ \ref{subsec locally hom space}. For a fixed subdomain $\Omega^{\prime
}\Subset\Omega\subset\mathbb{R}^{n}$ and a fixed point $\overline{x}\in
\Omega^{\prime},$ let us consider a lifted ball $\widetilde{B}\left(
\overline{\xi},R\right)  \subset\mathbb{R}^{N}$ (with $\overline{\xi}=\left(
\overline{x},0\right)  $) where the lifted vector fields $\widetilde{X}_{i}$
are defined and satisfy H\"{o}rmander's condition, the map $\Theta$ is defined
and satisfies the properties stated in \S \ \ref{subsection lifting}.

According to the procedure followed in \cite[\S \ 5]{bb4}, the proof of
$C_{X}^{\alpha}$ a-priori estimates for second order derivatives will proceed
in three steps: first, in the space of lifted variables, for test functions
supported in a ball $\widetilde{B}\left(  \overline{\xi},r\right)  $ with $r$
small enough; then for any function in $C_{\widetilde{X}}^{2,\alpha}\left(
\widetilde{B}\left(  \overline{\xi},r\right)  \right)  $ (not necessarily
vanishing at the boundary); then for any function in $C_{X}^{2,\alpha}\left(
B\left(  \overline{x},r\right)  \right)  ,$ that is in the original space. The
three steps will be performed in separate subsections. The theory of singular
integrals in locally homogeneous spaces will play its main role in the first
step, considering the space%
\[
\left(  \widetilde{\Omega},\left\{  \widetilde{\Omega}_{k}\right\}
_{k=1}^{\infty},d_{\widetilde{X}},d\xi\right)
\]
where%
\[
\widetilde{\Omega}=\widetilde{B}\left(  \overline{\xi},R\right)
;\widetilde{\Omega}_{k}=\widetilde{B}\left(  \overline{\xi},\frac{kR}%
{k+1}\right)  \text{ for }k=1,2,3,...
\]

\subsection{Schauder estimates for functions with small
support\label{subsec shauder small support}}

The first step in the proof of Schauder estimates is contained in the
following theorem, which is the main result in this subsection.

\begin{theorem}
\label{schaulder estimate for compactly function in ball}Let $\widetilde
{B}\left(  \overline{\xi},R\right)  $ be as before. There exist $R_{0}<R$ and
$c>0$ such that for every $u\in C_{\widetilde{X},0}^{2,\alpha}\left(
\widetilde{B}\left(  \overline{\xi},R_{0}\right)  \right)  ,$
\[
\left\Vert u\right\Vert _{C_{\widetilde{X}}^{2,\alpha}\left(  \widetilde
{B}\left(  \overline{\xi},R_{0}\right)  \right)  }\leq c\left\{  \left\Vert
\widetilde{\mathcal{L}}u\right\Vert _{C_{\widetilde{X}}^{\alpha}\left(
\widetilde{B}\left(  \overline{\xi},R_{0}\right)  \right)  }+\left\Vert
u\right\Vert _{L^{\infty}\left(  \widetilde{B}\left(  \overline{\xi}%
,R_{0}\right)  \right)  }\right\}
\]
where $c$ and $R_{0}$ depend on $R,\left\{  \widetilde{X}_{i}\right\}
,\alpha,\mu,$and $\left\Vert \widetilde{a}_{ij}\right\Vert _{C^{\alpha}\left(
\widetilde{B}\left(  \overline{\xi},R\right)  \right)  }$.
\end{theorem}

\begin{proof}
This theorem relies on the representation formulas proved in
\S \ \ref{Operators of type} and Theorem
\ref{continuity for frozen operator on holder space} about singular integrals
on $C^{\alpha}$, in \S \ \ref{subsec frozen C^alfa}. The proof is similar to
that of \cite[Thm. 5.2]{bb4}. We start from the representation formula
(\ref{Rep formula Schauder}), choosing $r<R$ such that $\widetilde{B}%
_{r}\equiv\widetilde{B}\left(  \overline{\xi},r\right)  $ be contained in the
set where $a\equiv1.$ Taking $C_{\widetilde{X}}^{\alpha}\left(  \widetilde
{B}\left(  \overline{\xi},r\right)  \right)  $ norm of both sides of
(\ref{Rep formula Schauder}) and applying Theorem
\ref{continuity for frozen operator on holder space} and (\ref{4.3}) in
Proposition \ref{basic properties for holder norm}, we can write, for any
$u\in C_{\widetilde{X},0}^{2,\alpha}\left(  \widetilde{B}\left(  \overline
{\xi},r\right)  \right)  $ and $m,l=1,2,...,q$%
\begin{align*}
\left\Vert \widetilde{X}_{m}\widetilde{X}_{l}u\right\Vert _{C_{\widetilde{X}%
}^{\alpha}(\widetilde{B}_{r})}  &  \leq c\left\{  \left\Vert \widetilde
{\mathcal{L}}u\right\Vert _{C_{\widetilde{X}}^{\alpha}(\widetilde{B}_{r}%
)}+\sum_{i,j=1}^{q}\left\Vert \left[  \widetilde{a}_{ij}(\xi_{0}%
)-\widetilde{a}_{ij}\left(  \cdot\right)  \right]  \,\widetilde{X}%
_{i}\widetilde{X}_{j}u\right\Vert _{C_{\widetilde{X}}^{\alpha}(\widetilde
{B}_{r})}\right. \\
&  \left.  +\overset{q}{\sum\limits_{k=1}}\left\Vert \widetilde{X}%
_{k}u\right\Vert _{C_{\widetilde{X}}^{\alpha}(\widetilde{B}_{r})}+\left\Vert
u\right\Vert _{C_{\widetilde{X}}^{\alpha}(\widetilde{B}_{r})}\right\}
\end{align*}
for some $c$ depending on $R,\left\{  \widetilde{X}_{i}\right\}  ,\alpha,\mu.$

To handle the terms involving $\widetilde{X}_{i}\widetilde{X}_{j}u$ in the
right-hand side of the last inequality, we now exploit the fact that, for
$u\in C_{\widetilde{X},0}^{2,\alpha}(\widetilde{B}_{r}),$ both $\widetilde
{X}_{i}\widetilde{X}_{j}u$ and $\left[  \widetilde{a}_{ij}(\xi_{0}%
)-\widetilde{a}_{ij}\left(  \cdot\right)  \right]  $ vanish at a point of
$\widetilde{B}_{r}$; then (\ref{4.4}) in Proposition
\ref{basic properties for holder norm} implies%
\[
\left\vert \left[  \widetilde{a}_{ij}(\xi_{0})-\widetilde{a}_{ij}\left(
\cdot\right)  \right]  \,\widetilde{X}_{i}\widetilde{X}_{j}u\right\vert
_{C_{\widetilde{X}}^{\alpha}(\widetilde{B}_{r})}\leq cr^{\alpha}\left\vert
\widetilde{a}_{ij}\right\vert _{C_{\widetilde{X}}^{\alpha}(\widetilde{B}_{r}%
)}\cdot\left\vert \widetilde{X}_{i}\widetilde{X}_{j}u\right\vert
_{C_{\widetilde{X}}^{\alpha}(\widetilde{B}_{r})},
\]
while obviously%
\[
\left\Vert \left[  \widetilde{a}_{ij}(\xi_{0})-\widetilde{a}_{ij}\left(
\cdot\right)  \right]  \,\widetilde{X}_{i}\widetilde{X}_{j}u\right\Vert
_{L^{\infty}(\widetilde{B}_{r})}\leq cr^{\alpha}\left\vert \widetilde{a}%
_{ij}\right\vert _{C_{\widetilde{X}}^{\alpha}(\widetilde{B}_{r})}%
\cdot\left\Vert \widetilde{X}_{i}\widetilde{X}_{j}u\right\Vert _{L^{\infty
}(\widetilde{B}_{r})}.
\]
This allows, for $r$ small enough, to get
\begin{equation}
\left\Vert \widetilde{X}_{m}\widetilde{X}_{l}u\right\Vert _{C_{\widetilde{X}%
}^{\alpha}(\widetilde{B}_{r})}\leq c\left\{  \left\Vert \widetilde
{\mathcal{L}}u\right\Vert _{C_{\widetilde{X}}^{\alpha}(\widetilde{B}_{r}%
)}+\overset{q}{\sum\limits_{k=1}}\left\Vert \widetilde{X}_{k}u\right\Vert
_{C_{\widetilde{X}}^{\alpha}(\widetilde{B}_{r})}+\left\Vert u\right\Vert
_{C_{\widetilde{X}}^{\alpha}(\widetilde{B}_{r})}\right\}  \label{6.16}%
\end{equation}
for some $c$ also depending on $\left\vert \widetilde{a}_{ij}\right\vert
_{C_{\widetilde{X}}^{\alpha}(\widetilde{B}_{r})}.$ From the equation
(\ref{basic equation}) we also read%
\begin{equation}
\left\Vert \widetilde{X}_{0}u\right\Vert _{C_{\widetilde{X}}^{\alpha
}(\widetilde{B}_{r})}\leq\left\Vert \widetilde{\mathcal{L}}u\right\Vert
_{C_{\widetilde{X}}^{\alpha}(\widetilde{B}_{r})}+\overset{q}{c\sum
\limits_{k,h=1}}\left\Vert \widetilde{X}_{k}\widetilde{X}_{h}u\right\Vert
_{C_{\widetilde{X}}^{\alpha}(\widetilde{B}_{r})}. \label{x0u  schaulder}%
\end{equation}
By (\ref{6.16}) and (\ref{x0u schaulder}) we get, for $r$ small enough,%
\begin{equation}
\left\Vert u\right\Vert _{C_{\widetilde{X}}^{2,\alpha}(\widetilde{B}_{r})}\leq
c\left\{  \left\Vert \widetilde{\mathcal{L}}u\right\Vert _{C_{\widetilde{X}%
}^{\alpha}(\widetilde{B}_{r})}+\left\Vert u\right\Vert _{C_{\widetilde{X}%
}^{\alpha}(\widetilde{B}_{r})}+\overset{q}{\sum\limits_{k=1}}\left\Vert
\widetilde{X}_{k}u\right\Vert _{C_{\widetilde{X}}^{\alpha}(\widetilde{B}_{r}%
)}\right\}  . \label{holder norm for u of 2}%
\end{equation}

Next, we want to get rid of the term $\left\Vert \widetilde{X}_{k}u\right\Vert
_{C_{\widetilde{X}}^{\alpha}(\widetilde{B}_{r})}$ in the last inequality.
Taking only one derivative in the parametrix formula (\ref{representation2})
we have
\[
\widetilde{X}_{l}\left(  u\right)  =S\left(  \xi_{0}\right)  \left(
\widetilde{\mathcal{L}}u+\sum_{i,j=1}^{q}\left[  \widetilde{a}_{ij}(\xi
_{0})-\widetilde{a}_{ij}(\xi)\right]  \,\widetilde{X}_{i}\widetilde{X}%
_{j}u\right)  +T\left(  \xi_{0}\right)  u,
\]
where $S\left(  \xi_{0}\right)  ,T\left(  \xi_{0}\right)  $ are frozen
operators of type $1,0,$ respectively. Taking $C_{\widetilde{X}}^{\alpha
}(\widetilde{B}_{r})$ norms of both sides and applying Theorem
\ref{continuity for frozen operator on holder space}, we can write%
\[
\left\Vert \widetilde{X}_{l}u\right\Vert _{C_{\widetilde{X}}^{\alpha
}(\widetilde{B}_{r})}\leq c\left\{  \left\Vert \widetilde{\mathcal{L}%
}u\right\Vert _{C_{\widetilde{X}}^{\alpha}(\widetilde{B}_{r})}+\left\Vert
\sum_{i,j=1}^{q}\left[  \widetilde{a}_{ij}(\xi_{0})-\widetilde{a}_{ij}%
(\xi)\right]  \,\widetilde{X}_{i}\widetilde{X}_{j}u\right\Vert _{C_{\widetilde
{X}}^{\alpha}(\widetilde{B}_{r})}+\left\Vert u\right\Vert _{C_{\widetilde{X}%
}^{\alpha}(\widetilde{B}_{r})}\right\}
\]
and reasoning as above,
\begin{equation}
\left\Vert \widetilde{X}_{l}u\right\Vert _{C_{\widetilde{X}}^{\alpha
}(\widetilde{B}_{r})}\leq c\left\{  \left\Vert \widetilde{\mathcal{L}%
}u\right\Vert _{C_{\widetilde{X}}^{\alpha}(\widetilde{B}_{r})}+r^{\alpha
}\left\Vert \widetilde{X}_{i}\widetilde{X}_{j}u\right\Vert _{C_{\widetilde{X}%
}^{\alpha}(\widetilde{B}_{r})}+\left\Vert u\right\Vert _{C_{\widetilde{X}%
}^{\alpha}(\widetilde{B}_{r})}\right\}
\label{holder norm for one derivatives}%
\end{equation}
Inserting (\ref{holder norm for one derivatives}) in
(\ref{holder norm for u of 2}), for $r$ small enough we get
\begin{equation}
\left\Vert u\right\Vert _{C_{\widetilde{X}}^{2,\alpha}(\widetilde{B}_{r})}\leq
c\left\{  \left\Vert \widetilde{\mathcal{L}}u\right\Vert _{C_{\widetilde{X}%
}^{\alpha}(\widetilde{B}_{r})}+\left\Vert u\right\Vert _{C_{\widetilde{X}%
}^{\alpha}(\widetilde{B}_{r})}\right\}  . \label{6.18}%
\end{equation}

Finally, we want to replace the term $\left\Vert u\right\Vert _{C_{\widetilde
{X}}^{\alpha}(\widetilde{B}_{r})}$ with $\left\Vert u\right\Vert _{L^{\infty
}(\widetilde{B}_{r})}$ in the last inequality. To do this, we apply
(\ref{4.8}) in Proposition \ref{basic properties for holder norm} and write
\[
\left\Vert u\right\Vert _{C_{\widetilde{X}}^{\alpha}(\widetilde{B}_{r})}%
\leq\left\Vert u\right\Vert _{L^{\infty}(\widetilde{B}_{r})}+cr^{1-\alpha
}\left(  \overset{q}{\underset{i=1}{\sum}}\left\Vert \widetilde{X}%
_{i}u\right\Vert _{L^{\infty}(\widetilde{B}_{r})}+r\left\Vert \widetilde
{X}_{0}u\right\Vert _{L^{\infty}(\widetilde{B}_{r})}\right)  .
\]
substituting this in (\ref{6.18}), for $r$ small enough\ the term%
\[
\left(  \overset{q}{\underset{i=1}{\sum}}\left\Vert \widetilde{X}%
_{i}u\right\Vert _{L^{\infty}(\widetilde{B}_{r})}+r\left\Vert \widetilde
{X}_{0}u\right\Vert _{L^{\infty}(\widetilde{B}_{r})}\right)
\]
can be taken to the left-hand side, to get%
\[
\left\Vert u\right\Vert _{C_{\widetilde{X}}^{2,\alpha}(\widetilde{B}_{r})}\leq
c\left\{  \left\Vert \widetilde{\mathcal{L}}u\right\Vert _{C_{\widetilde{X}%
}^{\alpha}(\widetilde{B}_{r})}+\left\Vert u\right\Vert _{L^{\infty}%
(\widetilde{B}_{r})}\right\}  ,
\]
so we are done.
\end{proof}

\subsection{Schauder estimates for nonvanishing
functions\label{subsec Schauder nonvanishing}}

The second step in the proof of Schauder estimates consists in establishing
\textit{a priori} estimates for functions non necessarily compactly supported:

\begin{theorem}
\label{schaulder estimate for usual function in ball}There exist $r_{0}<R_{0}$
and $c,\beta>0$ (with $R_{0}$ as in Theorem
\ref{schaulder estimate for compactly function in ball}) such that, for every
$u\in C_{\widetilde{X}}^{2,\alpha}\left(  \widetilde{B}(\overline{\xi}%
,r_{0})\right)  ,$ $0<t<s<r_{0},_{{}}$%
\[
\left\Vert u\right\Vert _{C_{\widetilde{X}}^{2,\alpha}\left(  \widetilde
{B}\left(  \overline{\xi},t\right)  \right)  }\leq\frac{c}{(s-t)^{\beta}%
}\left\{  \left\Vert \widetilde{\mathcal{L}}u\right\Vert _{C_{\widetilde{X}%
}^{\alpha}\left(  \widetilde{B}\left(  \overline{\xi},s\right)  \right)
}+\left\Vert u\right\Vert _{L^{\infty}\left(  \widetilde{B}\left(
\overline{\xi},s\right)  \right)  }\right\}  ,
\]
where $r_{0},c$ depend on $R,\left\{  \widetilde{X}_{i}\right\}  _{i=1}%
^{q},\alpha,\mu,\left\Vert \widetilde{a}_{ij}\right\Vert _{C_{\widetilde{X}%
}^{\alpha}(\widetilde{B}\left(  \overline{\xi},R\right)  )}$; $\beta$ depends
on $\left\{  \widetilde{X}_{i}\right\}  _{i=0}^{q}$ and $\alpha$.
\end{theorem}

As in \cite{bb4}, this result relies on interpolation inequalities for
$C_{\widetilde{X}}^{k,\alpha}$ norms and the use of suitable cutoff function.
The following result can be proved as \cite[Lemma 6.2]{bb4}, by the results in
Proposition \ref{basic properties for holder norm}.

\begin{lemma}
[cutoff functions]\label{Lemma schauder cutoff}For any $0<\rho<r,$ $\xi
\in\widetilde{B}\left(  \overline{\xi},R\right)  $ there exists $\varphi\in
C_{0}^{\infty}\left(  \mathbb{R}^{N}\right)  $ with the following properties:

\begin{description}
\item[i)] $0\leqslant\varphi\leqslant1$, $\varphi\equiv1$ on $\widetilde
{B}\left(  \xi,\rho\right)  $ and $\operatorname*{sprt}\varphi\subseteq
\widetilde{B}\left(  \xi,r\right)  $;

\item[ii)] for $i,j=1,2,...,q,$%
\begin{align}
\left\vert \widetilde{X}_{i}\varphi\right\vert  &  \leqslant\frac{c}{r-\rho
}\label{cutoff sup}\\
\left\vert \widetilde{X}_{0}\varphi\right\vert ,\left\vert \widetilde{X}%
_{i}\widetilde{X}_{j}\varphi\right\vert  &  \leqslant\frac{c}{\left(
r-\rho\right)  ^{2}}\nonumber
\end{align}

\item[iii)] For any $f\in C_{\widetilde{X}}^{\alpha}\left(  \widetilde
{B}\left(  \overline{\xi},R\right)  \right)  $, and $r-\rho$ small enough,%
\begin{align}
\left\Vert f\,\widetilde{X}_{i}\varphi\right\Vert _{C_{\widetilde{X}}^{\alpha
}\left(  \widetilde{B}\left(  \overline{\xi},R\right)  \right)  }  &
\leqslant\frac{c}{\left(  r-\rho\right)  ^{2}}\left\Vert f\,\right\Vert
_{C_{\widetilde{X}}^{\alpha}\left(  \widetilde{B}\left(  \overline{\xi
},R\right)  \right)  }\label{cutoff C-alfa}\\
\left\Vert f\,\widetilde{X}_{0}\varphi\right\Vert _{C_{\widetilde{X}}^{\alpha
}\left(  \widetilde{B}\left(  \overline{\xi},R\right)  \right)  },\left\Vert
f\widetilde{X}_{i}\widetilde{X}_{j}\varphi\right\Vert _{C_{\widetilde{X}%
}^{\alpha}\left(  \widetilde{B}\left(  \overline{\xi},R\right)  \right)  }  &
\leqslant\frac{c}{\left(  r-\rho\right)  ^{3}}\left\Vert f\,\right\Vert
_{C_{\widetilde{X}}^{\alpha}\left(  \widetilde{B}\left(  \overline{\xi
},R\right)  \right)  }.\nonumber
\end{align}

\end{description}
\end{lemma}

We will write
\[
\widetilde{B}_{\rho}\left(  \xi\right)  \prec\varphi\prec\widetilde{B}%
_{r}\left(  \xi\right)
\]
to indicate that $\varphi$ satisfies all the previous properties.

Next, let us state the following:

\begin{proposition}
[Interpolation inequality for test functions]%
\label{Interpolation inequality for test functions}Let%
\[
H=\sum_{i=1}^{q}\widetilde{X}_{i}^{2}+\widetilde{X}_{0}%
\]
and let $\widetilde{B}\left(  \overline{\xi},R\right)  $ be as before. Then
for every $\alpha\in(0,1)$, there exist constants $\gamma\geq1$ and $c>0$,
depending on $\alpha,R$ and $\left\{  \widetilde{X}_{i}\right\}  $, such that
for every $\varepsilon\in\left(  0,1\right)  $ and every $f\in C_{0}^{\infty
}(\widetilde{B}\left(  \overline{\xi},R/2\right)  ),$%
\begin{equation}
\left\Vert \widetilde{X}_{l}f\right\Vert _{C_{\widetilde{X}}^{\alpha}\left(
\widetilde{B}\left(  \overline{\xi},R/2\right)  \right)  }\leq\varepsilon
\left\Vert Hf\right\Vert _{C_{\widetilde{X}}^{\alpha}\left(  \widetilde
{B}\left(  \overline{\xi},R/2\right)  \right)  }+\frac{c}{\varepsilon^{\gamma
}}\left\Vert f\right\Vert _{L^{\infty}\left(  \widetilde{B}\left(
\overline{\xi},R/2\right)  \right)  } \label{interpolation one}%
\end{equation}
for $l=1,2,\cdots,q;$ moreover, we have%
\begin{equation}
\left\Vert Df\right\Vert _{C_{\widetilde{X}}^{\alpha}\left(  \widetilde
{B}\left(  \overline{\xi},R/2\right)  \right)  }\leq\varepsilon\left\Vert
\widetilde{\mathcal{L}}f\right\Vert _{C_{\widetilde{X}}^{\alpha}\left(
\widetilde{B}\left(  \overline{\xi},R/2\right)  \right)  }+\frac
{c}{\varepsilon^{\gamma}}\left\Vert f\right\Vert _{L^{\infty}\left(
\widetilde{B}\left(  \overline{\xi},R/2\right)  \right)  },
\label{deferential  operator of degree one}%
\end{equation}
where $D$ is any vector field of local degree $\leq1$.
\end{proposition}

To prove Proposition \ref{Interpolation inequality for test functions}, we
need the following

\begin{lemma}
\label{Lemma interp C_alfa}Let $P\left(  \xi_{0}\right)  $ be a frozen
operator of type $\lambda\geq1$ over $\widetilde{B}\left(  \overline{\xi
},R\right)  $ and $\alpha\in(0,1)$. Then there exist positive constants
$\gamma>1$ and $c$, depending on $\alpha,\mu$ and $\left\{  \widetilde{X}%
_{i}\right\}  $, such that for every $f\in C_{0}^{\infty}(\widetilde{B}\left(
\overline{\xi},R\right)  )$ and $\varepsilon\in\left(  0,1\right)  $%
\begin{equation}
\left\Vert PHf\right\Vert _{C_{\widetilde{X}}^{\alpha}\left(  \widetilde
{B}\left(  \overline{\xi},R\right)  \right)  }\leq\varepsilon\left\Vert
Hf\right\Vert _{C_{\widetilde{X}}^{\alpha}\left(  \widetilde{B}\left(
\overline{\xi},R\right)  \right)  }+\frac{c}{\varepsilon^{\gamma}}\left\Vert
f\right\Vert _{L^{\infty}\left(  \widetilde{B}\left(  \overline{\xi},R\right)
\right)  }. \label{interp C-alfa}%
\end{equation}

\end{lemma}

\begin{remark}
\label{Remark interp C-alfa}As will be clear from the proof,
(\ref{interp C-alfa}) still holds if $H$ is replaced by any differential
operator of weight two, like $\widetilde{X}_{i}\widetilde{X}_{j}$ or
$\widetilde{X}_{0}$.
\end{remark}

\begin{proof}
[Proof of the Lemma]This proof is adapted from \cite[Lemma 7.2]{bb4}. Let%
\[
PHf\left(  \xi\right)  =\int_{\widetilde{B}\left(  \overline{\xi},R\right)
}k(\xi,\eta)Hf\left(  \eta\right)  d\eta,
\]
where $k$ is a frozen kernel of type $\lambda\geq1$, and let $\varphi
_{\varepsilon}$ be a cutoff function such that $\widetilde{B}_{\varepsilon
/2}(\xi)\prec\varphi_{\varepsilon}\prec\widetilde{B}_{\varepsilon}(\xi)$, for
$\varepsilon\in\left(  0,1\right)  .$ We split $PH$ as follows: for $\xi
\in\widetilde{B}\left(  \overline{\xi},R\right)  $%
\begin{align*}
PHf(\xi)  &  =\int_{\widetilde{B}\left(  \overline{\xi},R\right)  ,\rho
(\xi,\eta)>\frac{_{\varepsilon}}{2}}k(\xi,\eta)[1-\varphi_{\varepsilon}%
(\eta)]Hf(\eta)d\eta\\
&  +\int_{\widetilde{B}\left(  \overline{\xi},R\right)  ,\rho(\xi,\eta
)\leq\varepsilon}k(\xi,\eta)Hf(\eta)\varphi_{\varepsilon}(\eta)d\eta\\
&  =I(\xi)+II(\xi).
\end{align*}
Then
\[
I(\xi)=\int_{\widetilde{B}\left(  \overline{\xi},R\right)  ,\rho(\xi
,\eta)>\frac{_{\varepsilon}}{2}}H^{T}\left(  k(\xi,\cdot)[1-\varphi
_{\varepsilon}(\cdot)]\right)  \left(  \eta\right)  f(\eta)d\eta.
\]

Let $h^{\varepsilon}(\xi,\eta)=H^{T}\left(  k(\xi,\cdot)[1-\varphi
_{\varepsilon}(\cdot)]\right)  \left(  \eta\right)  $. Since $k$ is a frozen
kernel of type $\lambda$, there exist $c>0,\gamma>1$, such that
\[
\left\vert h^{\varepsilon}(\xi,\eta)\right\vert +\left\vert \widetilde{X}%
_{0}h^{\varepsilon}(\xi,\eta)\right\vert +\left\vert \sum\widetilde{X}%
_{i}h^{\varepsilon}(\xi,\eta)\right\vert \leq c\varepsilon^{-\gamma}.
\]
By definition of frozen kernel, the function $\xi\longmapsto h^{\varepsilon
}(\xi,\eta)$ is compactly supported in $\widetilde{B}\left(  \overline{\xi
},R\right)  $ for any $\eta\in\widetilde{B}\left(  \overline{\xi},R\right)  ,$
hence by (\ref{4.7}) in Proposition \ref{basic properties for holder norm}, it
follows that%
\[
\left\vert h^{\varepsilon}(\xi_{1},\eta)-h^{\varepsilon}(\xi_{2}%
,\eta)\right\vert \leq c_{R}d_{\widetilde{X}}(\xi_{1},\xi_{2})\varepsilon
^{-\gamma}\leq c_{R}\rho(\xi_{1},\xi_{2})\varepsilon^{-\gamma}%
\]
for any $\xi_{1},\xi_{2}\in\widetilde{B}\left(  \overline{\xi},R\right)  ,$
and therefore%
\begin{align*}
\left\vert I(\xi_{1})-I(\xi_{2})\right\vert  &  \leq\int\left\vert
h^{\varepsilon}(\xi_{1},\eta)-h^{\varepsilon}(\xi_{2},\eta)\right\vert
\left\vert f(\eta)\right\vert d\eta\\
&  \leq c_{R}\varepsilon^{-\gamma}\rho(\xi_{1},\xi_{2})\left\vert
\widetilde{B}_{R}\right\vert \left\Vert f\right\Vert _{L^{\infty}%
(\widetilde{B}_{R})}.
\end{align*}
Also, since
\[
\left\vert I(\xi)\right\vert \leq\int_{\widetilde{B}\left(  \overline{\xi
},R\right)  ,\rho(\xi,\eta)>\frac{_{\varepsilon}}{2}}c\varepsilon^{-\gamma
}\left\vert f(\eta)\right\vert d\eta\leq c\varepsilon^{-\gamma}\left\vert
\widetilde{B}_{R}\right\vert \left\Vert f\right\Vert _{L^{\infty}\left(
\widetilde{B}\left(  \overline{\xi},R\right)  \right)  },
\]
we obtain
\[
\left\Vert I(\xi)\right\Vert _{C_{\widetilde{X}}^{\alpha}\left(  \widetilde
{B}\left(  \overline{\xi},R\right)  \right)  }\leq c\varepsilon^{-\gamma
}\left\Vert f\right\Vert _{L^{\infty}\left(  \widetilde{B}\left(
\overline{\xi},R\right)  \right)  }\text{ for any }\alpha\in(0,1).
\]

Now, let us consider $II(\xi)$, and let
\[
k_{\varepsilon}(\xi,\eta)=k(\xi,\eta)\varphi_{\varepsilon}(\eta).
\]
By the properties of frozen kernels of type 1, keeping into account the
support of $k_{\varepsilon}$ and applying again (\ref{4.7}) in Proposition
\ref{basic properties for holder norm}, we can say that for any fixed
$\delta\in(0,1)$, the kernel $k_{\varepsilon}(\xi,\eta)$\ satisfies the
following standard estimates of fractional integral kernels (see
\S \ \ref{subsec locally hom space}):%
\begin{align}
\left\vert k_{\varepsilon}(\xi,\eta)\right\vert  &  \leq c\rho(\xi,\eta
)^{1-Q}\leq c\varepsilon^{\delta}\rho(\xi,\eta)^{1-\delta-Q},\label{frac 1}\\
\left\vert k_{\varepsilon}(\xi,\eta)-k_{\varepsilon}(\xi_{1},\eta)\right\vert
&  \leq c\frac{\rho(\xi,\xi_{1})}{\rho(\xi,\eta)^{Q}}\leq c\varepsilon
^{\delta}\rho(\xi,\eta)^{-\delta-Q}\rho(\xi,\xi_{1}) \label{frac 2}%
\end{align}
for $\rho(\xi,\eta)\geq2\rho(\xi,\xi_{1})$. Therefore, by Theorem
\ref{Thm frac C^eta} and Remark \ref{remark linear constants} in
\S \ \ref{subsec locally hom space},
\[
\left\Vert II\right\Vert _{C_{\widetilde{X}}^{\alpha}\left(  \widetilde
{B}\left(  \overline{\xi},R\right)  \right)  }\leq c\varepsilon^{\delta
}\left\Vert Hf\right\Vert _{C_{\widetilde{X}}^{\alpha}\left(  \widetilde
{B}\left(  \overline{\xi},R\right)  \right)  }%
\]
for any $\alpha<1-\delta$. We conclude that for every $\alpha\in(0,1)$ there
exist $\delta,$ $\gamma,c>0$ such that%
\[
\left\Vert PHf\right\Vert _{C_{\widetilde{X}}^{\alpha}\left(  \widetilde
{B}\left(  \overline{\xi},R\right)  \right)  }\leq\varepsilon^{\delta
}\left\Vert Hf\right\Vert _{C_{\widetilde{X}}^{\alpha}\left(  \widetilde
{B}\left(  \overline{\xi},R\right)  \right)  }+\frac{c}{\varepsilon^{\gamma}%
}\left\Vert f\right\Vert _{L^{\infty}\left(  \widetilde{B}\left(
\overline{\xi},R\right)  \right)  },
\]
which implies the lemma.
\end{proof}

\bigskip

\begin{proof}
[Proof of Proposition \ref{Interpolation inequality for test functions}]By
Theorem \ref{parametrix}, we can write
\[
af=PHf(\xi)+Sf,
\]
where $P,S$ are frozen operators of type $2$ and $1$, respectively, over
$\widetilde{B}\left(  \overline{\xi},R\right)  $. More precisely, they should
be called \textquotedblleft constant kernels of type 2 and 1\textquotedblright%
, since they satisfy the definition of frozen kernels with the matrix
$\left\{  \widetilde{a}_{ij}\left(  \xi_{0}\right)  \right\}  $ replaced by
the identity matrix.

If we assume $a=1$ on $\widetilde{B}\left(  \overline{\xi},R/2\right)  $,
then, for $f\in C_{0}^{\infty}(\widetilde{B}\left(  \overline{\xi},R/2\right)
)$ we obtain
\begin{equation}
f=PHf(\xi)+Sf \label{7.1}%
\end{equation}
and therefore, by Theorem \ref{main theorem},%
\begin{equation}
\widetilde{X}_{i}f=S_{1}Hf(\xi)+Tf, \label{7.2}%
\end{equation}
where $S_{1},T$ are frozen operators of type $1$ and $0,$ respectively.
Substituting (\ref{7.1}) in (\ref{7.2}) yields%
\[
\widetilde{X}_{i}f=S_{1}Hf(\xi)+TPHf+TSf
\]
and therefore, by Theorem \ref{continuity for frozen operator on holder space}
and Lemma \ref{Lemma interp C_alfa}%
\begin{align}
\left\Vert \widetilde{X}_{i}f\right\Vert _{\alpha}  &  \leq\left\Vert
S_{1}Hf\right\Vert _{\alpha}+\left\Vert TPHf\right\Vert _{\alpha}+\left\Vert
TSf\right\Vert _{\alpha}\nonumber\\
&  \leq\left\Vert S_{1}Hf\right\Vert _{\alpha}+c\left\{  \left\Vert
PHf\right\Vert _{\alpha}+\left\Vert Sf\right\Vert _{\alpha}\right\}
\nonumber\\
&  \leq c\left\{  \varepsilon\left\Vert Hf\right\Vert _{\alpha}+\varepsilon
^{-\gamma}\left\Vert f\right\Vert _{\infty}+\left\Vert Sf\right\Vert _{\alpha
}\right\}  \label{interp epsi}%
\end{align}
where all the norms are taken over $\widetilde{B}\left(  \overline{\xi
},R/2\right)  $. We end the proof by showing that for an operator $S$ of type
$1$,
\[
\left\Vert Sf\right\Vert _{\alpha}\leq c\left\Vert f\right\Vert _{L^{\infty}%
},
\]
which by (\ref{interp epsi}) will complete the proof of the first inequality
in the proposition.

Indeed, if
\[
Sf(\xi)=\int\nolimits_{\widetilde{B}_{R}}k(\xi,\eta)f(\eta)d\eta.
\]
We have%
\begin{equation}
\left\vert Sf(\xi_{1})-Sf(\xi_{2})\right\vert \leq\left\Vert f\right\Vert
_{L^{\infty}(\widetilde{B}_{R})}\int\nolimits_{\widetilde{B}\left(
\overline{\xi},R\right)  }\left\vert k(\xi_{1},\eta)-k(\xi_{2},\eta
)\right\vert d\eta. \label{Sf}%
\end{equation}
Moreover,
\begin{align*}
\int\nolimits_{\widetilde{B}_{R}}\left\vert k(\xi_{1},\eta)-k(\xi_{2}%
,\eta)\right\vert d\eta &  =\int\nolimits_{\widetilde{B}\left(  \overline{\xi
},R\right)  ,\rho(\xi_{1},\eta)>M\rho(\xi_{1},\xi_{2})}\left\vert k(\xi
_{1},\eta)-k(\xi_{2},\eta)\right\vert d\eta\\
&  +\int\nolimits_{\widetilde{B}\left(  \overline{\xi},R\right)  ,\rho(\xi
_{1},\eta)\leq M\rho(\xi_{1},\xi_{2})}\left\vert k(\xi_{1},\eta)-k(\xi
_{2},\eta)\right\vert d\eta\\
&  \equiv I+II.
\end{align*}
By (\ref{frac 2}),%
\begin{align*}
I  &  \leq\int\nolimits_{\rho(\xi_{1},\eta)>M\rho(\xi_{1},\xi_{2})}\frac
{c}{\rho(\xi_{1},\eta)^{Q-1}}\frac{\rho(\xi_{1},\xi_{2})}{\rho(\xi_{1},\eta
)}d\eta\\
&  =\rho(\xi_{1},\xi_{2})^{\alpha}\int\nolimits_{\rho(\xi_{1},\eta)>M\rho
(\xi_{1},\xi_{2})}\frac{\rho(\xi_{1},\eta)^{1-\alpha}}{\rho(\xi_{1},\eta)^{Q}%
}\frac{\rho(\xi_{1},\xi_{2})^{1-\alpha}}{\rho(\xi_{1},\eta)^{1-\alpha}}d\eta\\
&  \leq c\rho(\xi_{1},\xi_{2})^{\alpha}\int\nolimits_{\widetilde{B}_{R}}%
\frac{\rho(\xi_{1},\eta)^{1-\alpha}}{\rho(\xi_{1},\eta)^{Q}}d\eta\\
&  \leq c\rho(\xi_{1},\xi_{2})^{\alpha}R^{1-\alpha},
\end{align*}
where in the last inequality we have used the following standard computation
(which will be useful also other times):%
\begin{equation}
\int\nolimits_{\widetilde{B}\left(  \overline{\xi},R\right)  ,\rho(\xi
_{1},\eta)<r}\frac{d\eta}{\rho(\xi_{1},\eta)^{Q-\beta}}\leq cr^{\beta}\text{
for any }\xi_{1}\in\widetilde{B}\left(  \overline{\xi},R\right)
\label{onions}%
\end{equation}

As to $II$, by (\ref{frac 1}),%
\[
II\leq\int\nolimits_{\rho(\xi_{1},\eta)\leq M\rho(\xi_{1},\xi_{2})}\left\vert
k(\xi_{1},\eta)\right\vert d\eta+\int\nolimits_{\rho(\xi_{1},\eta)\leq
M\rho(\xi_{1},\xi_{2})}\left\vert k(\xi_{2},\eta)\right\vert d\eta
\]
since there exists $M_{1}>0$ such that if $\rho(\xi_{1},\eta)\leq M\rho
(\xi_{1},\xi_{2})$ then $\rho(\xi_{2},\eta)\leq M_{1}\rho(\xi_{1},\xi_{2})$,
\[
\leq c\left\{  \int\nolimits_{\rho(\xi_{1},\eta)\leq M\rho(\xi_{1},\xi_{2}%
)}\frac{1}{\rho(\xi_{1},\eta)^{Q-1}}d\eta+\int\nolimits_{\rho(\xi_{2}%
,\eta)\leq M_{1}\rho(\xi_{1},\xi_{2})}\frac{1}{\rho(\xi_{2},\eta)^{Q-1}}%
d\eta\right\}
\]
again by (\ref{onions})%
\[
\leq c\rho\left(  \xi_{1},\xi_{2}\right)  \leq c\rho\left(  \xi_{1},\xi
_{2}\right)  ^{\alpha}R^{1-\alpha}.
\]
Hence for every $\alpha\in(0,1),$%
\[
\int\nolimits_{\widetilde{B}_{R}}\left\vert k(\xi_{1},\eta)-k(\xi_{2}%
,\eta)\right\vert d\eta\leq c_{\alpha}\rho(\xi_{1},\xi_{2})^{\alpha
}R^{1-\alpha}%
\]
and, by (\ref{Sf}),%
\[
\left\vert Sf\right\vert _{\alpha}\leq c\left\Vert f\right\Vert _{L^{\infty}%
}.
\]
Moreover,%
\begin{align*}
\left\vert Sf(\xi)\right\vert  &  \leq\int\nolimits_{\widetilde{B}_{R}%
}\left\vert k(\xi,\eta)f(\eta)\right\vert d\eta\\
&  \leq\left\Vert f\right\Vert _{L^{\infty}}\int\nolimits_{\rho(\xi,\eta)\leq
cR}\frac{c}{\rho(\xi,\eta)^{Q-1}}d\eta\leq cR\left\Vert f\right\Vert
_{L^{\infty}},
\end{align*}
hence
\[
\left\Vert Sf\right\Vert _{\alpha}\leq c\left\Vert f\right\Vert _{L^{\infty}%
}.
\]

This completes the proof of (\ref{interpolation one}). A similar argument
gives (\ref{deferential operator of degree one}).
\end{proof}

\begin{theorem}
[Interpolation inequality]\label{interpolation inequality} There exist
positive constants $c,\gamma$ and $r_{1}<R$ such that for any $u\in
C_{\widetilde{X}}^{2,\alpha}(\widetilde{B}\left(  \overline{\xi},r_{1}\right)
)$, $0<\rho<r_{1}$, $0<\delta<1/3,$%
\[
\left\Vert \widetilde{D}u\right\Vert _{C_{\widetilde{X}}^{\alpha}%
(\widetilde{B}\left(  \overline{\xi},\rho\right)  )}\leq\delta\sum_{i=1}%
^{q}\left\Vert \widetilde{D}^{2}u\right\Vert _{C_{\widetilde{X}}^{\alpha
}\left(  \widetilde{B}\left(  \overline{\xi},r_{1}\right)  \right)  }+\frac
{c}{\delta^{\gamma}\left(  r_{1}-\rho\right)  ^{2\gamma}}\left\Vert
u\right\Vert _{L^{\infty}\left(  \widetilde{B}\left(  \overline{\xi}%
,r_{1}\right)  \right)  }%
\]
where%
\[
\left\Vert \widetilde{D}u\right\Vert \equiv\sum_{i=1}^{q}\left\Vert
\widetilde{X}_{i}u\right\Vert \text{ and }\left\Vert \widetilde{D}%
^{2}u\right\Vert \equiv\sum_{i,j=1}^{q}\left\Vert \widetilde{X}_{i}%
\widetilde{X}_{i}u\right\Vert +\left\Vert \widetilde{X}_{0}u\right\Vert .
\]
The constants $c,r_{1},\gamma$ depend on $\alpha,\left\{  \widetilde{X}%
_{i}\right\}  ;\gamma$ is as in Proposition
\ref{Interpolation inequality for test functions}.
\end{theorem}

\begin{proof}
The proof can be carried out exactly as in \cite[Proposition 7.4]{bb4},
exploiting the properties of cutoff functions (Lemma
\ref{Lemma schauder cutoff}), the interpolation inequality for test functions
(Proposition \ref{Interpolation inequality for test functions}) and
(\ref{4.8}) in Proposition \ref{basic properties for holder norm}.
\end{proof}

\bigskip

We are now in position for the main goal of this subsection:

\begin{proof}
[Proof of Theorem \ref{schaulder estimate for usual function in ball}]This
proof can now be carried out exactly like in \cite[Theorem 5.3]{bb4},
exploiting: Schauder estimates for functions with small support (Theorem
\ref{schaulder estimate for compactly function in ball}), the properties of
H\"{o}lder continuous functions contained in (\ref{4.8}), (\ref{4.3}),
(\ref{4.11}), the properties of cutoff functions (Lemma
\ref{Lemma schauder cutoff}) and the interpolation inequalities contained in
Theorem \ref{interpolation inequality} and
(\ref{deferential operator of degree one}).
\end{proof}

\subsection{Schauder estimates in the original
variables\label{subsec schauder original}}

Let's now prove Theorem \ref{schauder estimate}. We finally come back to our
original context, which we are going to recall. We have a bounded domain
$\Omega$ where our vector fields and coefficients are defined, and a fixed
subdomain $\Omega^{\prime}\Subset\Omega.$ Fix $\overline{x}\in\Omega^{\prime}$
and $R$ such that in $B\left(  \overline{x},R\right)  \subset\Omega$ all the
construction of the previous two subsections (lifting to $\widetilde{B}\left(
\overline{\xi},R\right)  $ and so on) can be performed. Let $r_{0}$ be as in
Theorem \ref{schaulder estimate for usual function in ball}. To begin with, we
want to prove Schauder estimates for functions $u\in C_{X}^{2,\alpha}\left(
B\left(  \overline{x},r_{0}\right)  \right)  .$ By Theorem
\ref{original holder norm and lifted holder norm} we know that the function
$\widetilde{u}\left(  x,h\right)  =u\left(  x\right)  $ belongs to
$C_{\widetilde{X}}^{2,\alpha}\left(  B\left(  \overline{\xi},r_{0}\right)
\right)  $, so we can apply to $\widetilde{u}$ Schauder estimates contained in
Theorem \ref{schaulder estimate for usual function in ball}. Combining this
fact with the two estimates in Theorem
\ref{original holder norm and lifted holder norm} and choosing $t,s$ such that%
\[
r_{0}>s>t>0\text{ and }s-t=r_{0}-s,
\]
we get, for some exponent $\omega>2$%
\begin{align}
\left\Vert u\right\Vert _{C_{X}^{2,\alpha}(B\left(  \overline{x},s\right)  )}
&  \leq\frac{c}{\left(  s-t\right)  ^{2}}\left\Vert \widetilde{u}\right\Vert
_{C_{\widetilde{X}}^{2,\alpha}\left(  \widetilde{B}\left(  \overline{\xi
},t\right)  \right)  }\label{hold small ball}\\
&  \leq\frac{c}{\left(  r_{0}-t\right)  ^{\omega}}\left(  \left\Vert
\widetilde{\mathcal{L}}\widetilde{u}\right\Vert _{C_{\widetilde{X}}^{\alpha
}\left(  \widetilde{B}\left(  \overline{\xi},r_{0}\right)  \right)
}+\left\Vert \widetilde{u}\right\Vert _{L^{\infty}\left(  \widetilde{B}\left(
\overline{\xi},r_{0}\right)  \right)  }\right) \nonumber\\
&  \leq\frac{c}{\left(  r_{0}-s\right)  ^{\omega}}\left(  \left\Vert
\mathcal{L}u\right\Vert _{C_{X}^{\alpha}\left(  B\left(  \overline{x}%
,r_{0}\right)  \right)  }+\left\Vert u\right\Vert _{L^{\infty}(B\left(
\overline{x},r_{0}\right)  )}\right) \nonumber
\end{align}
since $\widetilde{\mathcal{L}}\widetilde{u}=\widetilde{\left(  \mathcal{L}%
u\right)  }.$

Next, let us choose a family of balls $B\left(  x_{i},r_{0}\right)  $ in
$\Omega$ such that%
\[
\Omega^{\prime}\subset%
{\displaystyle\bigcup\limits_{i=1}^{k}}
B\left(  x_{i},r_{0}/2\right)  \subset%
{\displaystyle\bigcup\limits_{i=1}^{k}}
B\left(  x_{i},r_{0}\right)  \subset\Omega.
\]
Then by Proposition \ref{basic properties for holder norm} (v) and
(\ref{hold small ball}), with $s=r_{0}/2,$
\begin{align*}
\left\Vert u\right\Vert _{C_{X}^{2,\alpha}(\Omega^{\prime})}  &
\leq\left\Vert u\right\Vert _{C_{X}^{2,\alpha}(\cup B\left(  x_{i}%
,r_{0}/2\right)  )}\leq c\sum_{i=1}^{k}\left\Vert u\right\Vert _{C_{X}%
^{2,\alpha}(B\left(  x_{i},r_{0}\right)  )}\\
&  \leq c\sum_{i=1}^{k}\left\{  \left\Vert \mathcal{L}u\right\Vert
_{C_{X}^{\alpha}(B\left(  x_{i},r_{0}\right)  )}+\left\Vert u\right\Vert
_{L^{\infty}(B\left(  x_{i},r_{0}\right)  )}\right\} \\
&  \leq c\left\{  \left\Vert \mathcal{L}u\right\Vert _{C_{X}^{\alpha}(\Omega
)}+\left\Vert u\right\Vert _{L^{\infty}(\Omega)}\right\}
\end{align*}
with $c$ also depending on $r_{0}$. Finally, let us note that the constant $c$
depends on the coefficients $a_{ij}$ through the norms
\[
\left\Vert \widetilde{a}_{ij}\right\Vert _{C_{\widetilde{X}}^{\alpha}\left(
\widetilde{B}\left(  \overline{\xi},R_{0}\right)  \right)  },
\]
which in turn are bounded by the norms%
\[
\left\Vert a_{ij}\right\Vert _{C_{X}^{\alpha}\left(  B\left(  \overline
{x},R_{0}\right)  \right)  }%
\]
(by Proposition \ref{original holder norm and lifted holder norm}), and hence
by $\left\Vert a_{ij}\right\Vert _{C_{X}^{\alpha}\left(  \Omega\right)  }$ (or
more precisely, by $\left\Vert a_{ij}\right\Vert _{C_{X}^{\alpha}\left(
\Omega^{\prime\prime}\right)  }$ for some $\Omega^{\prime\prime}$ such that
$\Omega^{\prime}\Subset\Omega^{\prime\prime}\Subset\Omega$). This completes
the proof of Theorem \ref{schauder estimate}.

\section{$L^{p}$ estimates\label{section L^p estimates}}

The logical structure of this section, as well as the general setting, is very
similar to that of the previous one. Here, following as close as possible the
strategy of \cite{bb2}, the proof will be still divided into three steps:
$L^{p}$ a-priori estimates in the space of lifted variables, for test
functions supported in a ball $\widetilde{B}\left(  \overline{\xi},r\right)  $
with $r$ small enough; then for any function in $S_{\widetilde{X}}%
^{2,p}\left(  \widetilde{B}\left(  \overline{\xi},r\right)  \right)  $ (not
necessarily vanishing at the boundary); then for any function in $S_{X}%
^{2,p}\left(  B\left(  \overline{x},r\right)  \right)  $, that is in the
original space.

Again, it is not restrictive to assume $a_{0}=1.$

The basic difference with the setting of Schauder estimates consists in the
fact that here we start with representation formulas where the
\textquotedblleft frozen\textquotedblright\ point has been finally unfrozen;
therefore now singular integrals with \emph{variable }kernels are involved,
together with their commutators with $VMO$ functions. This makes the singular
integral part of the theory more involved.

\subsection{$L^{p}$ estimates for functions with small
support\label{subsec L^p small support}}

\begin{theorem}
\label{theorem:sprt function regular}Let $\widetilde{B}\left(  \overline{\xi
},R\right)  $ be as in the previous section, and $p\in(1,\infty)$. There
exists $R_{0}<R$ such that for every $u\in C_{0}^{\infty}\left(  \widetilde
{B}\left(  \overline{\xi},R_{0}\right)  \right)  $,%
\begin{equation}
\left\Vert u\right\Vert _{S_{\widetilde{X}}^{2,p}\left(  \widetilde{B}\left(
\overline{\xi},R_{0}\right)  \right)  }\leq c\left\{  \left\Vert
\widetilde{\mathcal{L}}u\right\Vert _{L^{p}\left(  \widetilde{B}\left(
\overline{\xi},R_{0}\right)  \right)  }+\left\Vert u\right\Vert _{L^{p}\left(
\widetilde{B}\left(  \overline{\xi},R_{0}\right)  \right)  }\right\}
\label{sprt function regular}%
\end{equation}
for some constant $c$ depending on $\left\{  \widetilde{X}_{i}\right\}
_{i=0}^{q},p,\mu$\thinspace$,R;$ the number $R_{0}$ also depends on the local
$VMO$ moduli $\eta_{a_{ij},\Omega^{\prime},\Omega}^{\ast}.$
\end{theorem}

\begin{proof}
This theorem relies on the representation formula proved in Theorem
\ref{lemma representation for derivatives} and on the results about singular
integrals and commutators contained in Theorem \ref{iith:comm:optypezer}.

Let $u\in C_{0}^{\infty}\left(  \widetilde{B}\left(  \overline{\xi},r\right)
\right)  $ with $r<R$. Let us write the representation formula of Theorem
\ref{lemma representation for derivatives} choosing the cutoff function $a$
such that $a=1$ in $\widetilde{B}\left(  \overline{\xi},r\right)  $. Taking
$L^{p}$ norms of both sides of the formula we get (see also Remark
\ref{Remark easy representation}), for $p\in(1,\infty),$ any $m,l=1,2,...,q,$%
\begin{align*}
\left\Vert \widetilde{X}_{m}\widetilde{X}_{l}u\right\Vert _{L^{p}\left(
\widetilde{B}\left(  \overline{\xi},r\right)  \right)  }  &  \leq\left\Vert
T_{lm}\widetilde{\mathcal{L}}u\right\Vert _{L^{p}\left(  \widetilde{B}\left(
\overline{\xi},r\right)  \right)  }+\sum_{i,j=1}^{q}\left\Vert \left[
\widetilde{a}_{ij},T_{lm}\right]  \,\widetilde{X}_{i}\widetilde{X}%
_{j}u\right\Vert _{L^{p}\left(  \widetilde{B}\left(  \overline{\xi},r\right)
\right)  }+\\
&  +\sum_{k=1}^{q}\left\Vert T_{lm,k}\widetilde{X}_{k}u\right\Vert
_{L^{p}\left(  \widetilde{B}\left(  \overline{\xi},r\right)  \right)
}+\left\Vert T_{lm}^{0}u\right\Vert _{L^{p}\left(  \widetilde{B}\left(
\overline{\xi},r\right)  \right)  }%
\end{align*}
where all the $T_{lm},T_{lm,k},T_{lm}^{0}\ $are \emph{variable }operators of
type $0$ over $\widetilde{B}\left(  \overline{\xi},2r\right)  .$

We now apply Theorem \ref{iith:comm:optypezer}: there exists $c$ (depending on
$R$) and for every fixed $\varepsilon>0$ there exists $r<R$ such that for
every $u\in C_{0}^{\infty}\left(  \widetilde{B}\left(  \overline{\xi
},r\right)  \right)  $,%
\begin{align}
\left\Vert \widetilde{X}_{m}\widetilde{X}_{l}u\right\Vert _{L^{p}\left(
\widetilde{B}\left(  \overline{\xi},r\right)  \right)  }  &  \leq c\left\{
\left\Vert \widetilde{\mathcal{L}}u\right\Vert _{L^{p}\left(  \widetilde
{B}\left(  \overline{\xi},r\right)  \right)  }+\varepsilon\sum_{i,j=1}%
^{q}\left\Vert \widetilde{X}_{i}\widetilde{X}_{j}u\right\Vert _{L^{p}\left(
\widetilde{B}\left(  \overline{\xi},r\right)  \right)  }\right.
\label{norm of derivative of 2}\\
&  +\left.  \left\Vert u\right\Vert _{L^{p}\left(  \widetilde{B}\left(
\overline{\xi},r\right)  \right)  }+\sum\limits_{k}\left\Vert \widetilde
{X}_{k}u\right\Vert _{L^{p}\left(  \widetilde{B}\left(  \overline{\xi
},r\right)  \right)  }\right\}  .\nonumber
\end{align}

Now, let us come back to (\ref{representation2}) and take only one derivative
$\widetilde{X}_{l}$ (for $l=1,...,q$) of both sides; we find:%
\begin{align}
\widetilde{X}_{l}u  &  =\widetilde{X}_{l}P\left(  \xi_{0}\right)
\widetilde{\mathcal{L}}_{0}u+\widetilde{X}_{l}S\left(  \xi_{0}\right)
u\label{representation for one derivative}\\
&  =S_{l}\left(  \xi_{0}\right)  \widetilde{\mathcal{L}}_{0}+T\left(  \xi
_{0}\right)  u\nonumber
\end{align}
where $S_{l}\left(  \xi_{0}\right)  ,T\left(  \xi_{0}\right)  $ are frozen
operators of type $1$ and $0$, respectively. By (\ref{10.1}) we have
\[
\widetilde{X}_{l}u=S_{1,l}\widetilde{\mathcal{L}}u+\sum_{i,j=1}^{q}\left[
S_{1,l},\widetilde{a}_{ij}\right]  \widetilde{X}_{i}\widetilde{X}_{j}u+Tu,
\]
where $S_{1,l},$ $T$ are variable operators of type $1$ and $0$, respectively
(in particular, both can be seen as operators of type $0$). By Theorem
\ref{iith:comm:optypezer},%
\begin{equation}
\left\Vert \widetilde{X}_{l}u\right\Vert _{L^{p}\left(  \widetilde{B}\left(
\overline{\xi},r\right)  \right)  }\leq c\left\Vert \widetilde{\mathcal{L}%
}u\right\Vert _{L^{p}\left(  \widetilde{B}\left(  \overline{\xi},r\right)
\right)  }+\varepsilon\sum_{i,j=1}^{q}\left\Vert \widetilde{X}_{i}%
\widetilde{X}_{j}u\right\Vert _{L^{p}\left(  \widetilde{B}\left(
\overline{\xi},r\right)  \right)  }+c\left\Vert u\right\Vert _{L^{p}\left(
\widetilde{B}\left(  \overline{\xi},r\right)  \right)  }.
\label{norm of derivative of 1}%
\end{equation}
Finally, from the equation we can bound%
\begin{equation}
\left\Vert \widetilde{X}_{0}u\right\Vert _{L^{p}\left(  \widetilde{B}\left(
\overline{\xi},r\right)  \right)  }\leq c\sum_{i,j=1}^{q}\left\Vert
\widetilde{X}_{i}\widetilde{X}_{j}u\right\Vert _{L^{p}\left(  \widetilde
{B}\left(  \overline{\xi},r\right)  \right)  }+\left\Vert \widetilde
{\mathcal{L}}u\right\Vert _{L^{p}\left(  \widetilde{B}\left(  \overline{\xi
},r\right)  \right)  }. \label{norm X_0}%
\end{equation}
Combining (\ref{norm of derivative of 2}), (\ref{norm of derivative of 1}) and
(\ref{norm X_0}) we have, for $r$ small enough,%
\begin{equation}
\left\Vert u\right\Vert _{S_{\widetilde{X}}^{2,p}\left(  \widetilde{B}\left(
\overline{\xi},r\right)  \right)  }\leq c\left(  \left\Vert \widetilde
{\mathcal{L}}u\right\Vert _{L^{p}\left(  \widetilde{B}\left(  \overline{\xi
},r\right)  \right)  }+\left\Vert u\right\Vert _{L^{p}\left(  \widetilde
{B}\left(  \overline{\xi},r\right)  \right)  }\right)
\label{2 derivative of 1}%
\end{equation}
and the theorem is proved.
\end{proof}

\subsection{$L^{p}$ estimates for nonvanishing functions}

The main result in this subsection is the following:

\begin{theorem}
\label{lp estimates lifted operator}Let $\widetilde{B}\left(  \overline{\xi
},R\right)  $ be as before. There exists $r_{0}<R$ and for any $r\leq r_{0}$
there exists $c>0$ such that for any $u\in S_{\widetilde{X}}^{2,p}\left(
\widetilde{B}\left(  \overline{\xi},r\right)  \right)  $ we have%
\[
\left\Vert u\right\Vert _{S_{\widetilde{X}}^{2,p}\left(  \widetilde{B}\left(
\overline{\xi},r/2\right)  \right)  }\leq c\left\{  \left\Vert \widetilde
{\mathcal{L}}u\right\Vert _{L^{p}\left(  \widetilde{B}\left(  \overline{\xi
},r\right)  \right)  }+\left\Vert u\right\Vert _{L^{p}\left(  \widetilde
{B}\left(  \overline{\xi},r\right)  \right)  }\right\}  .
\]
The constants $c,r_{0}$ depend on $\left\{  \widetilde{X}_{i}\right\}
_{i=0}^{q},p,\mu$\thinspace$,R,$and $\eta_{a_{ij},\Omega^{\prime},\Omega
}^{\ast};$ $c$ also depends on $r$.
\end{theorem}

Analogously to what seen in \S \ \ref{subsec Schauder nonvanishing}, the proof
of the above theorem relies on interpolation inequalities for Sobolev norms
and the use of cutoff function.

Regarding cutoff functions, we need the following statement:

\begin{lemma}
[Radial cutoff functions]\label{cutoff function}For any $\sigma\in(\frac{1}%
{2},1),$ $r>0$ and $\xi\in\widetilde{B}\left(  \overline{\xi},r\right)  $,
there exists $\varphi\in C_{0}^{\infty}\left(  \mathbb{R}^{N}\right)  $ with
the following properties:

\begin{enumerate}
\item[i)] $\widetilde{B}_{\sigma r}\left(  \xi\right)  \prec\varphi
\prec\widetilde{B}_{\sigma^{\prime}r}\left(  \xi\right)  $\ with
$\sigma^{\prime}=\left(  1+\sigma\right)  /2$ (this means that $\varphi=1$ in
$\widetilde{B}_{\sigma r}\left(  \xi\right)  $ and it is supported in
$\widetilde{B}_{\sigma^{\prime}r}\left(  \xi\right)  $);

\item[ii)] for $i,j=1,\ldots,q$, we have%
\begin{equation}
\left\vert \widetilde{X}_{i}\varphi\right\vert \leq\frac{c}{(1-\sigma
)r};\left\vert \widetilde{X}_{0}\varphi\right\vert ,\left\vert \widetilde
{X}_{i}\widetilde{X}_{j}\varphi\right\vert \leq\frac{c}{(1-\sigma)^{2}r^{2}}.
\label{6.2}%
\end{equation}

\end{enumerate}
\end{lemma}

The above lemma, very similar to \cite[Lemma 3.3]{bb2}, is actually contained
in Lemma \ref{Lemma schauder cutoff}, but we have preferred to state it
explicitly because it is formulated in a slightly different notation, suitable
to our application to $L^{p}$ estimates.

\begin{theorem}
[Interpolation inequality for Sobolev norms]\label{theorem:interpolation}Let
$\widetilde{B}(\overline{\xi},R)$ be as before. For every $p\in(1,\infty)$
there exists $c>0$ and $r_{1}<R$ such that for every $0<\varepsilon\leq4r_{1}%
$, $u\in C_{0}^{\infty}\left(  \widetilde{B}(\overline{\xi},r_{1})\right)  $,
then
\begin{equation}
\left\Vert \widetilde{X}_{i}u\right\Vert _{L^{p}\left(  \widetilde
{B}(\overline{\xi},r_{1})\right)  }\leq\varepsilon\left\Vert Hu\right\Vert
_{L^{p}\left(  \widetilde{B}(\overline{\xi},r_{1})\right)  }+\ \frac
{c}{\varepsilon}\left\Vert u\right\Vert _{L^{p}\left(  \widetilde{B}%
(\overline{\xi},r_{1})\right)  } \label{interpolation}%
\end{equation}
for every $i=1,\ldots,q$, where $H\equiv\sum_{i=1}^{q}\widetilde{X}_{i}%
^{2}+\widetilde{X}_{0}$.
\end{theorem}

\begin{proof}
The proof of this proposition is adapted from \cite[Thm. 3.6]{bb2}.

Let $r_{1}$ be a small number to be fixed later. Like in the proof of Theorem
\ref{Interpolation inequality for test functions} we can write, for any $u\in
C_{0}^{\infty}\left(  \widetilde{B}(\overline{\xi},r_{1})\right)  $ and
$\xi\in\widetilde{B}(\overline{\xi},r_{1}),$
\[
\widetilde{X}_{i}u\left(  \xi\right)  =SHu(\xi)+Tu\left(  \xi\right)  ,
\]
where $S,T$ are constant operators of type $1$ and $0$, respectively, over
$\widetilde{B}(\overline{\xi},2r_{1})$, provided $2r_{1}<R$. (See the proof of
Proposition \ref{Interpolation inequality for test functions} for the
explanation of the term \textquotedblleft constant operators of type $\lambda
$\textquotedblright). Since%
\[
\left\Vert Tu\right\Vert _{L^{p}\left(  \widetilde{B}(\overline{\xi}%
,r_{1})\right)  }\leq c\left\Vert u\right\Vert _{L^{p}\left(  \widetilde
{B}(\overline{\xi},r_{1})\right)  },
\]
the result will follow if we prove that
\begin{equation}
\left\Vert SHu\right\Vert _{L^{p}\left(  \widetilde{B}(\overline{\xi}%
,r_{1})\right)  }\leq\varepsilon\left\Vert Hu\right\Vert _{L^{p}\left(
\widetilde{B}(\overline{\xi},r_{1})\right)  }+\frac{c}{\varepsilon}\left\Vert
u\right\Vert _{L^{p}\left(  \widetilde{B}(\overline{\xi},r_{1})\right)  }.
\label{T_1 Delta}%
\end{equation}
Let $k(\xi,\eta)$ be the kernel of $S$ and, for any fixed $\xi\in$
$\widetilde{B}(\overline{\xi},r_{1}),$ $\varphi_{\varepsilon}$ a cutoff
function (as in lemma \ref{cutoff function}) with $\widetilde{B}%
_{\frac{\varepsilon}{2}}\left(  \xi\right)  \prec\varphi_{\varepsilon}%
\prec\widetilde{B}_{\varepsilon}\left(  \xi\right)  $. Let us split:%
\begin{align*}
SHu(\xi)  &  =\int_{\widetilde{B}(\overline{\xi},r_{1}),\rho(\xi,\eta
)>\frac{_{\varepsilon}}{2}}k(\xi,\eta)[1-\varphi_{\varepsilon}(\eta
)]Hu(\eta)d\eta+\\
+  &  \int_{\widetilde{B}(\overline{\xi},r_{1}),\rho(\xi,\eta)\leq\varepsilon
}k(\xi,\eta)Hu(\eta)\varphi_{\varepsilon}(\eta)d\eta=I\left(  \xi\right)
+II\left(  \xi\right)  .
\end{align*}
Then
\begin{align*}
\left\vert I\left(  \xi\right)  \right\vert  &  =\left\vert \int
_{\widetilde{B}(\overline{\xi},r_{1}),\rho(\xi,\eta)>\frac{_{\varepsilon}}{2}%
}H^{T}\left(  k(\xi,\cdot)[1-\varphi_{\varepsilon}(\cdot)]\right)  \left(
\eta\right)  u(\eta)d\eta\right\vert \\
&  \leq\int_{\widetilde{B}(\overline{\xi},r_{1}),\rho(\xi,\eta)>\frac
{_{\varepsilon}}{2}}\left\{  \left\vert [1-\varphi_{\varepsilon}]H^{T}k\left(
\xi,\cdot\right)  \right\vert +\right. \\
&  +c\sum\left\vert \widetilde{X}_{i}[1-\varphi_{\varepsilon}]\cdot
\widetilde{X}_{j}k\left(  \xi,\cdot\right)  \right\vert +\left\vert k\left(
\xi,\cdot\right)  H^{T}[1-\varphi_{\varepsilon}]\right\vert \left(
\eta\right)  \left\vert u(\eta)\right\vert d\eta\\
&  \equiv A\left(  \xi\right)  +B\left(  \xi\right)  +C\left(  \xi\right)  .
\end{align*}
Recall that, for $i,j=1,2,...,q,$%
\begin{align*}
\left\vert k(\xi,\eta)\right\vert  &  \leq\frac{c}{d(\xi,\eta)^{Q-1}};\\
\left\vert \widetilde{X}_{i}k(\xi,\eta)\right\vert  &  \leq\frac{c}{d(\xi
,\eta)^{Q}};\\
\left\vert H^{T}k(\xi,\cdot)\left(  \eta\right)  \right\vert  &  \leq\frac
{c}{d(\xi,\eta)^{Q+1}};\\
\left\vert \widetilde{X}_{i}\left(  1-\varphi_{\varepsilon}\right)  \left(
\eta\right)  \right\vert  &  \leq\frac{c}{\varepsilon},\left\vert H^{T}\left(
1-\varphi_{\varepsilon}\right)  \left(  \eta\right)  \right\vert \leq\frac
{c}{\varepsilon^{2}}%
\end{align*}
and the derivatives of $\left(  1-\varphi_{\varepsilon}\right)  $ are
supported in the annulus $\frac{\varepsilon}{2}\leq d(\xi,\eta)\leq
\varepsilon$. \ Since $\xi,\eta\in\widetilde{B}(\overline{\xi},r_{1}),$ we
have $d\left(  \xi,\eta\right)  <2r_{1}$. Hence letting $k_{0}$ be the integer
such that $2^{k_{0}-1}\varepsilon<2r_{1}\leq2^{k_{0}}\varepsilon$ we have%
\begin{align}
\left\vert A\left(  \xi\right)  \right\vert  &  \leq c\sum_{k=0}^{k_{0}}%
\int_{2^{k-1}\varepsilon<\rho(\xi,\eta)\leq2^{k}\varepsilon}\frac{c}%
{d(\xi,\eta)^{Q+1}}\left\vert u\left(  \eta\right)  \right\vert d\eta
\nonumber\\
&  \leq c\sum_{k=0}^{k_{0}}\frac{1}{2^{k-1}\varepsilon}\frac{1}{\left(
\varepsilon2^{k-1}\right)  ^{Q}}\int_{\rho(\xi,\eta)\leq2^{k}\varepsilon
}\left\vert u\left(  \eta\right)  \right\vert d\eta\nonumber\\
&  \leq\frac{c}{\varepsilon}\cdot\sup_{r\leq4r_{1}}\frac{1}{\left\vert
\widetilde{B}\left(  \xi,r\right)  \right\vert }\int_{\widetilde{B}\left(
\xi,r\right)  }\left\vert u\left(  \eta\right)  \right\vert d\eta.
\label{A(csi)}%
\end{align}

We now have to recall the definition of the local maximal function
$\mathcal{M}$ (see \S \ \ref{subsec locally hom space}). With the notation of
Theorem \ref{Thm maximal}, we have $4r_{1}=r_{n}=\frac{2}{5}\varepsilon_{n},$
hence $\varepsilon_{n}=10r_{1}$ and, for $\xi\in\widetilde{B}\left(
\overline{\xi},r_{1}\right)  $, we have $\widetilde{B}\left(  \xi
,\varepsilon_{n}\right)  \subset\widetilde{B}\left(  \overline{\xi}%
,11r_{1}\right)  $. Therefore by (\ref{A(csi)}) we can write%
\[
\left\vert A\left(  \xi\right)  \right\vert \leq\frac{c}{\varepsilon}%
\cdot\mathcal{M}_{\widetilde{B}\left(  \overline{\xi},r_{1}\right)
,\widetilde{B}\left(  \overline{\xi},11r_{1}\right)  }u\left(  \xi\right)
\]
and by Theorem \ref{Thm maximal}, we have
\[
\left\Vert A\right\Vert _{L^{p}\left(  \widetilde{B}\left(  \overline{\xi
},r_{1}\right)  \right)  }\leq\frac{c}{\varepsilon}\left\Vert u\right\Vert
_{L^{p}\left(  \widetilde{B}\left(  \overline{\xi},11r_{1}\right)  \right)
}=\frac{c}{\varepsilon}\left\Vert u\right\Vert _{L^{p}\left(  \widetilde
{B}\left(  \overline{\xi},r_{1}\right)  \right)  },
\]
since $u\in C_{0}^{\infty}\left(  \widetilde{B}(\overline{\xi},r_{1})\right)
,$ provided $11r_{1}<R.$ Also%
\begin{align*}
\left\vert B\left(  \xi\right)  \right\vert  &  \leq c\int_{\frac{\varepsilon
}{2}<\rho(\xi,\eta)\leq\varepsilon,}\frac{1}{\varepsilon}\cdot\frac{1}%
{d(\xi,\eta)^{Q}}\left\vert u\left(  \eta\right)  \right\vert d\eta\\
&  \leq\frac{c}{\varepsilon^{Q+1}}\int_{\rho(\xi,\eta)\leq\varepsilon
}\left\vert u\left(  \eta\right)  \right\vert d\eta\\
&  \leq\frac{c}{\varepsilon}\cdot\sup_{r\leq\varepsilon}\frac{1}{\left\vert
\widetilde{B}\left(  \xi,r\right)  \right\vert }\int_{\widetilde{B}\left(
\xi,r\right)  }\left\vert u\left(  \eta\right)  \right\vert d\eta\\
&  \leq\frac{c}{\varepsilon}\cdot\mathcal{M}_{\widetilde{B}\left(
\overline{\xi},r_{1}\right)  ,\widetilde{B}\left(  \overline{\xi}%
,11r_{1}\right)  }u\left(  \xi\right)
\end{align*}
\text{provided }$\varepsilon<4r_{1}.$ As before we have%
\[
\left\Vert B\right\Vert _{L^{p}\left(  \widetilde{B}\left(  \overline{\xi
},r_{1}\right)  \right)  }\leq\frac{c}{\varepsilon}\left\Vert u\right\Vert
_{L^{p}\left(  \widetilde{B}\left(  \overline{\xi},r_{1}\right)  \right)  }.
\]
Finally,%
\begin{align*}
\left\vert C\left(  \xi\right)  \right\vert  &  \leq c\int_{\frac{\varepsilon
}{2}<\rho(\xi,\eta)\leq\varepsilon}\frac{1}{\varepsilon^{2}}\cdot\frac
{1}{d(\xi,\eta)^{Q-1}}\left\vert u\left(  \eta\right)  \right\vert \eta dy\\
&  \leq\frac{c}{\varepsilon^{Q+1}}\int_{\rho(\xi,\eta)\leq\varepsilon
}\left\vert u\left(  \eta\right)  \right\vert d\eta
\end{align*}
as for the term $B\left(  \xi\right)  $, therefore%
\[
\left\Vert I\right\Vert _{L^{p}\left(  \widetilde{B}\left(  \overline{\xi
},r_{1}\right)  \right)  }\leq\frac{c}{\varepsilon}\left\Vert u\right\Vert
_{L^{p}\left(  \widetilde{B}\left(  \overline{\xi},r_{1}\right)  \right)  }.
\]

Let us bound $II$:%
\[
\left\vert II\left(  \xi\right)  \right\vert \leq c\int_{\rho(\xi,\eta
)\leq\varepsilon}\frac{\left\vert Hu\left(  \eta\right)  \right\vert }%
{\rho(\xi,\eta)^{Q-1}}d\eta.
\]
Then, a computation similar to that of $C\left(  \xi\right)  $ gives%
\[
\left\vert II\left(  \xi\right)  \right\vert \leq c\varepsilon\mathcal{M}%
_{\widetilde{B}\left(  \overline{\xi},r_{1}\right)  ,\widetilde{B}\left(
\overline{\xi},11r_{1}\right)  }u\left(  \xi\right)
\]
and%
\[
\left\Vert II\right\Vert _{L^{p}\left(  \widetilde{B}\left(  \overline{\xi
},r_{1}\right)  \right)  }\leq c\varepsilon\left\Vert u\right\Vert
_{L^{p}\left(  \widetilde{B}\left(  \overline{\xi},r_{1}\right)  \right)  },
\]
provided $\varepsilon<4r_{1}.$ So we are done.
\end{proof}

\begin{theorem}
\label{iiith:interp:semino}For any $u\in S_{\widetilde{X}}^{2,p}(\widetilde
{B}\left(  \overline{\xi},r\right)  ),$ $p\in\lbrack1,\infty),$ $0<r<r_{1}$
(where $r_{1}$ is the number in Theorem \ref{theorem:interpolation}), define
the following quantities:%
\[
\Phi_{k}\text{ }=\,\sup_{1/2<\sigma<1}\left(  (1-\sigma)^{k}r^{k}\left\Vert
\widetilde{D}^{k}u\right\Vert _{L^{p}\left(  \widetilde{B}_{r\sigma}\right)
}\,\right)  \text{ \ \ for }k=0,1,2.
\]
Then for any $\delta>0$ (small enough)
\[
\Phi_{1}\leq\delta\,\Phi_{2}+\frac{c}{\delta}\Phi_{0}.
\]

\end{theorem}

\begin{proof}
This result follows exactly as in \cite[Thm. 21]{bb1} exploiting the
interpolation result for compactly supported functions (Theorem
\ref{theorem:interpolation}), cutoff functions (Lemma \ref{cutoff function})
and Proposition \ref{Prop Sobolev vanishing}.
\end{proof}

\bigskip

We can now come to the

\begin{proof}
[Proof of Theorem \ref{lp estimates lifted operator}]This proof is similar to
that of theorem \cite[Thm. 3]{bb1}. Due to the different context, we include a
complete proof for convenience of the reader.

Pick $r_{0}=\min\left(  R_{0},r_{1}\right)  $ where $R_{0},r_{1}$ are the
numbers appearing in Theorems \ref{theorem:sprt function regular},
\ref{theorem:interpolation}, respectively. For $r\leq r_{0},$ let $u\in
S_{\widetilde{X}}^{2,p}\left(  \widetilde{B}\left(  \overline{\xi},r\right)
\right)  $. Let $\varphi$ be a cutoff function as in lemma
\ref{cutoff function},
\[
\widetilde{B}\left(  \overline{\xi},\sigma r\right)  \prec\varphi
\prec\widetilde{B}\left(  \overline{\xi},\sigma^{\prime}r\right)  .
\]
By Theorem \ref{theorem:sprt function regular}, $\varphi u\in S_{\widetilde
{X},0}^{2,p}\left(  \widetilde{B}\left(  \overline{\xi},r\right)  \right)  $;
then, by density, we can apply Theorem \ref{theorem:sprt function regular} to
$\varphi u$:%
\[
\left\Vert \varphi u\right\Vert _{S^{2,p}\left(  \widetilde{B}\left(
\overline{\xi},r\right)  \right)  }\leq c\left\{  \left\Vert \widetilde
{\mathcal{L}}\left(  \varphi u\right)  \right\Vert _{L^{p}\left(
\widetilde{B}\left(  \overline{\xi},r\right)  \right)  }+\left\Vert \varphi
u\right\Vert _{L^{p}\left(  \widetilde{B}\left(  \overline{\xi},r\right)
\right)  }\right\}  .
\]
For $1\leq i,j\leq q$, from the above inequality we get%
\begin{align*}
\left\Vert \widetilde{X}_{i}\widetilde{X}_{j}u\right\Vert _{L^{p}%
(\widetilde{B}_{\sigma r})}  &  \leq c\left\{  \left\Vert \widetilde
{\mathcal{L}}u\right\Vert _{L^{p}(\widetilde{B}_{\sigma^{\prime}r}%
)}+\left\Vert u\right\Vert _{L^{p}(\widetilde{B}_{\sigma^{\prime}r})}+\right.
\\
&  \left.  +\frac{1}{(1-\sigma)r}\left\Vert \widetilde{D}u\right\Vert
_{L^{p}(\widetilde{B}_{\sigma^{\prime}r})}+\frac{1}{(1-\sigma)^{2}r^{2}%
}\left\Vert u\right\Vert _{L^{p}(\widetilde{B}_{\sigma^{\prime}r})}\right\} \\
&  \leq c\left\{  \left\Vert \widetilde{\mathcal{L}}u\right\Vert
_{L^{p}(\widetilde{B}_{\sigma^{\prime}r})}+\frac{1}{(1-\sigma)r}\left\Vert
\widetilde{D}u\right\Vert _{L^{p}(\widetilde{B}_{\sigma^{\prime}r})}\right. \\
&  \left.  +\frac{1}{(1-\sigma)^{2}r^{2}}\left\Vert u\right\Vert
_{L^{p}(\widetilde{B}_{\sigma^{\prime}r})}\right\}
\end{align*}
where, as before, we let%
\[
\left\Vert \widetilde{D}u\right\Vert \equiv\sum_{i=1}^{q}\left\Vert
\widetilde{X}_{i}u\right\Vert \text{ and }\left\Vert \widetilde{D}%
^{2}u\right\Vert \equiv\sum_{i,j=1}^{q}\left\Vert \widetilde{X}_{i}%
\widetilde{X}_{i}u\right\Vert +\left\Vert \widetilde{X}_{0}u\right\Vert .
\]
Multiplying both sides for $(1-\sigma)^{2}r^{2}$ we get%
\begin{align}
(1-\sigma)^{2}r^{2}  &  \parallel\widetilde{X}_{i}\widetilde{X}_{j}%
u\parallel_{L^{p}(\widetilde{B}_{\sigma r})}\leq c\left\{  (1-\sigma)^{2}%
r^{2}\left\Vert \widetilde{\mathcal{L}}u\right\Vert _{L^{p}(\widetilde
{B}_{\sigma^{\prime}r})}+\right. \label{xxf}\\
&  \left.  +(1-\sigma)r\left(  \left\Vert \widetilde{D}u\right\Vert
_{L^{p}(\widetilde{B}_{\sigma^{\prime}r})}\right)  +\left\Vert u\right\Vert
_{L^{p}(\widetilde{B}_{\sigma^{\prime}r})}\right\}  .\nonumber
\end{align}
Next, we compute $(1-\sigma)^{2}r^{2}\left\Vert \widetilde{X}_{0}u\right\Vert
_{L^{p}(\widetilde{B}_{\sigma r})}:$%
\begin{align}
(1-\sigma)^{2}r^{2}\left\Vert \widetilde{X}_{0}u\right\Vert _{L^{p}%
(\widetilde{B}_{\sigma r})}  &  =(1-\sigma)^{2}r^{2}\left\Vert \widetilde
{\mathcal{L}}u-\sum_{i,j=1}^{q}\widetilde{a}_{ij}\widetilde{X}_{i}%
\widetilde{X}_{j}u\right\Vert _{L^{p}(\widetilde{B}_{\sigma r})}\label{xof}\\
&  \leq c(1-\sigma)^{2}r^{2}\left(  \left\Vert \widetilde{\mathcal{L}%
}u\right\Vert _{L^{p}(\widetilde{B}_{\sigma r})}+\left\Vert \widetilde{X}%
_{i}\widetilde{X}_{j}u\right\Vert _{L^{p}(\widetilde{B}_{\sigma r})}\right)
.\nonumber
\end{align}
Combining (\ref{xxf}) and (\ref{xof}), we have
\begin{align}
&  (1-\sigma)^{2}r^{2}\left\Vert \widetilde{D}^{2}u\right\Vert _{L^{p}%
(\widetilde{B}_{\sigma r})}\leq c\left\{  (1-\sigma)^{2}r^{2}\left\Vert
\widetilde{\mathcal{L}}u\right\Vert _{L^{p}(\widetilde{B}_{\sigma^{\prime}r}%
)}+\right. \label{D2f}\\
&  \left.  +(1-\sigma)r\left\Vert \widetilde{D}u\right\Vert _{L^{p}%
(\widetilde{B}_{\sigma^{\prime}r})}+\left\Vert u\right\Vert _{L^{p}%
(\widetilde{B}_{\sigma^{\prime}r})}\right\}  .\nonumber
\end{align}
Adding $(1-\sigma)r\left\Vert Du\right\Vert _{L^{p}(\widetilde{B}_{\sigma r}%
)}$ to both sides of (\ref{D2f}),%
\begin{align}
&  (1-\sigma)r\left\Vert \widetilde{D}u\right\Vert _{L^{p}(\widetilde
{B}_{\sigma r})}+(1-\sigma)^{2}r^{2}\left\Vert \widetilde{D}^{2}u\right\Vert
_{L^{p}(\widetilde{B}_{\sigma r})}\label{Df,D2f}\\
&  \leq c\left\{  (1-\sigma)^{2}r^{2}\left\Vert \widetilde{\mathcal{L}%
}u\right\Vert _{L^{p}(\widetilde{B}_{\sigma^{\prime}r})}+(1-\sigma)r\left\Vert
\widetilde{D}u\right\Vert _{L^{p}(\widetilde{B}_{\sigma^{\prime}r}%
)}+\left\Vert u\right\Vert _{L^{p}(\widetilde{B}_{\sigma^{\prime}r})}\right\}
,\nonumber
\end{align}
by Theorem \ref{iiith:interp:semino},%
\[
\leq c\left\{  (1-\sigma)^{2}r^{2}\left\Vert \widetilde{\mathcal{L}%
}u\right\Vert _{L^{p}(\widetilde{B}_{\sigma^{\prime}r})}+\left(  \delta
\Phi_{2}+\frac{c}{\delta}\Phi_{0}\right)  +\left\Vert u\right\Vert
_{L^{p}(\widetilde{B}_{\sigma^{\prime}r})}\right\}  .
\]
Choosing $\delta$ small enough, we have%
\[
\Phi_{2}+\Phi_{1}\leq c\left\{  r^{2}\left\Vert \widetilde{\mathcal{L}%
}u\right\Vert _{L^{p}(\widetilde{B}_{r})}+\left\Vert u\right\Vert
_{L^{p}(\widetilde{B}_{r})}\right\}  ,
\]
then%
\[
r^{2}\left\Vert \widetilde{D}^{2}u\right\Vert _{L^{p}\left(  \widetilde
{B}\left(  \overline{\xi},r/2\right)  \right)  }+r\left\Vert \widetilde
{D}u\right\Vert _{L^{p}\left(  \widetilde{B}\left(  \overline{\xi},r/2\right)
\right)  }\leq c\left\{  r^{2}\left\Vert \widetilde{\mathcal{L}}u\right\Vert
_{L^{p}\left(  \widetilde{B}\left(  \overline{\xi},r\right)  \right)
}+\left\Vert u\right\Vert _{L^{p}\left(  \widetilde{B}\left(  \overline{\xi
},r\right)  \right)  }\right\}  ,
\]
hence%
\[
\left\Vert u\right\Vert _{S_{\widetilde{X}}^{2,p}\left(  \widetilde{B}\left(
\overline{\xi},r/2\right)  \right)  }\leq c\left\{  \left\Vert \widetilde
{\mathcal{L}}u\right\Vert _{L^{p}\left(  \widetilde{B}\left(  \overline{\xi
},r\right)  \right)  }+\left\Vert u\right\Vert _{L^{p}\left(  \widetilde
{B}\left(  \overline{\xi},r\right)  \right)  }\right\}  ,
\]
which is the desired result.
\end{proof}

\subsection{$L^{p}$ estimates in the original variables}

Let's now prove Theorem \ref{lp estimate}, which follows from Theorem
\ref{lp estimates lifted operator} in a way which is analogous to that
followed in \S \ \ref{subsec schauder original} to prove Schauder estimates.

Fix $\overline{x}\in\Omega^{\prime}\Subset\Omega$ and $R$ such that in
$B\left(  \overline{x},R\right)  \subset\Omega$ all the construction of the
previous two subsections (lifting to $\widetilde{B}\left(  \overline{\xi
},R\right)  $ and so on) can be performed. Let $r_{0}<R$ as in Theorem
\ref{lp estimates lifted operator}, and let $u\in S_{X}^{2,p}\left(  B\left(
\overline{x},r_{0}\right)  \right)  $. By Theorem \ref{Theorem lifted sobolev}
we know that the function $\widetilde{u}\left(  x,h\right)  =u\left(
x\right)  $ belongs to $S_{\widetilde{X}}^{2,p}\left(  B\left(  \overline{\xi
},r_{0}\right)  \right)  $, so we can apply to $\widetilde{u}$ the $L^{p}$
estimates contained in Theorem \ref{lp estimates lifted operator}. Combining
this fact with the two estimates in Theorem \ref{Theorem lifted sobolev} we
get%
\begin{align*}
\left\Vert u\right\Vert _{S_{X}^{2,\alpha}\left(  B\left(  \overline{x}%
,\delta_{0}r_{0}/2\right)  \right)  }  &  \leq c\left\Vert \widetilde
{u}\right\Vert _{S_{\widetilde{X}}^{2,\alpha}\left(  \widetilde{B}\left(
\overline{\xi},r_{0}/2\right)  \right)  }\\
&  \leq c\left(  \left\Vert \widetilde{\mathcal{L}}\widetilde{u}\right\Vert
_{L^{p}\left(  \widetilde{B}\left(  \overline{\xi},r_{0}\right)  \right)
}+\left\Vert \widetilde{u}\right\Vert _{L^{p}\left(  \widetilde{B}\left(
\overline{\xi},r_{0}\right)  \right)  }\right) \\
&  \leq c\left(  \left\Vert \mathcal{L}u\right\Vert _{L^{p}\left(  B\left(
\overline{x},r_{0}\right)  \right)  }+\left\Vert u\right\Vert _{L^{p}\left(
B\left(  \overline{x},r_{0}\right)  \right)  }\right)
\end{align*}
since $\widetilde{\mathcal{L}}\widetilde{u}=\widetilde{\left(  \mathcal{L}%
u\right)  }.$

Next, let us choose a family of balls $B\left(  x_{i},r_{0}\right)  $ in
$\Omega$ such that%
\[
\Omega^{\prime}\subset%
{\displaystyle\bigcup\limits_{i=1}^{k}}
B\left(  x_{i},\delta_{0}r_{0}/2\right)  \subset%
{\displaystyle\bigcup\limits_{i=1}^{k}}
B\left(  x_{i},r_{0}\right)  \subset\Omega.
\]
Therefore
\begin{align*}
\left\Vert u\right\Vert _{S_{X}^{2,p}(\Omega^{\prime})}  &  \leq\left\Vert
u\right\Vert _{S_{X}^{2,p}(\cup B\left(  x_{i},\delta_{0}r_{0}/2\right)
)}\leq\sum_{i=1}^{k}\left\Vert u\right\Vert _{S_{X}^{2,p}(B\left(
x_{i},\delta_{0}r_{0}/2\right)  )}\\
&  \leq c\sum_{i=1}^{k}\left\{  \left\Vert \mathcal{L}u\right\Vert
_{L^{p}\left(  B\left(  x_{i},r_{0}\right)  \right)  }+\left\Vert u\right\Vert
_{L^{p}\left(  B\left(  x_{i},r_{0}\right)  \right)  }\right\} \\
&  \leq c\left\{  \left\Vert \mathcal{L}u\right\Vert _{L^{p}(\Omega
)}+\left\Vert u\right\Vert _{L^{p}(\Omega)}\right\}
\end{align*}
with $c$ also depending on $r_{0}$.

\bigskip

Marco Bramanti

Dipartimento di Matematica, Politecnico di Milano

Via Bonardi 9. 20133 Milano. ITALY

marco.bramanti@polimi.it

\bigskip

Maochun Zhu

Department of Applied Mathematics, Northwestern Polytechincal University

127 West Youyi Road. 710072 Xi'an. P. R. CHINA

zhumaochun2006@126.com

\end{document}